\documentclass[12pt,notitlepage]{article}
\usepackage{amsthm}
\usepackage{amsmath}
\usepackage{mycommands}
\usepackage{thesiscommands}
\usepackage{url}
\usepackage{amssymb}
\usepackage{amsthm}
\begin{document}
%% Comment out items by inserting a percent % character
\title{Floer cohomology of real Lagrangians in the Fermat quintic threefold}
\author{Garrett Alston}
\maketitle

\abstract{Let $X=\sett{X_0^5+\cdots+X_4^5=0}$ be the Fermat quintic threefold.
The set of real solutions $L$ forms a Lagrangian submanifold of $X$.
Multiplying the homogeneous coordinates of $X$ by various fifth roots of unity gives automorphisms of $X$; the images of $L$ under these automorphisms defines a family of 625 different Lagrangian submanifolds, called real Lagrangians.
In this paper we try to calculate the Floer cohomology between all pairs of these Lagrangians.
We are able to complete most of the calculations, but there are a few cases we cannot do.

The basic idea is to explicitly describe some low energy moduli spaces and then use this knowledge to calculate the differential on the $E_2$ page of the standard spectral sequence for Floer cohomology.
It turns out that this is often enough to calculate the cohomology completely.
Several techniques are developed to help describe these low energy moduli spaces, including a formula for the Maslov index, a formula for the obstruction bundle, and a way to relate holomorphic strips and discs to holomorphic spheres.
The real nature of the Lagrangians is crucial for the development of these techniques.}

\tableofcontents
\part{Introduction, definitions and constructions}
%\chapter{Introduction, Definitions, and Constructions}
\section{Introduction}
Let $X=\sett{X_0^5+\cdots+X_4^5=0}$ be the Fermat quintic threefold.
The set of real solutions $L$ forms a Lagrangian submanifold of $X$.
Multiplying the homogeneous coordinates of $X$ by various fifth roots of unity gives automorphisms of $X$; the images of $L$ under these automorphisms defines a family of 625 different Lagrangian submanifolds, called real Lagrangians.
In this paper we try to calculate the Floer cohomology between all pairs of these Lagrangians.
We are able to complete most of the calculations, but there are a few cases we cannot do.

The basic idea is to explicitly describe some low energy moduli spaces and then use this knowledge to calculate the differential on the $E_2$ page of the standard spectral sequence for Floer cohomology.
It turns out that this is often enough to calculate the cohomology completely.
Several techniques are developed to help describe these low energy moduli spaces, including a formula for the Maslov index, a formula for the obstruction bundle, and a way to relate holomorphic strips and discs to holomorphic spheres.
The real nature of the Lagrangians is crucial for the development of these techniques.

The problem, results, and approach will be discussed in more detail in the next few sections.
Right now, we would like to turn to putting Floer cohomology and this set of examples into a broader context and point out some interesting features.

First, note that this class of examples may be of interest to the story of homological mirror symmetry.
The homological mirror symmetry conjecture asserts that for any Calabi-Yau manifold $X$, there exists a mirror Calabi-Yau manifold $\check X$ such that the derived Fukaya category on $X$ is equivalent to the derived category of sheaves on $\check X$.
The Calabi-Yau manifolds of most interest are the threefolds, and the easiest ones to understand are probably the quintic hypersurfaces in $\cp^4$.
The Fukaya category is an $A_\infty$-category whose objects are Lagrangian submanifolds.
This means that set of morphisms between two objects $L_0$ and $L_1$ is a graded vector space $CF(L_0,L_1)$, and composition of morphisms is not associative but rather is given by multilinear maps $m_k$ that satsfy the $A_\infty$-axioms.
%Composition of morphisms is not associative, but rather is given by maps
%$$m_k:CF(L_{k-1},L_k)\otimes\cdots\otimes CF(L_0,L_1)\to CF(L_0,L_k)$$
%that satisfy the $A_\infty$ axioms.
%These maps can be deformed by bounding cochains, which are elements of $CF^1(L_i,L_i)$ that solve the Maurer-Cartan equation.
%A Lagrangian is called unobstructed if there exists a bounding cochain; one implication of unobstructedness is that there exists a deformation $m_1^b$ of $m_1$ such that $(m_1^b)^2=0$.
%That is, $(CF(L_i,L_i),m_1^b)$ is a cochain complex, and hence its cohomology is defined.
Under certain conditions (unobstructedness), the graded vector spaces $CF(L_0,L_1)$ become chain complexes with respect to the operator $m_1$, and hence one can take their cohomology.
This cohomology is the Floer cohomology of the pair $(L_0,L_1)$.
(The notion of Floer cohomology can be generalized by considering bounding cochains; we do not go into this here.)
Thus the real Lagrangians can be thought of as objects in the derived Fukaya category of the Fermat quintic, and their Floer cohomology gives the Hom spaces between them.

Next, note that real Lagrangians are the fixed point sets of anti-symplectic involutions, and such Lagrangians have been studied in the literature recently.
In particular, Solomon has shown that open Gromov-Witten invariants can be defined for them \cite{sol}.
Also, it is conjectured that some fixed point sets have special significance to mirror symmetry--for instance, they can be used to construct auto-equivalences of the Fukaya category that are thought to be mirror to twisting by a line bundle on the derived category side (see \cite{cbms}, although it should be noted that the real Lagrangians studied here are not the same as the Lagrangian sections constructed there).
The moduli space of holomorphic discs with boundary on a fixed-point sets has also been studied in \cite{foooasi}.

Finally, we would like to point out that finding interesting Lagrangians is in general a non-trivial task.
In the case of the quintic, only two other classes of examples come readily to mind.
One class of examples is given by the fixed point sets and Lagrangian torus fibrations constructed in \cite{cbms}.
To get a second class of examples, one could apply Seidel's Lefschetz fibration and Picard theory to the three-dimensional case; in \cite{seidel2} Seidel says that in principle the Lagrangian spheres obtained in this way should have the same significance that they do in the quartic surface case.
We do not know the relationship between the real Lagrangians studied here and these two examples.
%Secondly, note that calculating Floer cohomology is in general difficult, and not many explicit examples are known.
%One notable class of examples is that of Lagrangian torus fibers in toric manifolds, see \cite{toricfooo}.
\\
\\
\noindent{{\bf Acknowledgement:} This paper is a revised version of my dissertation at the University of Wisconsin-Madison and could not have been written without the help of my advisor Yong-Geun Oh.
I would like to thank him, as well as my many other friends and colleagues in Madison who made my time there so enjoyable.
}

\section{Overview}
\subsection{The problem and summary of results}\label{section:summary}
The Fermat quintic is the hypersurface in $\cp^4$ defined by
$$X=\sett{X_0^5+X_1^5+X_2^5+X_3^5+X_4^5=0}.$$
The set of real points 
$$L=\set{[X_0:\cdots:X_4]\in X}{X_i\in\rr}$$
is a Lagrangian submanifold.
$L$ is diffeomorphic to $\rp^3$ via the diffeomorphism $\rp^3\to L$ given by
$$[X_0:X_1:X_2:X_3]\mapsto [X_0:X_1:X_2:X_3:-(X_0^5+X_1^5+X_2^5+X_3^5)^\frac{1}{5}].$$
$L$ is the fixed point set of the complex conjugation map
$$\tau:X\to X,$$$$[X_0:\cdots:X_4]\mapsto [\bar X_0:\cdots:\bar X_4].$$
(The map $\tau$ will often be denoted by a bar, for example $p\mapsto \bar p$.)
%\begin{eqnarray*}
%\tau &:& X\to X\\
%&& [X_0:\cdots:X_4]\mapsto [\bar X_0:\cdots:\bar X_4].
%\end{eqnarray*}

Let $\gamma=e^{2\pi i/5}$ and let $\z_5$ be the cyclic group of order $5$ generated by $\gamma$.
Let $G$ be the group 
$$G=\z_5^5/\textrm{diagonal}\subset PU(5).$$
The elements
$$g_1=[(1,\gamma,1,1,1)],$$
$$g_2=[(1,1,\gamma,1,1)],$$
$$g_3=[(1,1,1,\gamma,1)],$$
$$g_4=[(1,1,1,1,\gamma)]$$
give an explicit isomorphism $G\cong \z_5^4$.
$G$ acts on $\cp^4$, and also on $X$ because it preserves the defining equation.
$G$ acts symplectically on $\cp^4$, so the $G$ action is also symplectic on $X$.
$G$ gives rise to a family of 625 different Lagrangian submanifolds
$$\mathcal L=\set{gL}{g\in G}.$$
Elements of $\mathcal L$ will be referred to as real Lagrangians, because they are the fixed point sets of the anti-holomorphic involutions $g\circ\tau\circ g^{-1}$.

In this paper we try to calculate the Floer cohomology $HF(L_1,L_2)$ for every $L_1,L_2\in\mathcal L$.
Clearly it suffices to calculate $HF(L,gL)$ for every $g\in G$.
Due to the standard spectral sequence for Floer cohomology, it is useful to think of the Floer cohomology as living between $0$ and the singular cohomolgy of $L\cap gL$, that is
\begin{displaymath}
0\leq HF(L,gL) \leq H(L\cap gL).
\end{displaymath}
Therefore it is natural break the problem down into cases depending on what $L\cap gL$ is diffeomorphic to.
Based on this classification, here is a summary of the results in this paper:

%\begin{center}
%\begin{tabular}{c|l}

%$L\cap gL$ & $HF(L,gL)$ \\
%\hline
\begin{enumerate}
\item $L\cap gL\cong\rp^3$: $HF(L, gL)\cong H^*(\rp^3;\Lambda_{nov})$.
\item $L\cap gL\cong\rp^2$: $HF(L,gL)\cong H^*(\rp^2;\Det(T\rp^2)\otimes\Theta^-_{\rp^2};\Lambda_{nov})$ (i.e. cohomology in a local system),
\item $L\cap gL\cong \rp^1$: $HF(L,gL)$ is $0$ or unknown,
\item $L\cap gL\cong \rp^1\amalg \rp^0$: $HF(L\cap gL)$ is $H^1(\rp^1;\Lambda_{nov})$, $H^0(\rp^1;\Lambda_{nov})$, or has rank at least $1$.
\item $L\cap gL\cong \rp^0$: $HF(L,gL)\cong H^*(\rp^0;\Lambda_{nov})$.
\item $L\cap gL\cong \rp^0\amalg\rp^0$: $HF(L,gL)$ is $0$, $H^*(\rp^0\amalg\rp^0;\Lambda_{nov})$, or unknown.
\end{enumerate}
These results are contained in Theorems \ref{thm:case1-main_thm}, \ref{thm:case2-main_thm}, \ref{thm:case3-main_thm}, \ref{thm:case3-main_thm2}, \ref{thm:case4-main_thm_1}, \ref{thm:case4-main_thm_2}, \ref{thm:case4-main_thm_3}, \ref{thm:case4-main_thm_4}, \ref{thm:case5-main_thm}, \ref{thm:case6-main_thm_1}, \ref{thm:case6-main_thm_2}, and \ref{thm:case6-main_thm_3}.
Strictly speaking, Floer cohomology depends also on the choice of bounding cochains and spin structures.
The convention in this paper is that $0$ is always used for the bounding cochain, and $L$ is given an arbitrary spin sructure, $gL$ is given the push forward of this spin structure by $g$.
Moreover, $\Lambda_{nov}$ is used as the coefficient ring.
This will all be explained in more detail later.
%Note that $\pi_1(X)=0$, so if $g$ is in the identity component of $\textrm{Symp}(X,\omega)$ then $g$ is a Hamiltonian diffeomorphism.
%Since Floer cohomology is invariant under Hamiltonian diffeomorphism, it follows that
%\begin{corollary}
%If $HF(L,gL)$ does not have rank 2, then $g$ is not in the identity component of $\textrm{Symp}(X,\omega)$.
%\end{corollary}
%\begin{proof}
%Additionally, the fact that $gL$ is given the push forward Spin structure and orientation is needed to conclude this.
%\end{proof}
%Furthermore, since the Floer cohomology is graded (in the sense of \cite{seidel1}), and $g$ can be lifted to a graded symplectomorphism of $X$ (since $\pi_1(X)=0$), the corollary can be sharpened as
%\begin{corollary}
%If $HF(L,gL)$ is not isomorphic as a graded vector space to some shift of $H^*(\rp^3)$, then $g$ is not in the identity component of $\textrm{Symp}(X,\omega)$.
%\end{corollary}
\subsection{The method of attack}
The type of Floer cohomology used throughout this paper is the Bott-Morse Floer cohomology developed in \cite{fooo}.
Namely, to each Lagrangian $L$ is associated an $A_\infty$-algebra, and to each pair of Lagrangians $(L_1,L_2)$ is associated an $A_\infty$-bimodule.
$L_0$ and $L_1$ are not required to be transverse, rather, they are only required to have clean intersection, which means that $T(L_0\cap L_1)=TL_0\cap TL_1$.
The Floer cohomologies are then by definition the cohomologies of these algebras and bimodules.
Complete details of the definitions and constructions sketched in this paper can be found in \cite{fooo}.
One main difference, however, should be noted. 
In \cite{fooo}, the Floer cohomology is given a relative grading.
In this paper, the Floer cohomology is given an absolute grading (in the sense of \cite{seidel1}, see Section \ref{section:gradings}).
Therefore the grading parameter $e$ is dropped from the Novikov rings.

The coboundary operator in Floer cohomology is defined by counting (in a broad sense) holomorphic strips.
To gain information about holomorphic strips in the real Lagrangians, several techniques are developed:
\begin{enumerate}
\item A formula for the Maslov index,
\item an involution of the moduli space of holomorphic strips and a calculation of the sign of the involution,
\item a characterization of holomorphic strips as holomorphic spheres with certain symmetries,
\item and a characterization of cokernels as a subspace of certain sheaf cohomology groups.
\end{enumerate}

These techniques allow some Floer cohomologies to be computed easily.
Technique (1) allows the computation of some cohomologies purely from degree considerations, and technique (2) implies that in some cases the Floer coboundary operator is $0$.

In the more difficult cases, technique (3) has to be applied to explicitly describe the lowest energy moduli spaces.
The lowest energy moduli spaces allow the differential on the $E_2$ page of the spectral sequence to be calculated, and often this is enough to calculate the Floer cohomology.
Technique (4) is needed to determine if the moduli spaces are regular.

In the cases where these techniques are not enough to complete the calculation, the difficulty is that the higher energy moduli spaces need to be known.
These moduli spaces are probably impossible to explicitly describe using the techniques above (although techniques (3) and (4) relate these moduli spaces in some abstract sense to the moduli space of rational curves).

\subsection{Organization}
This paper is broken down into three parts.
Part I contains general definitions and constructions needed to define Floer cohomology.
Most of the sections are expository; however Sections 3 and 4 contain some useful formulas for the Maslov index that are usually not explicitly mentioned in the literature, as well as a generalization of Seidel's absolute index in Floer cohomology to the Bott-Morse case.
In Part II the discussion is specialized to real Lagrangians, and the rest of the techniques mentioned in the previous section are developed.
In Part III, the real Lagrangians are considered on a case by case basis, and the techniques are applied to calculate as many of the Floer cohomologies as possible.
An appendix on Kuranishi spaces is included, mainly to fix notation.

Sections \ref{section:local_systems} and \ref{section:conjugation-orientation} deal with orientation issues and are rather technical.
Readers not familiar with these issues may want to skip these sections, and take Propositions \ref{prop:conjugation-main_prop} and \ref{prop:conjugation-main_prop_2} on faith.

One of the inherent difficulties of Floer cohomology is that the moduli spaces have to be given the structure of some complicated type of space; following \cite{fooo}, as is done in this paper, the complicated space is a Kuranishi space.
A Kuranishi space is a space that is locally modelled on the zero set of a section of an orbibundle over an orbifold.
Kuranishi spaces have virtual fundamental classes, and are the main technical tool used in \cite{fooo} to allow intersection theory with moduli spaces of holomorphic curves to be done.
For most of the purposes of this paper, the reader can think of a Kuranishi space as the zero set of a vector bundle over a closed manifold, or even more simply just as a closed manifold.

In the next section we set some notation that will be used throughout the paper.

\subsection{Conventions and notation}\label{section:overview-conventions}
Several conventions will be used throughout.
(Periodically the symbols will refer to something else, in which case it should be clear from context.)
\begin{itemize}
\item $\tau:X\to X$ is complex conjugation,
\item $L$ is the set of real points of $X$,
\item $L$ has a fixed orientation $or(L)$ and a fixed spin structure $P_{Spin}(L)$,
\item $gL$ has the fixed pushforward orientation $or(gL)=g_*or(L)$ and the fixed pushforward spin structure $P_{Spin}(gL)=g_*P_{Spin}(L)$,
\item $gL$ has a fixed grading (see Definition \ref{dfn:real_lagrangians-grading_convention}),
\item $L_0$ and $L_1$ are arbitrary Lagrangians in an arbitrary symplectic manifold,
\item $L_0$ and $L_1$ are said to intersect cleanly if $L_0\cap L_1$ is a submanifold and $TL_0\cap TL_1=T(L_0\cap L_1)$,
\item $R$ or $R_h$ is a connected component of $L_0\cap L_1$,
\item $\omega$ and $J$ are the standard symplectic form and complex structure on $X$,
\item $g=\omega(\cdot,J\cdot)$ is the standard Riemannian metric on $X$,
\item $h=g-i\omega$ is the standard Hermitian form on $(TX,J)$, and
\item $\nabla$ is the Levi-Civita connection associated to $g$.    
\end{itemize}

Cohomology will usually be taken with coefficients in $\qq$ or one of the Novikov rings
\begin{displaymath}
\Lambda_{0,nov}=\set{\sum_{i=1}^\infty q_iT^{\lambda_i}}{q_i\in \qq,\,\lambda_i\in\rr_{\geq0}\textrm{ and }\lambda_i\to\infty}
\end{displaymath}
or
\begin{displaymath}
\Lambda_{nov}=\set{\sum_{i=1}^\infty q_iT^{\lambda_i}}{q_i\in\qq,\,\lambda_i\in\rr \textrm{ and }\lambda_i\to\infty}.
\end{displaymath}
If the Novikov ring is used then it will be included in the notation.
$H(\cdot)$ will always denote singular cohomology and $HF(\cdot)$ Floer cohomology.
An isomorphism written as $HF^*\cong H^*$ indicates a graded isomorphism.

The strip $S=\rr\times[0,1]\subset\cc$ will have coordinate $s+it$.
$S$ is biholomorphic to $S_0=\set{z=x+iy=re^{i\theta}\in\cc}{r>0,0\leq\theta\leq e^{2\pi i/10}}$.
Holomorphic strips are maps $u_0:S\to X$ such that the bottom boundary of $S$ is mapped into $L_0$ and the top boundary is mapped into $L_1$.
The domain of $u_0$ can be thought of as $S$ or $S_0$, whichever is more convenient.
Sometimes the domain will be thought of as $D^2\setminus\sett{-1,1}$.
With this identification, the domain points $\pm\infty$ correspond to $\pm1$.

The moduli space of holomorhic strips with bottom boundary on $L_0$ and top boundary on $L_1$ is denoted $\mathcal M(L_0,L_1)$.
If $R_h,R_{h'}$ are connected components of $L_0\cap L_1$, then $\mathcal M(L_0,L_1:R_h,R_{h'})$ denotes the moduli space of strips that start in $R_h$ and end in $R_{h'}$.
$\mathcal M(L_0,L_1:R_h,R_{h'}:E)$ denotes the subset of these strips that have energy $E\in\rr_+$.
The energy of a strip $u$ is
\begin{displaymath}
E(u)=\int_S|du|^2dsdt.
\end{displaymath}
$\mathcal M(L_0,L_1:R_h,R_{h'}:\beta)$ denotes strips of homotopy class $\beta$.
Sometimes $\mathcal M_{0,0}$ will be used instead of $\mathcal M$ to refer to these moduli spaces.

$W^{1,p;\delta}_\lambda(\Sigma,E)$ and $L^{p;\delta}(\Sigma,\Lambda^{0,1}\otimes E)$ are weighted Sobolev spaces of sections of the complex bundles $E$ and $\Lambda^{0,1}\otimes E$ over the Riemann surface $\Sigma$ (usually $\Sigma$ will be $S$, $S_0$, $D^2$, or $\cp^1$).
($\Lambda^{0,1}\otimes E$ denotes the complex bundle $\Lambda^{0,1}(T\Sigma)\otimes_\cc E$.)
The weights will be explained in more detail in Section \ref{section:moduli_spaces}.
The $\lambda$ subscript denotes some sort of Lagrangian boundary condition.
Also, $p>2$ to ensure that $W^{1,p}$ consists of maps that are continuous.
Operators of the form $\dbar:W^{1,p;\delta}_\lambda(\Sigma,E)\to L^{p;\delta}(\Sigma,\Lambda^{0,1}\otimes E)$ will be Cauchy-Riemann operators (the exact zeroeth order part will either not be relevant or clear from the context), that is of the form
\begin{displaymath}
\xi(s,t)\mapsto (\nabla_s\xi+J\nabla_t\xi)\otimes(ds-idt),
\end{displaymath}
where $\nabla$ is a connection on $E$.
%Note that if $E$ has almost complex structure $J$ and $T\Sigma$ has complex structure $j$, then $\Lambda^{0,1}\otimes E$ is isomorphic to the set of $(j,J)$ anti-linear maps from $T\Sigma$ to $E$.
%(This requires, of course, that $s+it$ is a holomorphic coordinate on $\Sigma$.)

A holomorphic map is said to be regular if the cokernel of the linearized $\dbar$ operator is $0$.
The cokernel of the operator is also called the obstruction bundle.

Let $M$ be a smooth manifold, and let $f:A\to M$ and $g:B\to M$ be smooth maps.
The fiber product of $A$ and $B$ is
\begin{displaymath}
A\times_M B=A\times_{f\times g}B=\set{(a,b)\in A\times B}{f(a)=g(b)}.
\end{displaymath}
If $f$ is transverse to $g$ then the fiber product is a smooth manfiold.
If $f$ is not transverse to $g$, then the fiber product is not necessarily a smooth manifold.
If $A$ and $B$ are Kuranishi spaces, the fiber product $A\times_M B$ has the natural structure of a Kuranishi space.

If $\mathcal M$ is a Kuranishi space (generally it will be some moduli space), and $f:\mathcal M\to X$ is a map, then $f_*[\mathcal M]$ denotes the virtual fundamental class of $\mathcal M$, which is a singular chain on $X$.

If $F:U\to V$ is a Fredholm map between Banach spaces, the determinant line $\Det(F)$ is defined to be $\Det(\ker(F))\otimes\Det(\textrm{coker}(F))^*$, where $\Det$ of a vector space is the top exterior power of the vector space.

\section{Maslov indices and Fredholm theory}
In this section we review the various Maslov indices and Fredholm operators that will be needed.
Also, we write down some useful formulas for the Maslov index that will be useful later on.
%The Maslov indices enter into the formulas for the indices of the Fredholm operators.
The Fredholm operators are all linearized versions of various $\dbar$ operators.

The following Maslov indices will be defined:
\begin{enumerate}
\item $\mu(u)$ for a disc $u$,
\item $\mu(u)$ for a strip $u$,
\item $\mu^{poly}(u)$ for a polygon $u$, 
%\item $\mu([w,h];\lambda_0)$ for a homotopy class $[w,h]$.
\item $\mu(\lambda;\Lambda_0,\Lambda_1)$ for a path of Lagrangian subspaces from $\Lambda_0$ to $\Lambda_1$, 
\item $\mu(\lambda,\Lambda_0)$ for a Lagrangian path $\lambda$ and fixed Lagrangian subspace $\Lambda_0$, and
\item $\mu(\lambda_0,\lambda_1)$ for a pair of Lagrangian paths $\lambda_0$ and $\lambda_1$.
\end{enumerate}

The following Fredholm operators will be defined:
\begin{eqnarray}
\nonumber&&D^{disc}:W^{1,p}_\lambda(D^2,u^*TX)\to L^p(D^2,\Lambda^{0,1}\otimes u^*TX),\\
\nonumber&&D^{strip}:W^{1,p;\delta}_\lambda(\rr\times[0,1],u^*TX)\to L^{p;\delta}(\rr\times[0,1],\Lambda^{0,1}\otimes u^*TX),\\
\label{eq:index_theory-fredholm_operators}&&D^{poly}:W^{1,p;\delta}_\lambda(\Sigma_{k+1},u^*TX)\to L^{p;\delta}(\Sigma_{k+1},\Lambda^{0,1}\otimes u^*TX),\\
\nonumber&&\dbar_{\lambda,Z_{\pm}}:W^{1,p;\delta}_\lambda(Z_{\pm},T_{p}X)\to L^{p;\delta}(Z_{\pm},\Lambda^{0,1}\otimes T_{p}X).
\end{eqnarray}
All of the operators are Cauchy-Riemann operators.%, i.e. of the form $\del_s+i\del_t$ where $z=s+it$ is the coordinate on the domain.
The Banach spaces and other notation will be explained below.
The next two lemmas summarize the main relationships between these indices.

\begin{lemma}\label{lemma:index_theory-strip_poly}
If $u$ is a strip (i.e. a 2-gon, holomorphic or not) then
$$\mu(u)-\dim(R_h)=\mu^{poly}(u)-n,$$
where $R_h$ is the component of $L\cap gL$ that contains $u(-\infty,\cdot)$.
%If $[w,h],[w',h']$ are homotopy classes such that $w\#u=w'$ (up to homotopy) then
%$$\mu(u)=\mu([w',h'];\lambda_0)-\mu([w,h];\lambda_0).$$
\end{lemma}
\begin{lemma}\label{lemma:index_theory-fredholm_indices}
\begin{eqnarray*}
&&\Ind(D^{disc})=\mu(u)+n,\\
&&\Ind(D^{strip})=\mu(u)-\dim(R_h)=\mu^{poly}-n,\\
&&\Ind(D^{poly})=\mu^{poly}(u)-kn,\\
%&&\Ind(\dbar_{\lambda_{\lambda_0,w},Z_-})=\mu([w,h];\lambda_0).
&&\Ind(\dbar_{\lambda,Z_-})=\mu(\lambda;T_pL_0,T_pL_1).
\end{eqnarray*}
\end{lemma}
The first two formulas are well-known.
The third can be thought of as a generalization of the second formula to polygons.
(The notion of holomorphic polygons is not needed in this paper; however, it is anticipated that it will be useful in later work, and it is no more difficult to prove than the case of strips, so it is included here.)
The final formula is useful for proving other formulas.
The proofs will be given at the end of this section.

In constructing the indices the following problem will often be encountered:
Given a symplectic vector space V and two Lagrangian subspaces $\Lambda_0,\Lambda_1$, construct a path from $\Lambda_0$ to $\Lambda_1$ that moves in the positive definite direction (for the origin of the terminology ``positive definite direction'', see Section \ref{section:index_theory-maslov_index_paths}).
If $\Lambda_0$ and $\Lambda_1$ intersect transversely, the construction of such a path is well-known and unique up to homotopy with fixed endpoints.
The easiest way to construct the path is to choose a complex structure $J$ compatible with $\omega$ such that $J\cdot\Lambda_0=\Lambda_1$, then the path can be taken to be $t\mapsto e^{\pi Jt/2}\cdot\Lambda_0$.
The fact that this does not depend on the choice of $J$ follows from the fact that if $J_0$, $J_1$ are two such complex structures then there exists a one parameter family $J_t$ of complex structures that joins $J_0$ to $J_1$ and is such that $J_t\cdot\Lambda_0=\Lambda_1$ for all $t$.
The following alternate characterization of this path will be useful.
\begin{lemma}\label{lemma:index_theory-matrix}
Let $(V,\omega,J)$ be given.
Suppose $\Lambda_0$ and $\Lambda_1$ are two Lagrangian subspaces such that $\Lambda_0\cap\Lambda_1=0$.
Then there exists a unitary basis $e_1,\ldots,e_n$ of $V$ and a $U\in \textrm{Aut}(V,\omega,J)$ (i.e. $U$ is unitary) such that the $\rr$-span of $e_1,\ldots,e_n$ is $\Lambda_0$, $\Lambda_1=U\cdot\Lambda_0$, and the matrix representation of $U$ with respect to the basis $e_1,\ldots,e_n$ is
\begin{displaymath}
U=\left[
\begin{array}{ccc}
e^{2\pi i \alpha_1}&&0\\
&\ddots &\\
0&&e^{2\pi i\alpha_n}
\end{array}\right],
\end{displaymath}
where $0<\alpha_i<\frac{1}{2}$.
The positive definite path $\lambda:[0,1]\to\Lambda(V)$ from $\Lambda_0$ to $\Lambda_1$ can be taken to be $\lambda(t)=U^t\cdot L_0$, where $U^t$ is the $t^{th}$ power of $U$.
\end{lemma}
\begin{proof}
Without loss of generality assume $(V,\omega,J)=(\cc^n,\omega_0,J_0)$ and $\Lambda_0=\rr^n$.
$\Lambda_1$ is transverse to $\rr^n$ so there exists an $n\times n$ symmetric matrix $A$ with real entries such that $\Lambda_1=(A+iI)\cdot\rr^n$.
Let $U=(A+iI)(A^2+I)^{-1/2}$. 
Then $U$ is unitary and $U\cdot\rr^n=\Lambda_1$.
Let $\lambda_1,\ldots,\lambda_n\in\rr$ be the eigenvalues of $A$ and $e_1,\ldots,e_n\in\rr^n$ a corresponding orthononormal set of eigenvectors.
Then $U$ is diagonalizable, with eigenvalues $(\lambda_i+\sqrt{-1})/(\lambda_i^2+1)^{-1/2}=e^{2\pi i\alpha_i}$ and eigenvectors $e_i$.
With respect to the basis $e_1,\ldots,e_n$, $U$ has the form above.

If $J'$ is a new complex structure defined by $J'e_i=Ue_i$, $J'Ue_i=-e_i$ then $\lambda$ is homotopic to $t\mapsto e^{\pi J't/2}\cdot\rr^n$.
This proves that $\lambda$ is the positive definite path, up to homotopy.
\end{proof}

Some other properties of the numbers $\alpha_1,\ldots,\alpha_n$ that will be useful later on are:
\begin{lemma}\label{lemma:index_theory-linear_algebra_1}
The linear transformation $U$ given in Lemma \ref{lemma:index_theory-matrix} is unique.
In particular, the numbers $\alpha_1,\ldots,\alpha_n$ are well-defined.
%That is, if $U'$ is another matrix of the form given in Lemma \ref{lemma:index_theory-matrix}, then the eigenvalues of $U'$ are
%$e^{2\pi i \alpha_1},\ldots,e^{2\pi i \alpha_n}$.
\end{lemma}
\begin{proof}
Let $U'$ be another unitary transformation of the form given in the lemma.
Then by assumption, $\Lambda_0$ has a basis $$e_1',\ldots,e_n'$$ of eigenvectors of $U'$, with eigenvalues $$e^{2\pi i\alpha_1'},\ldots,e^{2\pi i\alpha_n'}$$ with $0<\alpha_i'<\frac{1}{2}$.
Thus $U'e_i'=e^{2\pi i\alpha_i'}e_i'\in \Lambda_1$.
On the other hand, $e_i'=c_1e_1+\cdots+c_ne_n$ for some real numbers $c_i$, so
%\begin{displaymath}
%Ue_i'=c_1e^{2\pi i\alpha_1}e_1+\cdots+c_ne^{2\pi i \alpha_n}e_n\in\Lambda_1
%\end{displaymath}
\begin{displaymath}
e^{2\pi i\alpha_i'}e_i'=c_1e^{2\pi i\alpha_i'}e_1+\cdots+c_ne^{2\pi i \alpha_i'}e_n\in\Lambda_1.
\end{displaymath}
Furthermore, $\Lambda_1=U\cdot \Lambda_0$, so
\begin{displaymath}
c_1e^{2\pi i\alpha_i'}e_1+\cdots+c_ne^{2\pi i \alpha_i'}e_n=d_1e^{2\pi i\alpha_1}e_1+\cdots+d_ne^{2\pi i\alpha_n}e_n
\end{displaymath}
for some real numbers $d_1,\ldots,d_n$.
Since $e_1,\ldots,e_n$ is a basis for $V$ over $\cc$, it follows that $c_je^{2\pi i\alpha_i'}=d_je^{2\pi i\alpha_j}$ for each $j$.
It follows that $\alpha_i'=\alpha_j$ for some $j$ and $e_i'$ is an eigenvector of $U$.

Switching the roles of $U$ and $U'$ shows that every eigenvector of $U$ is an eigenvector of $U'$ with the same eigenvalue.
It follows that $U=U'$.
\end{proof}

\begin{lemma}\label{lemma:index_theory-linear_algebra_2}
Suppose $U'$ is a unitary matrix, $\Lambda_0$ has a basis consisting of eigenvectors of $U'$, and $\Lambda_1=U'\cdot \Lambda_0$.
Then the eigenvectors of $U'$ are $\pm e^{2\pi i\alpha_1},\ldots,\pm e^{2\pi i\alpha_n}$.
\end{lemma}
\begin{proof}
The proof is similar to the proof of the previous lemma, except near the end where $c_je^{2\pi i\alpha_i'}=d_je^{2\pi i\alpha_j}$ for each $j$ implies $\alpha_i'=\alpha_j$ for some $j$, the equation instead implies that $e^{2\pi i\alpha_i'}=\pm e^{2\pi i\alpha_j}$ for some $j$ (because it is no longer required that $0<\alpha_i'<\frac{1}{2}$ in this lemma).
\end{proof}

\begin{lemma}\label{lemma:index_theory-linear_algebra_3}
Let $\tau:V\to V$ be a linear transformation such that $\omega(\tau\cdot,\tau\cdot)=-\omega(\cdot,\cdot)$, $\tau\circ J=-J\circ\tau$, and $\tau^2=1$.
Let $\Lambda_0=\textrm{Fix}(\tau)$.
Suppose $U$ is a unitary transformation such that $\Lambda_1=U\cdot \Lambda_0$ is transverse to $\Lambda_0$.
Furthermore, suppose that $U\circ\tau=\tau\circ U^{-1}$.
Then $\Lambda_0$ has a basis consisting of eigenvectors of $U$.
\end{lemma}
\begin{proof}
Without loss of generality assume that $(V,\omega,J)=(\cc^n,\omega_0,J_0)$, $\Lambda_0=\rr^n$, and $\tau$ is the usual complex conjugation.
Then $U$ is a unitary matrix, so $U^{-1}=\bar U^T$.
The equation $U\circ\tau=\tau\circ U^{-1}=\tau\bar U^T$ implies 
\begin{displaymath}
U\bar v=\overline{\bar U^Tv}=U^T\bar v
\end{displaymath}
for all $v\in\cc^n$.
It follows that $U=U^T$.
Thus $U$ is of the form $U=A+iB$ where $A$ and $B$ are real symmetric matrices.
Thus
\begin{displaymath}
1=UU^{-1}=U\bar U=(A+iB)(A-iB)=A^2+B^2+i(AB-BA).
\end{displaymath}
Therefore $AB=BA$, and it follows that $A$ and $B$ are simultaneously diagonalizable.
That is, $\rr^n$ has a basis consisting of eigenvectors of $U$.
\end{proof}

If $\Lambda_0$ and $\Lambda_1$ do not intersect transversely, the positive definite path is defined as follows.
Let $N=(\Lambda_0+\Lambda_1)/(\Lambda_0\cap\Lambda_1)$, and view $N$ as a symplectic subspace of $V$.
Let $\bar\Lambda_i$ be the images of $\Lambda_i$.
Let $\bar\lambda$ be the positive definite path from $\bar\Lambda_0$ to $\bar\Lambda_1$ in $\Lambda(N)$.
Then the positive definite path from $\Lambda_0$ to $\Lambda_1$ is defined to be $\lambda=\bar\lambda+(\Lambda_0\cap\Lambda_1)$.

We now turn to the definitions of the Maslov indices and Fredholm operators.

\subsection{(1) The case of a disc}
The Maslov index $\mu(u)$ of a map $u:(D^2,\partial D^2)\to (X,L)$ is well-known, it is the Maslov index of the bundle pair $(u^*TX,u^*TL)$.
Linearizing the $\dbar$ operator at $u$ gives a Fredholm operator
$$D^{disc}=D_u\dbar:W^{1,p}_\lambda(D^2,u^*TX)\to L^p(D^2,\Lambda^{0,1}\otimes u^*TX).$$
The $\lambda$ subscript denotes the Lagrangian boundary conditions given by $u^*TL$.
The Riemann-Roch theorem states that $\Ind(D^{disc})=\mu(u)+n$, where $n=\dim(L)$.

\subsection{(2) The case of a strip}
Let $L_0,L_1$ be two Lagrangian submanifolds of $X$ that intersect cleanly.
Let $u:\rr\times[0,1]\to X$ be a map with top boundary on $L_1$, bottom boundary on $L_0$, and $$u(-\infty,\cdot),u(+\infty,\cdot)\in L_0\cap L_1.$$
Assume that $u$ converges in an appropriate sense at $\pm\infty$ (in particular, assume that $u(\pm\infty,t)$ does not depend on $t$).
Let $p=u(-\infty,\cdot)$ and $q=u(+\infty,\cdot)$.
By identifying the compactification of $\rr\times[0,1]$ with a square, $u^*TX$ becomes a symplectic vector bundle over the square.
A Lagrangian subbundle over the boundary of the square can then defined in a canonical way (up to homotopy):
The top boundary is $u(\cdot,1)^*TL_1$ and the bottom boundary is $u(\cdot,0)^*TL_0$.
For the left boundary choose a positive definite path from $T_pL_0$ to $T_pL_1$ in the Lagrangian grassmannian $\Lambda(T_pX)$, and for the right boundary choose a positive definite path from $T_qL_0$ to $T_qL_1$ in the Lagrangian grassmannian $\Lambda(T_qX)$.
The Maslov index $\mu(u)$ of the strip $u$ is defined to be the Maslov index of this bundle pair.

\subsection{(3) The polygonal Maslov index}
The polygonal Maslov index plays a role for holomorphic polygons similar to the role that the Maslov index of a strip plays for a strip.
(However, note that they are not the same for a 2-gon.
They in fact differ by a factor of $n$.)

Let $L_0,\ldots,L_k$ be Lagrangian submanifolds such that $L_{i-1}$ and $L_i$ intersect cleanly.
Let $u:D^2\to X$ be a map, and let $z_0,\ldots,z_k$ be distinct points on the boundary of $D^2$, labelled in counterclockwise cyclic order.
Suppose further that $u$ maps the boundary of $D^2$ between $z_{i-1}$ and $z_i$ to $L_i$.
$u^*TX$ is then a symplectic vector bundle over $D^2$, and $u^*TL_i$ is a Lagrangian subbundle over the boundary arc from $z_{i-1}$ to $z_i$.
This subbundle can be extended to a Lagrangian subbundle over the entire boundary of $D^2$ by slightly homotoping the given Lagrangian subbundles to be constant near the marked points, and then extending the bundle over the marked points by moving in the positive definite direction from $TL_{i-1}$ to $TL_i$.
$\mu^{poly}(u)$ is defined to be the Maslov index of this bundle pair.

\subsection{(4) The case $\mu(\lambda;\Lambda_0,\Lambda_1)$}
Let $(V,\omega)$ be a symplectic vector space.
Let $\Lambda_0$ and $\Lambda_1$ be Lagrangian subspaces, and let $\lambda:[0,1]\to\Lambda(V)$ be a path of Lagrangians from $\Lambda_0$ to $\Lambda_1$.
Let $\lambda':[0,1]\to\Lambda(V)$ be the positive definite path from $\Lambda_0$ to $\Lambda_1$.
Define $\tilde\lambda:S^1\to \Lambda(V)$ by
\begin{displaymath}
\tilde\lambda(e^{\pi i\theta})=\left\{
\begin{array}{ll}
\lambda'(\theta), & 0\leq \theta\leq 1,\\
\lambda(2-\theta), & 1\leq \theta\leq 2.
\end{array}
\right.
\end{displaymath}
Then $\mu(\lambda;\Lambda_0,\Lambda_1)$ is defined to be the Maslov index of the loop $\tilde\lambda$.

\subsection{(5-6) The cases $\mu(\lambda,\Lambda_0)$ and $\mu(\lambda_0,\lambda_1)$}\label{section:index_theory-maslov_index_paths}
These are the Maslov indices for paths, as defined in \cite{rs}, and they are half-integer valued.
We will need to calculate some of these indices, so we repeat here the definition.
Let $\lambda,\lambda_0,\lambda_1$ be paths of Lagrangian subspaces and let $\Lambda_0$ be a fixed Lagrangian subspace.
The index $\mu(\lambda,\Lambda_0)$ is a special case of the index $\mu(\lambda_0,\lambda_1)$, namely the case where $\lambda_0=\lambda$ and $\lambda_1\equiv \Lambda_0$.
Therefore it suffices to define $\mu(\lambda_0,\lambda_1)$.

Some preliminary definitions are needed first.
The tangent space of $\Lambda(V)$ at $\Lambda_0$ is isomorphic to the space of symmetric quadratic forms on $\Lambda_0$.
The isomorphism can be explicitly given:
Let $W$ be a fixed Lagrangian complement of $\Lambda_0$.
Let $\lambda(t)$ be a smooth path of Lagrangian subspaces with $\lambda(0)=\Lambda_0$.
For $v\in\Lambda_0$ and small $t$, there exists a unique vector $w(t)\in W$ such that $v+w(t)\in \lambda(t)$.
The quadratic form $Q(\Lambda_0,\dot\lambda(0))$ associated to the tangent vector $\dot\lambda(0)$ is given by the formula
\begin{displaymath}
Q(\Lambda_0,\dot\lambda(0))(v)=\frac{d}{dt}\biggr|_{t=0}\omega(v,w(t)).
\end{displaymath}
(Notice that if $\lambda(t)=e^{\pi J t/2}\cdot \Lambda_0$, then $Q(\Lambda_0,\dot\lambda(0))$ is positive-definite, hence the terminology ``positive definite direction''.)

If $\lambda(t)$ is any path of Lagrangian subspaces and $\Lambda_0$ is any Lagrangian subspace, the crossing form $\Gamma$ is defined by
\begin{displaymath}
\Gamma(\lambda,\Lambda_0,t)=Q(\lambda(t),\dot\lambda(t))|_{\lambda(t)\cap \Lambda_0}.
\end{displaymath}
This is a symmetric quadratic form on $\lambda(t)\cap \Lambda_0$, hence its signature (that is, the number of positive eigenvalues minus the number of negative eigenvalues) is well-defined.

Now the definition of $\mu(\lambda_0,\lambda_1)$ can be given:
\begin{eqnarray*}
\mu(\lambda_0,\lambda_1)&=&\frac{1}{2}\textrm{sign}(\Gamma(\lambda_0,\lambda_1(0),0)-\Gamma(\lambda_1,\lambda_0(0),0))+\\
&&\sum_{0<t<1}\textrm{sign}(\Gamma(\lambda_0,\lambda_1(t),t)-\Gamma(\lambda_1,\lambda_0(t),t))+\\
&&\frac{1}{2}(\textrm{sign}(\Gamma(\lambda_0,\lambda_1(1),1)-\Gamma(\lambda_1,\lambda_0(1),1)).
\end{eqnarray*}
For a generic perturbation of $\lambda_0$ and $\lambda_1$ with fixed endpoints, the sum becomes a finite sum and all the quadratic forms
$$\Gamma(\lambda_0,\lambda_1(t),t)-\Gamma(\lambda_1,\lambda_0(t),t)$$
are non-degenerate.
In this case the sum is well-defined, and it can be shown that the sum does not depend upon the choice of generic perturbation.

%\subsection{(4) The case $\mu([w,h];\lambda_0)$}

%Let $L_0,L_1$ be two cleanly intersecting oriented Lagrangians, $\ell_0:[0,1]\to M$ a fixed path with $\ell_0(0)\in L_0,\ell_0(1)\in L_1$, and $\lambda_0:[0,1]\to \Lambda^{ori}(TM)$ a path of oriented Lagrangian planes covering $\ell_0$ such that $\lambda_0(0)=T_{\ell_0(0)}L_0$ and $\lambda_0(1)=T_{\ell_0(1)}L_1$.
%Let $R_h$ be a component of $L_0\cap L_1$ and let $p_h\in R_h$.
%Let $w:[0,1]\times[0,1]\to X$ be a map such that
%\begin{equation}\label{eq:pathofpaths}
%\left\{
%\begin{array}{l}
%w(0,t)=\ell_0(t),\\
%w(1,t)=p_h,\\
%w(s,0)\in L_0,\\
%w(s,1)\in L_1.
%\end{array}
%\right.
%\end{equation}
%$w$ can also be thought of as a path in the space of paths from $L_0$ to $L_1$.
%Let $[w,h]$ denote the homotopy class of the path $w$.

%The top and bottom boundary conditions of $w$, the positive definite path from $T_{p_h}L_0$ to $T_{p_h}L_1$, and the chosen path $\lambda_0$ piece together to give a Lagrangian subbundle $F$ over the boundary of the square.
%The Maslov index $\mu([w,h];\lambda_0)$ is defined to be the Maslov index of the bundle pair $(w^*TX,F)$.

%$\mu([w,h];\lambda_0)$ is a relative index, in the sense that it depends on $\lambda_0$.
%However, the difference of two such indices does not depend on $\lambda_0$, that is
%\begin{displaymath}
%\mu([w,h];\lambda_0)-\mu([w',h'];\lambda_0)
%\end{displaymath}
%is independent of $\lambda_0$.

\subsection{Fredholm theory}
Now consider the Fredholm operators in (\ref{eq:index_theory-fredholm_operators}).
They are all defined to be Cauchy-Riemann operators.
The superscript $\delta$ on the Banach spaces means that weight\-ed norms are used.
The weight is $\delta(s)=e^{\delta_0|s|}$, where $\delta_0>0$ is sufficiently small, so the norms are of the form
\begin{displaymath}
\|\xi\|_{L^{p;\delta}}^p=\int |\xi|^pe^{\delta(s)} dsdt.
\end{displaymath}
%Note that $p>2$ so that the integral is not conformally invariant.
%That is, it is important that the standard holomorphic coordinates $s+it$ on $\rr\times [0,1]$ are used to define this integral.
The weights are needed to make the operators closed when the boundary conditions are not transverse at $\pm\infty$.
They are not needed if the boundary conditions are transverse; however if $\delta_0>0$ is sufficiently small then the inclusion of the weight does not affect the index.
The $\lambda$ subscript means that Lagrangian  boundary value conditions are imposed.
The conditions imposed are the obvious ones in the first three operators.
(In the definition of the third operator, $\Sigma_{k+1}$ stands for disc with $k+1$ marked boundary points removed.)
The last operator requires more explanation.
For $p\in R_h\subset L_0\cap L_1$, let $\lambda:[0,1]\to\Lambda(T_pX)$ be a path such that $\lambda(0)=T_pL_0$ and $\lambda(1)=T_pL_1$.
Let
\begin{eqnarray}\label{eq:index_theory-caps}
Z_-&=&\set{z\in\cc}{|z|\leq1}\cup\set{z\in\cc}{\textrm{Re } z\geq0,|\textrm{Im } z|\leq 1},\\
\nonumber Z_+&=&\set{z\in\cc}{|z|\leq1}\cup\set{z\in\cc}{\textrm{Re } z\leq0,|\textrm{Im } z|\leq 1}.
\end{eqnarray}
$\lambda$ gives Lagrangian boundary conditions along $Z_\pm$.
Namely, the top boundary condition is $T_{p}L_1$, the bottom is $T_{p}L_0$, and the semicircle is $\lambda$ (parameterizing the semicircle as $[0,1]$, with $0$ the bottom point and $1$ the top point).
Then associated to these boundary conditions are the operators
\begin{equation}\label{eq:index_theory-cap_operators}
\bar\partial_{\lambda,Z_{\pm}}:W^{1,p;\delta}_\lambda(Z_{\pm},T_{p}M)\to L^{p;\delta}(Z_{\pm},\Lambda^{0,1}(Z_{\pm})\otimes T_{p}M).
\end{equation}

%Let $[w,h]$ be a homotopy class as before.
%To each such path can be associated a (homotopy class of) $\lambda$ as above.
%Namely, choose a trivialization $w^*TM\to [0,1]\times[0,1]\times T_{p_h}M$ that is constant along the right edge $\sett{1}\times[0,1]$.
%$w$ is covered by an oriented Lagrangian bundle along the bottom, left, and top boundaries of the square $[0,1]\times[0,1]$: the bottom side is covered by $TL_0$, the left side is covered by $\lambda_0$, and the top side is covered by $TL_1$.
%Topologically, the union of these three sides is just an interval, and thus the bundle can be viewed as an element $\lambda_{\lambda_0,w}$ of $\mathcal P_{R_h}(T_{p_h}L_0,T_{p_h}L_1)$.
%In particular, an operator is associated to the path $\lambda_{\lambda_0,w}$.

\begin{proof}[Proof of Lemma \ref{lemma:index_theory-strip_poly}]
Let $u$ be a strip with $u(-\infty,\cdot)\in R_h$.
%Consider first the statement
%$$\mu(u)-\dim(R_h)=\mu^{poly}(u)-n.$$
Let $TL_0$ be the bottom Lagrangian boundary condition, and $TL_1$ the top Lagrangian boundary condition.
The difference in the bundle pairs used to define $\mu$ and $\mu^{poly}$ lies in how the Lagrangian subbundles are extended over $-\infty$.
To define $\mu(u)$, the positive definite path is taken from $TL_0$ to $TL_1$, and then traversed backwards (because the orientation of the boundary is counter-clockwise direction).
To define $\mu^{poly}(u)$, the positive definite path is taken from $TL_1$ to $TL_0$, and then traversed forwards.
Thus $\mu^{poly}(u)-\mu(u)$ is equal to the Maslov index of a positive definite loop at $-\infty$. 
Since only the transverse part moves at $-\infty$, the positive definite loop has Maslov index $n-\dim(R_h)$.

%Now consider the statement
%\begin{displaymath}
%\mu(u)+\mu([w,h])=\mu([w',h']),
%\end{displaymath}
%for $w\#u=w'$.
%The common boundary of $u$ and $w$ cancel out on the left-hand side, because it is traversed in opposite directions.
%The remaining boundary is homotopic to the boundary of $w'$.
\end{proof}

\begin{proof}[Proof of Lemma \ref{lemma:index_theory-fredholm_indices}]
The first statement is well-known.
The last statement is Lem\-ma 12.69 in \cite{fooo}.
The second statement,
$$\Ind(D^{strip})=\mu(u)-\dim(R_h),$$
follows from Lemmas 12.64 and 12.69 in \cite{fooo}.

Now consider the third statement,
\begin{displaymath}
\Ind(D^{poly})=\mu^{poly}(u)-kn.
\end{displaymath}
Recall that the domain of $u$ is $\Sigma_{k+1}$, which is a disc with $k+1$ boundary points removed.
The proof is by induction:
If $k+1=2$, then $D^{poly}=D^{strip}$ and the formula is equivalent to $\Ind (D^{strip})=\mu(u)-\dim(R_h)$.
Now suppose the statement is true for $k\geq 1$.
The proof for $k+1$ then follows from the gluing principle:
Suppose the marked points on $\partial D^2$ are $z_0,\ldots,z_k$, and suppose $z_0=-1$.
In strip-like coordinates near $z_0$, think of $-1$ as corresponding to $(-\infty,\cdot)$ in the strip.
The top boundary of the strip is $TL_0$, and the bottom boundary is $TL_1$.
Glue a cap $Z_-$ onto the $-\infty$ end of the strip-like end, and take the path along the semicircle end of $Z_-$ in the up direction to be the positive definite path from $TL_1$ to $TL_0$, and let $w$ denote this path.
The glued domain can be thought of as the domain of a $k$-gon $u_k$, and $\Ind (D^{poly}_{u_k})=\mu^{poly}(u_k)-(k-1)n$ by the induction hypothesis.
By the gluing principle,
\begin{displaymath}
\Ind (\dbar_{\lambda,Z_-})+\Ind (D^{poly})+\dim R_h=\Ind (D^{poly}_{u_k}).
\end{displaymath}
The term $\dim R_h$ appears on the left-hand side because Banach spaces with weights are used, so the strip like end at $z_0$ cannot ``move'' within $R_h$, but on the right-hand side the marked point $z_0$ is no longer present, so the values can ``move'' within $R_h$.
Therefore 
\begin{displaymath}
\Ind(D^{poly})=\mu^{poly}(u_k)-\mu(w;TL_1,TL_0)-\dim R_h-(k-1)n.
\end{displaymath}
By construction, $\mu(w;TL_1,TL_0)=0$, and it is straightforward to check that $$\mu^{poly}(u_k)=\mu^{poly}(u)-(n-\dim R_h),$$ thus finishing the proof.
\end{proof}

\section{Gradings}\label{section:gradings}
We begin this section with a brief review of graded Lagrangians in the context of Calabi-Yau manifolds.
Then some of the Maslov index theory is recast in this context.
The main point of this theory is that it allows an absolute grading to be put on the Floer cohomology.
This application will be developed later, in Section \ref{section:a_infinity-floer_cohomology}.
Most of the material of this section is from \cite{seidel1} or is a straightforward generalization thereof to the clean intersection case.

\subsection{Graded Lagrangians}\label{section:gradings-graded}
For the remainder of this section, let $(X,\omega,J,\Omega)$ be any Calabi-Yau n-fold.
$\Omega$ is a non-vanishing section of the canonical bundle $K_X$ (it is not required that $\nabla \Omega=0$).
Let $\mathcal L\to X$ be the fiber bundle whose fiber $\mathcal L_p$ over the point $p\in X$ is the Lagrangian Grassmannian $\Lambda(T_pX,\omega_p)$.
The section $\Omega$ can be used to define a map
\begin{displaymath}
\Det^2:\mathcal L\to S^1.
\end{displaymath}
Namely, if $\Lambda_p\in \mathcal L_p$ is a Lagrangian plane, let $e^{i\theta}$ be such that $\textrm{Im}(e^{-i\theta}\Omega_p)|_{\Lambda_p}=0$.
The square of $e^{i\theta}$ is uniquely determined and $\Det^2(\Lambda_p)$ is defined to be $e^{2i\theta}\in S^1$.
Let $L\subset X$ be a Lagrangian submanifold and let $s_L:L\to \mathcal L$ be the tautological map.
\begin{definition}
$L$ is ($\z$-)graded if there exists a function $\theta_L:L\to\rr$ such that $\Det^2\circ s_L=e^{2\pi i\theta_L}$.
$L$ has constant phase if $\theta_L$ is constant.
\end{definition}

\subsection{An absolute index: the transverse case}
Let $L_0$ and $L_1$ be two transversely intersecting Lagrangian submanifolds of $X$, and suppose they have gradings $\theta_{L_0}$ and $\theta_{L_1}$.
To each point $p\in L_0\cap L_1$ can be associated an integer $\tilde \mu(L_0,L_1;p)$, defined as follows.
Choose a path $(\lambda,\theta):[0,1]\to \mathcal L_p\times\rr$ such that $\lambda(0)=T_pL_0$, $\lambda(1)=T_pL_1$, $\theta(0)=\theta_{L_0}(p)$, $\theta(1)=\theta_{L_1}(p)$, and $e^{2\pi i\theta(t)}=\Det^2(\lambda(t))$.
\begin{definition}
The integer
\begin{displaymath}
\tilde\mu(L_0,L_1;p)=\frac{n}{2}-\mu(\lambda,T_pL_0)
\end{displaymath}
is called the (absolute) index between $(L_0,\theta_{L_0})$ and $(L_1,\theta_{L_1})$ at $p$.
\end{definition}
Equivalently, let $(\lambda_i,\theta_i):[0,1]\to\mathcal L_p\times\rr$ be such that $\lambda_0(0)=\lambda_1(0)$, $\theta_0(0)=\theta_1(0)$, $\theta_i(1)=\theta_{L_i}(p)$, $\lambda_0(1)=T_pL_0$, $\lambda_1(1)=T_pL_1$, and $\Det^2(\lambda_i(t))=e^{2\pi i\theta_{i}(t)}$.
Then
\begin{displaymath}
\tilde\mu(L_0,L_1;p)=\frac{n}{2}-\mu(\lambda_1,\lambda_0)=\frac{n}{2}+\mu(\lambda_0,\lambda_1).
\end{displaymath}
(Notice that the sign in front of $\mu(\lambda_0,\lambda_1)$ is different than that given in \cite{seidel1}.
The reason is that Seidel takes the Floer cochain complex to be the dual of the Floer chain complex.
In \cite{fooo} and this paper, cohomological grading is achieved not by dualizing but instead simply by redindexing.)
\begin{lemma}\label{lemma:gradings-index_formula}
Suppose $T_pL_1=U\cdot T_pL_0$ where $U$ is a unitary transformation of the form given in Lemma \ref{lemma:index_theory-matrix}.
Suppose the eigenvalues of $U$ are $e^{2\pi i\alpha_1},\ldots,e^{2\pi i\alpha_n}$, where $0<\alpha_i<\frac{1}{2}$.
Then
\begin{displaymath}
\tilde\mu(L_0,L_1;p)=\theta_{L_0}(p)-\theta_{L_1}(p)+2(\alpha_1+\cdots+\alpha_n).
\end{displaymath}
Moreover, $\tilde\mu(L_0,L_1;p)$ is an integer.
\end{lemma}
\begin{proof}
Let $\theta_0=\theta_{L_0}(p)$ and $\theta_1=\theta_{L_1}(p)$.
Let $N=\theta_1-\theta_0-2(\alpha_1+\cdots+\alpha_n)$.
Let $e_1,\ldots,e_n$ be an orthonormal frame of $T_pL_0$ such that $e_i$ is an eigenvector of $U$ with eigenvalue $e^{2\pi i\alpha_i}$.
Then $e_1\wedge\cdots \wedge e_n$ is the volume form on $T_pL_0$ and $Ue_1\wedge\cdots\wedge Ue_n$ is the volume form on $T_pL_1=U\cdot T_pL_0$.
$\Det^2(T_pL_0)=e^{2\pi i\theta_0}$ so $\textrm{Im }e^{-\pi i\theta_0}\Omega(e_1\wedge\cdots\wedge e_n)=0$.
Therefore
\begin{displaymath}
\textrm{Im }\Det(U)^{-1}e^{-\pi i \theta_0}\Omega(Ue_1\wedge \cdots\wedge Ue_n)=\textrm{Im }e^{-\pi i\theta_0}\Omega(e_1\wedge\cdots\wedge e_n)=0
\end{displaymath}
and it follows that $\Det(U)^2=e^{2\pi i(\theta_1-\theta_0)}$.
Therefore
\begin{displaymath}
e^{2\pi i 2(\alpha_1+\cdots+\alpha_n)}=e^{2\pi i(\theta_1-\theta_0)}
\end{displaymath}
and thus $N$ is an integer.

With respect to the basis $e_1,\ldots,e_n$, $U$ is diagonal and the $t^{th}$ power $U^t$  has the matrix form
\begin{displaymath}
U^t=\left[
\begin{array}{ccc}
e^{2\pi i\alpha_1t} &&0\\
&\ddots&\\
0&&e^{2\pi i\alpha_n t}
\end{array}
\right].
\end{displaymath}
Let $E^t$ be the matrix
\begin{displaymath}
E^t=\left[
\begin{array}{cccc}
e^{\pi itN} &&&0\\
&1&&\\
&&\ddots&\\
0&&&1
\end{array}
\right].
\end{displaymath}
Let $\lambda(t)=U^t E^t\cdot T_pL_0$.
Notice that the entries of $E^1$ consist of $\pm1$'s, so $E^1\cdot T_pL_0=T_pL_0$.
Therefore $\lambda(0)=T_pL_0$, $\lambda(1)=T_pL_1$, and
\begin{displaymath}
\Det^2(\lambda(t))=e^{2\pi i t((2\alpha_1+\cdots+2\alpha_n)+N)}\Det^2(T_pL_0)=e^{2\pi i (t(\theta_1-\theta_0)+\theta_0)}.
\end{displaymath}
Let $\theta(t)=\theta_0+t(\theta_1-\theta_0)$.
Then $(\lambda,\theta):[0,1]\to\Lambda(T_pX)\times\rr$ is a path with the required properties, so
\begin{displaymath}
\tilde\mu(L_0,L_1;p)=\frac{n}{2}-\mu(\lambda,T_pL_0)=-N=\theta_0-\theta_1+2(\alpha_1+\cdots+\alpha_n).
\end{displaymath}
\end{proof}

\begin{definition}
The angle between $L_0$ and $L_1$ at $p\in L_0\cap L_1$ is defined to be
\begin{displaymath}
\aangle{L_0}{L_1}{p}=\alpha_1+\cdots+\alpha_n,
\end{displaymath}
where $0<\alpha_i<\frac{1}{2}$ and $e^{2\pi i\alpha_1},\ldots,e^{2\pi i\alpha_n}$ are the eigenvalues of the unique unitary matrix $U$ such that $T_pL_1=U\cdot T_pL_0$, $T_pL_0$ has a basis consisting of eigenvectors of $U$, and the eigenvalues of $U$ lie in the upper half-plane (as in Lemma \ref{lemma:index_theory-linear_algebra_1}).
\end{definition}

Lemma \ref{lemma:gradings-index_formula} can be restated as
\begin{corollary}
\begin{displaymath}
\tilde\mu(L_0,L_1;p)=\theta_{L_0}(p)-\theta_{L_1}(p)+2\aangle{L_0}{L_1}{p}.
\end{displaymath}
\end{corollary}

The next lemma states that the Maslov index of a strip depends only on where the strip begins and ends.
This is very different from the non-graded case, where the Maslov index can depend also on the homotopy class of the strip.
\begin{lemma}
Let $p,q\in L_0\cap L_1$ and let $u$ be a holomorphic strip that connects $p$ to $q$.
Then 
\begin{displaymath}
\mu(u)=\tilde\mu(L_0,L_1;q)-\tilde\mu(L_0,L_1;p).
\end{displaymath}
\end{lemma}
\begin{proof}
Fix a trivialization $u^*TX\cong (\rr\times[0,1])\times \cc^n$ such that $\Omega$ maps to $1$ under the induced trivialization of $\Lambda^{n}_{\cc}(u^*TX)$.
Then all of the grading data can be viewed as referring to Lagrangian subspaces of $\cc^n$.
In particular, the bottom boundary of the strip gives a path $\lambda_0$ from $T_pL_0$ to $T_qL_0$ and the top boundary gives a path $\lambda_1$ from $T_pL_1$ to $T_qL_1$.
Let $(\gamma,\theta):[0,1]\to\Lambda(\cc^n)\times\rr$ be a path such that $\gamma(0)=T_pL_0$, $\gamma(1)=T_pL_1$, $\theta(0)=\theta_{L_0}(p)$, $\theta(1)=\theta_{L_1}(p)$, and $\Det^2(\gamma(t))=e^{2\pi i\theta(t)}$.
Then
\begin{displaymath}
\mu(u)=\mu(\lambda_0,\lambda_1)=-\mu(\gamma\#\lambda_1,\lambda_0)+\mu(\gamma,\lambda_0(0))=\tilde\mu(L_0,L_1;q)-\tilde\mu(L_0,L_1;p).
\end{displaymath}

\end{proof}

\begin{corollary}
Suppose $L_0$ and $L_1$ have constant phases.
Let $u$ be a strip that connects $p$ to $q$.
Then 
\begin{displaymath}
\mu(u)=2(\aangle{L_0}{L_1}{q}-\aangle{L_0}{L_1}{p}).
\end{displaymath}
\end{corollary}

\subsection{An absolute index: the Bott-Morse case}\label{section:gradings-bott_morse}
Now suppose that $L_0$ and $L_1$ are graded Lagrangians that intersect cleanly.
The results of the previous section can be generalized to this case.
That is, an integer $\tilde\mu(L_0,L_1;R_h)$ can be assigned to each connected component $R_h$ of $L_0\cap L_1$, as follows.
First, pick a point $p\in R_h$.
Let $g=\omega(\cdot,J\cdot)$ be the Riemannian metric on $TX$.
Let $N$ be the $g$-orthogonal complement of $T_p(L_0\cap L_1)$ in $T_pL_0+T_pL_1$.
Then $N$ is a complex and symplectic subspace of $TX$.
Let $\overline{ T_pL_0}=T_pL_0\cap N$ and $\overline{ T_pL_1}=T_pL_1\cap N$.
Let $e_1,\ldots,e_k$ be a basis of $T_p(L_0\cap L_1)$.
Then $\Omega_N=\Omega( e_1,\ldots,e_k,\cdot)$ is a form on $\Det_{\cc}(N)$, and hence $\Det^2:\Lambda(N)\to S^1$ is defined.
Moreover, if $\Lambda_0\in \Lambda(N)$, then $\Det^2(\Lambda_0+T_p(L_0\cap L_1))=\Det^2(\Lambda_0)$.
Let $(\lambda,\theta):[0,1]\to\Lambda(N)\times\rr$ be a path such that $\lambda(0)=\overline{T_pL_0}$, $\lambda(1)=\overline{T_pL_1}$, $\Det^2(\lambda(t))=e^{2\pi i\theta(t)}$, $\theta(0)=\theta_{L_0}(p)$, and $\theta(1)=\theta_{L_1}(p)$.
Let $\tilde\lambda:[0,1]\to\Lambda(T_pX)$ be the path $\tilde\lambda(t)=\lambda(t)+T_p(L_0\cap L_1)$.
\begin{definition}
The integer
\begin{displaymath}
\tilde\mu(L_0,L_1;R_h)=\frac{n-\dim R_h}{2}-\mu(\tilde\lambda,T_pL_0)
\end{displaymath}
is the absolute index between $(L_0,\theta_{L_0})$ and $(L_1,\theta_{L_1})$ at $R_h$.
(The right hand side depends continuously on the choice of $p$, and hence is independent of the choice of $p$.)
\end{definition}

The definition of the angle between $L_0$ and $L_1$ can also be extended.
Let $U$ be the unitary matrix with eigenvalues $1,\ldots,1,e^{2\pi i \alpha_1},\ldots,e^{2\pi i \alpha_{n-k}}$, $0<\alpha_i<\frac{1}{2}$ such that $U\cdot T_pL_0=T_pL_1$, $U|T_p(L_0\cap L_1)=Id$, and $T_pL_0\cap N$ has a basis consisting of eigenvalues of $U$.
\begin{definition}
The angle between $L_0$ and $L_1$ at $p\in R_h$ is defined to be 
\begin{displaymath}
\aangle{L_0}{L_1}{p}=\alpha_1+\cdots+\alpha_{n-k}.
\end{displaymath}
\end{definition}

The next two lemmas can be proved in exactly the same way as in the previous section.
\begin{lemma}\label{lemma:gradings-absolute_formula}
For any $p\in R_h$
\begin{displaymath}
\tilde\mu(L_0,L_1;R_h)=\theta_{L_0}(p)-\theta_{L_1}(p)+2\aangle{L_0}{L_1}{p}.
\end{displaymath}
\end{lemma}
\begin{corollary}
If $L_0,L_1$ have constant phase then $\aangle{L_0}{L_1}{p}$ does not depend upon the choice of $p\in R_h$.
Then, the notation $\aangle{L_0}{L_1}{R_h}$ is justified.
\end{corollary}
\begin{lemma}\label{lemma:gradings-formula}
Let $u$ be a strip that connects $R_h$ to $R_{h'}$.
Then
\begin{displaymath}
\mu(u)=\tilde\mu(L_0,L_1;R_{h'})-\tilde\mu(L_0,L_1;R_h).
\end{displaymath}
If $L_0$ and $L_1$ have constant phase, this implies that
\begin{displaymath}
\mu(u)=2\bigl(\aangle{L_0}{L_1}{R_{h'}}-\aangle{L_0}{L_1}{R_h}\bigr).
\end{displaymath}
\end{lemma}

The previous lemma can be extended to the case of polygons:
  Let $u$ be a polygon, that is $u$ is a smooth map from $D^2\setminus\sett{z_0,\ldots,z_k}$ to $X$ such that the arc between $z_i$ and $z_{i+1}$ maps to the Lagrangian $L_{i+1}$.
  Moreover, assume $u$ extends continuously to $D^2$.
  Let $p_i=u(z_i)$ and let $R_i$ be the connected component of $L_i\cap L_{i+1}$ containing $p_i$.
  \begin{lemma}
Assume that $L_0,\ldots,L_k$ have constant phase.
Then
  $$\mu^{poly}(u)={2}(\aangle{L_0}{L_1}{R_0}+\cdots+\aangle{L_k}{L_0}{R_k}).$$
  \end{lemma}
  \begin{proof}
Without loss of generality assume that $L_0$ has constant phase 1.
  Take a trivialization $u^*TX\to D^2\times\cc^n$ such that under the induced trivialization $\Det(u^*TX)\to \Det(\cc^n)=\cc$, the complex volume form $\Omega$ maps to 1.
Because the $L_i$'s have constant phase, it follows that $\rr\cdot\Det(u^*TL_0)$ maps to $\rr$ and the determinant of the Lagrangian boundary condition along the arc between $z_i$ and $z_{i+1}$ is constant for each $i$.

By the definition of the angle, there exists a unitary matrix $U_i$ such that $U_i\cdot T_{p_i}L_i=L_{i+1}$ and, in an appropriate basis for $T_{p_i}L_i$, $U_i$ is a diagonal matrix with eigenvalues $e^{2\pi i\alpha_1},\ldots,e^{2\pi i\alpha_{k_i}},1,\ldots,1$, where $\alpha_1+\cdots+\alpha_{k_i}=\aangle{L_i}{L_{i+1}}{R_i}$.
The positive definite path from $L_i$ to $L_{i+1}$ is then $U^t\cdot T_{p_i}L_i$, $0\leq t\leq1$.
  Taking $\Det^2$, it follows that the Maslov index of the path around the entire polygon is
\begin{displaymath}
\mu^{poly}(u)={2}(\aangle{L_0}{L_1}{R_0}+\cdots+\aangle{L_k}{L_0}{R_k}).
\end{displaymath}
  \end{proof}

\section{Moduli spaces}\label{section:moduli_spaces}
In this section we discuss the moduli spaces used to construct the $A_\infty$-structures on the Lagrangian submanifolds.
For more details see \cite{fooo}.
The moduli spaces come in essentially two different flavors: one, the moduli space of discs with boundary lying on a fixed Lagrangian; and two, the moduli space of strips (domains of which can be thought of as $D^2\setminus\sett{-1,1}$) with top and bottom boundaries lying on different Lagrangians.
In both cases, marked points can also be added to the boundary of the disc.
For this section let $X$ be an arbitrary symplectic manifold and $L,L_0,L_1$ be arbitrary Lagrangian submanifolds.
  \subsection{The moduli space of discs}
We start with a sketch of the moduli space of holomorphic discs $$(D^2,\partial D^2)\to(X,L)$$ with $k+1$ distinct marked boundary points $(z_0,\ldots,z_k)$ in counter-clockwise cyclic order.
  Let $W^{1,p}_{k+1}(X,L)$ denote the set of triples $((\overrightarrow{\mathbf{t}},\ell),w,\overrightarrow{p})$ such that
  \begin{eqnarray}
  \label{eq:moduli_spaces-triple_1}
  &&(\overrightarrow{\mathbf{t}},\ell)\in Gr_{k+1},\\
  \label{eq:moduli_spaces-triple_2}
  &&w:(X(\overrightarrow{\mathbf{t}},\ell),\partial X(\overrightarrow{\mathbf{t}},\ell))\to(X,L)\textrm{ is in }W^{1,p}_{loc},\\
  \label{eq:moduli_spaces-triple_3}
  &&\overrightarrow p=(p_0,\ldots,p_k)\in L^{k+1},\textrm{ and }\\
  \label{eq:moduli_spaces-triple_4}
  &&\int_{L_{e_i}}e^{\delta}|\nabla w|^pds dt+\int_{L_{e_i}}e^{\delta}\textrm{dist}(w(s,t),p_i)^pds dt<\infty\\
  &&\nonumber\textrm{ for every $e_i\in C^1_{ext}(\overrightarrow{\mathbf{t}})$, where $\delta(s)=\delta_0|s|$ and $\delta_0>0$ is small.}
  \end{eqnarray}
  Here 
\begin{equation}\label{eq:moduli_spaces-grk}
Gr_{k+1}=((\partial D^2)^{k+1}\setminus\textrm{(fat diagonal)}))/PSL(2,\rr)
\end{equation}
  denotes the moduli space of discs with $k+1$ marked boundary points removed.
  It is a manifold of dimension $k-2$.
  In \cite{fooo}, an explicit description of $Gr_{k+1}$ is given.
  Moreover, given $(\overrightarrow{\mathbf{t}},\ell)\in Gr_{k+1}$, an explicit Riemann surface $X(\overrightarrow{\mathbf{t}},\ell)$ representing the class
  $(\overrightarrow{\mathbf{t}},\ell)$ is constructed.
  Here $\overrightarrow{\mathbf{t}}$ is a tree, and each exterior edge corresponds to a marked point, and $\ell$ is a real-valued function on the set of edges of $\overrightarrow{\mathbf{t}}$.
The set of exterior edges of $\overrightarrow{\mathbf{t}}$ is denoted $C^1_{ext}(\overrightarrow{\mathbf{t}})$.
The exact construction of $X(\overrightarrow{\mathbf{t}},\ell)$ is not important for the purposes of this paper, other than the following fact:
If $D^2\setminus\sett{z_0,\ldots,z_{k}}$ is biholomorphic to $X(\overrightarrow{\mathbf{t}},\ell)$, then for each point $z_i$ an explicit biholomorphism of a neighborhood of the point is given with an open subset $L_{e_i}$ of $X(\overrightarrow{\mathbf{t}},\ell)$, and furthermore an explicit biholomorphism of $L_{e_i}$ with $(0,\infty)\times[0,1]$ is given.
The open subsets $L_{e_i}$ are called strip-like ends, and they come equipped with natural holomorphic coordinates $s+it$ given via the identification $L_{e_i}\cong(0,\infty)\times[0,1]$.

  Condition (\ref{eq:moduli_spaces-triple_4}) says that the map $w$ decays exponentially on each strip-like end.
Conditions (\ref{eq:moduli_spaces-triple_2}) and (\ref{eq:moduli_spaces-triple_4}) imply that $w$ extends continuously to the closure of $X(\vec{\bf t},\ell)$.
That is, if the domain of $w$ is thought of as the disc minus $k+1$ marked boundary points, then $w$ extends continuously to $D^2$.

  $W^{1,p}_{k+1}(X,L)$ is a Banach manifold and the tangent space at a point $((\overrightarrow{\mathbf{t}},\ell),w,\overrightarrow{p})$ is the set of all triples $(Y,\xi,\overrightarrow{V})$ such that
  \begin{itemize}
\item $Y\in T_{(\overrightarrow{\mathbf{t}},\ell)}Gr_{k+1},$
\item $\xi\textrm{ is a section of }W^{1,p}_{loc}(w^*TX),$
\item $\overrightarrow{V}=(V_0,\ldots,V_k),V_i\in T_{p_i}L,$
\item $\xi(z)\in T_{w(z)}L\textrm{ if }z\in\partial X(\overrightarrow{\mathbf{t}},\ell),$
\item $\int_{L_{e_i}}e^{\delta(s)}(|\nabla(\xi-\Par(V_i))|+|\xi-\Par(V_i)|)^pds dt<\infty$\\
 $\textrm{ for every $e_i\in C^1_{ext}(\overrightarrow{\mathbf{t}})$.}$
  \end{itemize}
Here $\Par$ denotes parallel translation in the horizontal direction ($s$ direction) over the strip like ends.

  The norm on the tangent space is
  \begin{eqnarray*}
  &&\|(Y,\xi,\overrightarrow{V})\|^p=|Y|^p+|\overrightarrow{V}|^p+\sum_{e\in C^1_{int}(\overrightarrow{\mathbf{t}})}\int_{L_e}(|\nabla \xi|^p+|\xi|^p)ds dt\\
  \nonumber&&+\sum_{e_i\in C^1_{ext}(\overrightarrow{\mathbf{t}})}\int_{ L_{e_i}}e^{\delta}(|\nabla(\xi-\Par(V_i))|+|\xi-\Par(v_i)|)^pds dt.
  \end{eqnarray*}
($Gr_{k+1}$ can be compactified in a standard way, so the precise metric used to define the first term $|Y|^p$ is not important.
See Section \ref{section:moduli_spaces-compactification} for more on the compactification of $Gr_{k+1}$.
)

  Over $W^{1,p}_{k+1}(X,L)$ is the Banach bundle $\mathcal E^{0,p}$ with fiber over the point 
$$((\overrightarrow{\mathbf{t}},\ell),w,\overrightarrow{p})\in W^{1,p}_{k+1}(X,L)$$
  the set of sections $\eta\otimes d\bar z$ of $\Lambda^{0,1}\otimes w^*TX$ such that
  \begin{eqnarray*}
  &&\eta\in W^{0,p}_{loc},\\
  &&\int_{L_{e_i}}e^{\delta}|\eta|^pds dt<\infty\textrm{ for all }{e_i}\in C^1_{ext}(\overrightarrow{\mathbf{t}}).
  \end{eqnarray*}
Notice that since $p>2$, the $L^p$ norm of $\eta\otimes d\bar z$ is not a conformal invariant.
In particular, it is important that the section is written in the form $\eta\otimes(ds-idt)$ for the integral over the strip like end to have a well-defined meaning.

  The standard Cauchy-Riemann operator gives a section
  \begin{displaymath}
  \dbar:W^{1,p}_{k+1}(X,L)\to\mathcal E^{0,p}.
  \end{displaymath}
As a topological space, the zero set of $\dbar$ is the moduli space of holomorphic discs with $k+1$ marked points; it will be denoted as 
\begin{displaymath}
{\mathcal M}_{k+1}^{reg}(L)=\dbar^{-1}(0).
\end{displaymath}
Note that $\dbar^{-1}(0)$ does not have to be modded out by $PSL(2,\rr)$ (the automorphism group of the $D^2$), because the space $Gr_{k+1}$ has already been modded out by $PSL(2,\rr)$, see (\ref{eq:moduli_spaces-grk}).

However, if $k\leq 1$ then $Gr_{k+1}$ is a space of negative dimension, and the setup does not make sense.
In this case, the space $Gr_{k+1}$ is excluded from Banach space setup and the maps $w$ are all thought of as being defined on $D^2$ minus $k+1$ marked boundary points.
Then the zero set of $\dbar$ is modded out by $PSL(2,\rr)$ to get the moduli spaces $\mathcal M_{k+1}^{reg}(L)$.

If $k=1$, then the moduli spaces $\mathcal M_2^{reg}(L)$ can be constructed in another way.
Again, exclude $Gr_2$ from the Banach space setup.
Then, since $D^2$ minus 2 points is biholomorphic to the strip $\rr\times[0,1]$, think of the maps $w$ as being defined on the strip.
The marked points correspond to $(\pm\infty,\cdot)$.
Then mod out the zero set of $\dbar$ by $\rr=\textrm{Aut}(\rr\times[0,1])$ to get $\mathcal M_2^{reg}(L)$.

Elements of $\mathcal M_{k+1}^{reg}(L)$ will be denoted as $[w,(z_0,\ldots,z_{k})]$.
For $\beta\in\pi_2(X,L)$, let $\mathcal M_{k+1}^{reg}(L,\beta)$ deonte the moduli space of discs of homotopy class $\beta$.
For $E$ a real number let $\mathcal M_{k+1}^{reg}(L,E)$ denote the moduli space of discs of energy $E$.

  Linearizing $\dbar$ (using the Levi-Civita connection) gives the linearized Cauchy-Rie\-mann operator
  \begin{equation}\label{eq:linearizeddolbeault}
  D_{((\overrightarrow{\mathbf{t}},\ell),w,\overrightarrow{p})}\dbar:T_{((\overrightarrow{\mathbf{t}},\ell),w,\overrightarrow{p})}W_{k+1}^{1,p}(X,L)\to\mathcal E^{0,p}_{((\overrightarrow{\mathbf{t}},\ell),w,\overrightarrow{p})}.
  \end{equation}
If $k=1$, and the domains of the maps $w$ are thought of as strips, with the marked points being $(\pm\infty,\cdot)$, then the linearized operator simplifies to
\begin{displaymath}
D_w\dbar:T_wW_2^{1,p}(X,L)\cong W^{1,p;\delta}_\lambda(S,w^*TX)\oplus T_pL\oplus T_qL\to\mathcal E^{0,p}_w\cong L^{p;\delta}(S,w^*TX),
\end{displaymath}
where $p=w(-\infty,\cdot)$ and $q=w(+\infty,\cdot)$.
Note that the kernel of $D_w\dbar$ will always have a one-dimensional subspace corresponding to the $\rr$-action coming from $\textrm{Aut}(S)=\rr$.

If the linearized operator is surjective, the implicit function theorem implies that $\mathcal M_k^{reg}(L)$ is a smooth manifold with tangent space equal to the kernel of $D_w\dbar$.
If the operator is not surjective, then the best that can be said is that $\mathcal M_{k+1}^{reg}(L)$ is a Kuranishi space.

By Gromov's compactness theorem, the moduli spaces $\mathcal M_{k+1}^{reg}(L,E)$ can be compactified to get compact topological spaces ${\mathcal M_{k+1}(L,E)}$.
Let $\mathcal M_{k+1}(L)$ be the union of all the compactified moduli spaces $\mathcal M_{k+1}(L,E)$.
The compactification will be discussed in Section \ref{section:moduli_spaces-compactification}.

%  If $z=s+it$ are holomorphic coordinates on $\Xdomain$ and $(\xi,V,\overrightarrow{v})$ is a tangent vector, then the formula for $D\dbar$ is
 % \begin{equation}\label{eq:linearizedformula}
 % \nabla^{0,1}(\xi,V,\overrightarrow{v})=d\bar z\otimes(\nabla_s+J(w)\nabla_t)(V+\sum \chi_iv_i),
 % \end{equation}
%where $\chi_i$ are appropriate cutoff functions.
  \begin{prop}[\cite{fooo} Proposition 10.2]
  $ {\mathcal M_{k+1}(L,\beta)}$ is a Kuranishi space of virtual dimension $\mu_L(\beta)+k-2+n$.
  \end{prop}

%Let $\overrightarrow{P}=(P_1,\ldots,P_k)$ be a tuple of chains in $L$, and let $\mathcal M_{k+1}(\beta,\overrightarrow{P})$ denote the moduli space of discs such that the $i^{th}$ marked point goes through the chain $P_i$, for $i=1,\ldots,k$.
%This moduli space can be define in a manner similar to discs.
%Note that the zeroeth marked point $z_0$ is free.
The moduli spaces ${\mathcal M_{k+1}}(L)$ come equipped with natural evaluation maps $$ev_i:\mathcal M_{k+1}(L)\to L,\ [w,(z_0,\ldots,z_k)]\mapsto w(z_i)$$ for $i=0,\ldots,k$.

  \subsection{The moduli space of strips}\label{section:moduli_spaces-bott_morse}
We now turn to the moduli space of strips.
  Let $L_0,L_1$ be two Lagrangians that intersect cleanly (the case $L_0=L_1$ is allowed, so this is a generalization of the previous section).
  Let $W^{1,p}_{k_0,k_1}(L_0,L_1)$ denote the space of maps in the Bott-Morse setting.
  More specifically, maps are quadruples $\ptmapsbm$ with $k=k_0+k_1+2$ such that
\begin{itemize}
\item the domain $X(\overrightarrow{\bf t,\ell})$ of $w$ is biholomorphic to a disc with $k_0+k_1+2$ marked boundary points removed,
\item two of the marked points are $1$ and $-1$, which will also be denoted as $(\pm\infty,\cdot)$ or simply $\pm\infty$.
\item $X(\overrightarrow{\bf t,\ell})$ is given by an explicit model with strip-like ends,
\item the bottom boundary of $w$ lies on $L_0$,
\item the top boundary of $w$ lies on $L_1$,
\item $w(\pm1)\in L_0\cap L_1,$
\item $k_0$ marked points are on the bottom arc of the disc between $-1$ and $1$,
\item $k_1$ marked points are on the top arc of the disc between $-1$ and $1$,
\item $w\in W^{1,p}_{loc}$,
\item $w$ decays exponentially as before,
\item $\overrightarrow{p_0}=(p^{(0)}_1,\ldots,p^{(0)}_{k_0})\in L_0^{k_0}$ and $w$ maps the $k_0$ marked points to $\overrightarrow{p_0}$,
\item $\overrightarrow{p_1}=(p^{(1)}_1,\ldots,p^{(1)}_{k_1})\in L_1^{k_1}$ and $w$ maps the $k_1$ marked points to $\overrightarrow{p_1}$.
\end{itemize}

If $k_0=k_1=0$, then the parameter $(\vec{\bf t},\ell)$ is excluded from the setup and the domain of $w$ is thought of as the strip $\rr\times[0,1]$ (this is similar to what was done in the previous section with 2 marked points).
For the purposes of defining Floer cohomology, $k_0$ and $k_1$ are taken to be $0$, so the exact description of $X(\overrightarrow{\bf t,\ell})$ is not needed in this paper, other than the fact mentioned in the previous section.

The exponential decay over the strip like ends implies that $w$ extends continuously to the closure of $X(\vec{\bf t},\ell)$.
In particular, this implies that the value of $w$ at the infinity end of a strip like end does not depend on the $t$ parameter.

  The tangent space to $W^{1,p}_{k_0,k_1}$ is defined in the same way as before, except $w(\pm\infty,\cdot)$ are only allowed to vary in $L_0\cap L_1$.

As in the previous section, sitting over the Banach manifold $W^{1,p}_{k_0,k_1}(L_0,L_1)$ is the Banach bundle $\mathcal E^{0,p}$ whose fiber over the point $((\vec{\bf t},\ell),w,\vec p_0,\vec p_1)$ is
\begin{displaymath}
\mathcal E_w^{0,p}=L^{p;\delta}(X(\vec{\bf t},\ell),\Lambda^{0,1}\otimes w^*TX).
\end{displaymath}
Again, the Cauchy-Riemann operator $\dbar$ gives a section of this bundle, and the zero set is the moduli space of holomorphic strips:
\begin{displaymath}
\mathcal M_{k_0,k_1}^{reg}(L_0,L_1)=\dbar^{-1}(0).
\end{displaymath}
If $k_0=k_1=0$, then the domain of the maps is the strip $\rr\times[0,1]$ with no additional marked points, so $\dbar^{-1}(0)$ has to be modded out by $\rr$ to get the moduli space:
\begin{displaymath}
\mathcal M_{0,0}^{reg}(L_0,L_1)=\dbar^{-1}(0)/\rr.
\end{displaymath}
This moduli space will often be denoted by $\mathcal M^{reg}(L_0,L_1)$ as well.

Let
\begin{displaymath}
\mathcal M_{k_0,k_1}^{reg}(L_0,L_1:E)
\end{displaymath}
denote the moduli space of maps that have energy $E$.
Similarly, let 
\begin{displaymath}
\mathcal M_{k_0,k_1}^{reg}(L_0,L_1:\beta)
\end{displaymath}
denote the moduli space of strips of homotopy class $\beta$.

By Gromov's compactness theorem, the moduli spaces $\mathcal M_{k_0,k_1}^{reg}(L_0,L_1:E)$ can be compactified to get compact topological spaces
\begin{displaymath}
\mathcal M_{k_0,k_1}(L_0,L_1:E).
\end{displaymath}
Let $\mathcal M_{k_0,k_1}(L_0,L_1)$ denote the union of these compactified moduli spaces.
See Section \ref{section:moduli_spaces-compactification} for the description of the compactification.
The compactified moduli spaces are Kuranishi spaces.

For components $R_h$ and $R_{h'}$ of $L_0\cap L_1$, let 
\begin{displaymath}
\mathcal M_{k_0,k_1}(L_0,L_1:R_h,R_{h'})
\end{displaymath}
denote the subset of the (compactified) moduli space of strips that start in the component $R_h$ and end in the component $R_{h'}$.
That is, $w(-\infty,\cdot)\in R_h$ and $w(+\infty,\cdot)\in R_{h'}$.

  \begin{prop}[\cite{fooo} Proposition 12.59]\label{prop:moduli_spaces-dimension}
  $\mathcal M_{k_0,k_1}(L_0,L_1:R_h,R_{h'}:\beta)$ is a Kuranishi space of virtual dimension
  $$\dim S+\dim R_{h'}-\dim R_h+\mu(\beta)+k_0+k_1-1.$$
  \end{prop}
%  denote the moduli space of holomorphic maps with $w\#u=w'$.
%  Let
%  $$\modstripschainsbm$$
%  denote the moduli space of holomorphic maps such that the $k_0$ marked points pass through the chains $\overrightarrow P_0$, the $k_1$ marked points pass through the chains $\overrightarrow P_1$, and $-\infty$ passes through the chain $S$.
%  This moduli space is constructed as a fiber product.
%  $+\infty$ should be though of as a free end, and under the evaluation map it will trace out a chain.
%  \begin{prop}
%  $\modstripschainsbm$ is a Kuranishi space with virtual dimension
%  $$\dim S+d_{h'}-d_d+\mu([h',w'])-\mu([h,w])-\sum\deg P_j^0-\sum\deg P_j^1+k_0+k_1-1.$$
%  \end{prop}
%
%  The moduli space
%  $$\modstripschainsbmzero$$
%  is the relevant one for defining Floer cohomology.

\subsection{Kuranishi structure on $\mathcal M^{reg}$}\label{section:moduli_spaces-kuranishi_structure_interior}
In this section briefly describe how a Kuranishi structure is put on the interior of the moduli spaces.
See Appendix A for the definition of a Kuranishi structure as well as an explanation of the notation.

Let $w=[w,(z_0,\ldots,z_k)]\in\mathcal M_{k+1}^{reg}(L)$, and consider the linearized $\dbar$ operator
\begin{displaymath}
D_w\dbar:T_wW^{1,p}_{k+1}(X,L)\to \mathcal E^{0,p}_{w}.
\end{displaymath}
If $D_w\dbar$ is surjective, then by the implicit function theorem there exists a neighborhood $V$ of $w$ in $\dbar^{-1}(0)$ such that $V$ is a manifold with $\dim V=\Ind(D_w\dbar)$.
The Kuranishi chart containing $[w,(z_0,\ldots,z_k)]$ is then defined to be
\begin{displaymath}
(V,0,1,i,0).
\end{displaymath}
That is, the obstruction bundle $E$ is $0$, the group $\Gamma$ is the trivial group $1$, the map $\psi$ is the inclusion $i:V\to\dbar^{-1}(0)=\mathcal M_{k+1}^{reg}(L)$, and the Kuranishi map $s$ is $0$.

If $D_w\dbar$ is not surjective, let $E_w$ be a finite dimensional complex subspace of $\mathcal E^{0,p}_w$ that is a complement of the image of $D_w\dbar$.
Parallel translate $E_w$ to maps near $w$ to get a finite dimensional vector subbundle $E$ of $\mathcal E^{0,p}$ that is defined in a neighborhood of $w$.
Then by the implicit function theorem again, there exists a neighborhood $V$ of $w$ in $\dbar^{-1}(E)$ such that $V$ is a manifold with
\begin{displaymath}
\dim V=\Ind(D_w\dbar)+\dim_{\rr}E_w.
\end{displaymath}
Note that $\dbar$ restricted to $V$ is a section of $E$.
The Kuranishi chart near $w$ is then taken to be 
\begin{displaymath}
(V,E,1,i,\dbar|V),
\end{displaymath}
where $i:(\dbar|V)^{-1}(0)\to \dbar^{-1}(0)=\mathcal M_{k+1}^{reg}(L)$ is the inclusion map.

The construction of Kuranishi charts on $\mathcal M_{k_0,k_1}^{reg}(L_0,L_1)$ is similar.

\subsection{Compactification of the moduli spaces}\label{section:moduli_spaces-compactification}
First, note that the moduli space of discs $Gr_{k+1}$ can be compactified in a standard way to get a compact topological space $\overline{Gr_{k+1}}$.
The source of non-compactness in $Gr_{k+1}$ is that marked boundary points can collide.
To account for this, collision of marked points result in discs bubbling off, with the colliding marked points becoming marked points on the bubbled disc.
Briefly, then, the elements of $\overline{Gr_{k+1}}$ are stable curves with marked points, where all components of the curve are discs.
(Stable means that each disc component has at least three special points.)

By Gromov's compactness theorem, the spaces $\mathcal M_{k+1}^{reg}(L:E)$ can be compactified by adding in all stable maps.
That is, add in holomorphic maps defined on semi-stable curves, with possible sphere components added in as well.
(The difference between a semi-stable curve and a stable curve is that semi-stable curves do not need to have three special points on each component; however, to get a stable map with such a domain, the map is required to be non-constant on each component with less than three marked points.)
The ``boundary'' of $\mathcal M_{k+1}(L:E)$ (that is, the set of stable maps with more than one domain component) can be realized as the union of fiber products of lower energy discs and spheres with marked points whose total energy is $E$.

One subtle point is that it is possible for a point to be in the boundary but for there not to exist a sequence $w_n\in\mathcal M^{reg}_{k+1}(L)$ that converges to the boundary point.
This is because by definition the boundary includes all stable maps, even if they are not limit points of elements in $\mathcal M^{reg}_{k+1}(L)$.

The compactification of $\mathcal M_{k_0,k_1}^{reg}(L_0,L_1)$ is similar.

\subsection{Kuranishi structure on the boundary of the moduli space}

Let $$w=[w,(z_0,\ldots,z_k)]$$ be an element in the boundary of $\mathcal M_{k+1}(L,E)$.
Then the domain of $w$ is a semi-stable disc with more than one component, hence $w$ can be realized as an element in a (topological) fiber product of lower energy moduli spaces whose total energy is $E$.
The fiber product of the Kuranishi structures on the moduli spaces in this fiber product define a Kuranishi neighborhood of $w$ in a boundary stratum of $\mathcal M_{k+1}(L,E)$ containing $w$.
By resolving the nodal points of the domain and using the gluing construction, this Kuranishi neighborhood can be enlarged to an open Kuranishi neighborhood of $w$.
The gluing of maps involves parallel translation of obstruction bundles with the Levi-Civita connection and solving the $\dbar$ equation modulo the translated obstruction bundle.
The details of this analysis are technical and difficult; see \cite{fooo} for more information.

\subsection{Technical remarks on perturbation of moduli spaces}
Let $L_0$ and $L_1$ be two real Lagrangians.
Let $u$ be a holomorphic strip with boundary on $L_0$ and $L_1$.
By the discussion in Section \ref{section:strips-spheres}, $u$ can be extended to a holomorphic sphere and hence is given by a polynomial; let $k$ be the degree of the polynomial.
By Lemma \ref{lemma:strips-energy}, the energy $E(u)$ of $u$ is $kE_0/10$, where $E_0$ is the energy of a degree $1$ holomorphic sphere in $X$.
Let $\mathcal M(L_1,L_2:kE_0/10)$ denote the moduli space of such $u$.
A holomorphic disc with boundary on $L_i$ has energy $kE_0/2$, where again $k$ is the degree of the polynomial given by the map.
Let $\mathcal M(L_i:kE_0/2)$ denote the moduli space of such discs.

The construction of Kuranishi structures on the moduli spaces needs to be done in a particular order to ensure that the boundary of a particular moduli spaces consists of fiber products of the lower energy moduli spaces.
An over-simplified description of the process is as follows:
First, Kuranishi strutures are constructed for $\mathcal M(L_i:kE_0/2)$, starting with $k=1$ and then proceeding to progressively higher values of $k$.
The moduli spaces for $i=1,2$ are independent of each other.
Then, the Kuranishi structures on the moduli spaces $\mathcal M(L_1,L_2:kE_0/10)$ are constructed, starting with $k=1$ and working up.

Actually, to ensure that compatibility of the Kuranishi structures at the boundary is maintained, this can only be done for all moduli spaces of energy less than some fixed constant $E$. 
Replacing $E$ with a larger value potentially requires re-perturbing all moduli spaces.
Since not all the moduli spaces can be coherently perturbed at once, an intermediate structure needs to be defined to bridge the gap to an $A_\infty$ structure.
The intermediate structure is an $A_{N,K}$ structure, and it is a kind of truncated version of an $A_\infty$ structure, where higher order operations are ignored after a certain point.
Increasing $E$ gives progressively more accurate $A_{N,K}$ structures, and passing to the limit gives the $A_\infty$ structure.

Strictly speaking, the intermediate $A_{N,K}$ structures are ignored in this paper.
The justification for this is as follows:
\begin{enumerate}
\item In the cases where the Floer cohomology is calculated solely from degree considerations, the intermediate structure is clearly irrelevant.
\item In the cases where the Floer cohomology is calculated by using the sign of the involution of the moduli space (essentially to say that the Floer differential is $0$), the same reasoning applies to each $A_{N,K}$ structure to show that the differential is $0$, hence it must also be $0$ in the limit.
\item In the cases where the Floer cohomology is calculated by explicitly computing the lowest energy moduli spaces, note that the lowest energy moduli spaces are always closed manifolds (and different perturbations yield homologous manifolds).
Since the results are then applied to determine the differential on the $E_2$ term of the spectral sequence, which consists of the singular homology groups, the result is independent of the actual perturbation chosen.
Thus the induced map on the $E_2$ term is the same regardless of the $A_{N,K}$ structure, and hence determines the map induced by the $A_\infty$ structure.
See \cite{toricfooo} Section 11 for similar arguments along this line, in particular Lemma 11.7 and Remark 11.11.
\end{enumerate}

\section{$A_\infty$-structures}\label{section:a_infinity}
The main result of \cite{fooo} is the assocation to each Lagrangian $L$ an $A_\infty$-algebra $$(CF(L),m)$$ and to each pair of Lagrangians $(L_0,L_1)$ a $(CF(L_0),CF(L_1))$ $A_\infty$-bimodule $$(CF(L_0,L_1),n).$$
%The Floer cohomology $HF(L)$ of $L$ is by definition the cohomology of $CF(L)$, and the Floer cohomology $HF(L_0,L_1)$ of the pair $(L_0,L_1)$ is by definition the cohomology of the bimodule $CF(L_0,L_1)$ (if these cohomologies exist, see below).
We begin this section by reviewing the algebraic framework of $A_\infty$-structures, and then we turn to the construction of these objects for Lagrangian submanifolds.
The material is taken from \cite{fooo}.
Other good references are \cite{fukaya1} and \cite{seidel2}.

 \subsection{$A_\infty$-algebras and -bimodules}\label{section:a_infinity-algebras_bimodules}
 The universal Novikov ring is
 $$\Lambda_{nov}=\set{\sum_{i=1}^\infty r_iT^{\lambda_i}}{r_i\in \qq,\lambda_i\in\rr,\lambda_i\rightarrow\infty}.$$
$\Lambda_{nov}$ is a field.
We will also need the local ring
 $$\Lambda_{0,nov}=\set{\sum r_iT^{\lambda_i}\in\Lambda_{nov}}{\lambda_i\geq 0}.$$
$\Lambda_{0,nov}$ is a subring of $\Lambda_{nov}$.
The maximal ideal of $\Lambda_{0,nov}$ is denoted by $\Lambda_{0,nov}^+$ and consists of all elements such that the exponent of each energy term $T^{\lambda_i}$ is positive.

 \begin{definition}
 A filtered $A_\infty$-algebra $(A,m)$ consists of a free $\z$-graded $\Lambda_{0,nov}$-module $A=\oplus A^k$, and maps $m_k:A^{\otimes k}\to A$ of degree $2-k$, for all $k\geq 0$, that satisfy the $A_\infty$ relations
 \begin{eqnarray}\label{eq:ainftyrelations}
 \sum_{k_1+k_2=k+1}\sum_i(-1)^{deg x_1+\cdots+deg x_{i-1}+i-1}m_{k_1}(x_1,\ldots,\\
 \nonumber m_{k_2}(x_i,\ldots,x_{i+k_2-1}),\ldots,x_k)=0.
 \end{eqnarray}
Moreover, it is required that $m_0(1)\equiv 0$ mod $\Lambda_{0,nov}^+$.
 \end{definition}

  The first few relations are
 \begin{itemize}
 \item $m_1(m_0(1))=0$,
 \item $m_2(m_0(1),x)+(-1)^{deg x+1}m_2(x,m_0(1))+m_1(m_1(x))=0$,
 \item $m_3(m_0(1),x,y)+(-1)^{deg x+1}m_3(x,m_0(1),y)+$\\
 $(-1)^{deg x+deg y+2}m_3(x,y,m_0(1))+m_1(m_2(x,y))+$\\
 $m_2(m_1(x),y)(-1)^{deg x+1}+m_2(x,m_1(y))=0$.
 \end{itemize}
 In particular, if $m_0(1)=0$ then $m_1^2=0$.
In this case, $A$ is a cochain complex with differential $m_1$.
Thus the cohomology of $A$ exists, and $m_2$ gives an associative product on it.

  In general $m_0(1)\neq 0$ and thus $m_1^2\neq 0$.
The operator $m_1$ needs to be deformed to get a chain complex.
This can be done by finding solutions of the Maurer-Cartan equation (also known as bounding cochains): elements $b\in A^1$ (and $b\equiv0$ mod $\Lambda_{0,nov}^+$) such that $\sum_{k\geq0}m_k(b,\ldots,b)=0$.
Then define the deformed operator $m_1^b$ by
\begin{displaymath}
m_1^b(x)=\sum m_k(b,\ldots,b,x,b,\ldots,b).
\end{displaymath}
The sum is over all possible expressions of the given form.
If $b$ is a bounding cochain, it can be shown that $(m_1^b)^2=0$.
Thus the cohomology of $(A,m_1^b)$ exists.
If a bounding cochain exists, $A$ is called unobstructed.
It turns out that the $A_\infty$-algebras associated to the real Lagrangians are unobstructed, and moreover that $b=0$ is a bounding cochain.

  The next concept is that of an $A_\infty$-bimodule.
  \begin{definition} Let $(A_0,m^0)$ and $(A_1,m^1)$ be $A_\infty$-algebras.
  An $(A_0,A_1)$ $A_\infty$-bi\-module $(D,n)$ consists of a free filtered $\z$-graded $\Lambda_{0,nov}$-module $D=\oplus D^k$ and filtered maps
  $$n_{k_0,k_1}:(A_0[1])^{\otimes k_0}\otimes_{\Lambda_{0,nov}}\otimes D[1] \otimes_{\Lambda_{0,nov}} (A_1[1])^{\otimes k_1}\to D[1]$$
  of degree 1 for all $k_0,k_1\geq0$ that satisfy the bimodule relations
  \begin{eqnarray*}
 0&=& \sum_{i_0,i_1} (-1)^{\epsilon}n_{k_0-i_0-1,k_1-i_1-1}(a^{0}_1,\ldots,a^{0}_{k_0-i_0-1},\\
&&n_{i_0,i_1}(a^{0}_{k_0-i_0},\ldots,a^{0}_{k_0},d,a^{1}_1,\ldots,a^1_{i_1}),a^1_{i_1+1},\ldots,a_{k_1})\\
&&+\sum_{i_0,j} (-1)^{\epsilon_0}n_{k_0-j,k_1}(a^{0}_1,\ldots,a^{0}_{k_0-i_0-1},m^0_{j+1}(a^{0}_{k_0-i_0},\ldots,a^0_{k_0-i_0+j}),\\
&&a^0_{k_0-i_0+j+1},\ldots,a^0_{k_0},d,a^1_1,\ldots,a^1_{k_1})\\
&&+\sum_{i_1,j} (-1)^{\epsilon_1}n_{k_0,k_1-j}(a^{0}_1,\ldots,a^{0}_{k_0},d,a^1_1,\ldots,a^1_{i_1-1},m^1_{j+1}(a^1_{i_1},\ldots,a^1_{i_1+j}),\\
&&a^1_{i_1+j+1},\ldots,a^1_{k_1}).
  \end{eqnarray*}
Here
\begin{eqnarray*}
\epsilon&=&\deg(a^0_1)+\cdots+\deg(a^0_{k_0-i_0-1})+k_0-i_0-1,\\
\epsilon_0&=&\deg(a^0_1)+\cdots+\deg(a^0_{k_0-i_0-1})+k_0-i_0-1,\\
\epsilon_1&=&\deg(a^0_1)+\cdots+\deg(a^0_{k_0})+k_0.
\end{eqnarray*}
$A_i[1]$ denotes the graded algebra with $A_i[1]^k=A_i^{k+1}$.
  \end{definition}

  In particular, the first relation is
  \begin{equation}\label{eq:bimoduledifferential}
  n_{0,0}^2(x)+n_{1,0}(m_0^0(1),x)+(-1)^{\deg(x)+1}n_{0,1}(x,m_0^1(1))=0.
  \end{equation}
  Thus if $m_0^1=m_0^0=0$ then $n_{0,0}^2=0$, and $D$ is a cochain complex with differential $n_{0,0}$, and the cohomology of $D$ exists.

If $b_0$ and $b_1$ are bounding cochains of $A_0$ and $A_1$, then deformed operations $n^{(b_0,b_1)}$ on $D$ can be defined such that $D$ becomes a bimodule with respect to the deformed algebras $(A_0,m^{b_0})$ and $(A_1,m^{b_1})$.
The deformed operator then satisfies $(n_{0,0}^{(b_0,b_1)})^2=0$, so the cohomology of $D$ exists.

  \subsection{The $A_\infty$-algebra associated to a Lagrangian}
In this section we review the construction of the filtered $A_\infty$-algebra $CF(L)$ associated to a graded Lagrangian $L$.
Assume that $L$ has a spin structure (or more generally a relative spin structure, see \cite{fooo}), and fix a choice of a spin structure on $L$.
The spin structure is needed to orient the moduli spaces, so the operations with $\qq$-coefficients can be defined.
See Section \ref{section:local_systems} for more on these issues.

First, it is a fact that there exists a countable set $\mathcal X_L$ of smooth singular chains with $\mathbb Q$ coefficients such that $\mathcal X_L$ is an (unfiltered) $A_\infty$-algebra generating the singular cohomology.
  Cohomological notation will be used, so if $f:P\to L$ is a chain in $\mathcal X_L$ with $\dim(P)=k$, then 
\begin{equation}\label{eq:a_infinity-grading_1}
\deg([P,f])=\dim(L)-k.
\end{equation}

For chains $[P_i,f_i]$ in $\mathcal X_L$, let $\overrightarrow P=P_1\times\cdots\times P_k$ and let
\begin{displaymath}
\mathcal M_{k+1}(L,\beta:\overrightarrow P)=\mathcal M_{k+1}(L,\beta)\times_{(ev_1\times\cdots\times ev_k)\times (f_1\times\cdots\times f_k)}(P_1\times\cdots\times P_k).
\end{displaymath}
(Notice that the zeroeth marked point is not used in the fiber product, i.e. it is free.)
Then, there exists a countable set of chains $\mathcal X_1(L)\subset \mathcal X_L$ and a system of multisections $s_{\beta,\vec P}$ such that
  \begin{itemize}
  \item $\vd(\mathcal M_{k+1}(L,\beta:\vec P))=n-2-\sum(g_i-1)\textrm{, where }\deg P_i=g_i,$
  \item $s_{\beta,\vec P}$ is transverse to the zero section, and
  \item $(s_{\beta,\vec p}^{-1}(0),ev_0)\textrm{ lies in }\mathcal X_1(L).$
  \end{itemize}
Intuitively, this means that the virtual fundamental classes of all chains, moduli spaces, and fiber products lie in the set $\mathcal X_1(L)$.

  Let $C^k(L;\qq)$ be the $\qq$ span of the degree $k$ chains in $\mathcal X_1(L)$.
  Let $CF^k(L)$ be the completion of $C^k(L;\qq)\otimes\Lambda_{0,nov}$.
  Define $m_k$ by
  $$m_k=\sum_{\beta\in\pi_2(X,L)} m_{k,\beta}\otimes T^{\omega(\beta)},$$
where

  \begin{equation}\label{eq:multi0}
  m_{0,\beta}(1)=\left\{
  \begin{array}{ll}
  ev_{0*}[\mathcal M_1(L,\beta)] & \beta\neq0\\
  0 & \beta=0
  \end{array}
  \right.,
  \end{equation}

  \begin{equation}\label{eq:multi1}
  m_{1,\beta}(P,f)=\left\{
  \begin{array}{ll}
  ev_{0*}[\mathcal M_2(L,\beta,P)] & \beta\neq0\\
  (-1)^n\partial P & \beta=0
  \end{array}
  \right.,
  \end{equation}

  \begin{equation}\label{eq:multik}
  m_{k,\beta}(P_1,\ldots,P_k)=ev_{0*}[\mathcal M_{k+1}(L,\beta,\vec P)] \textrm{ for $k\geq2$.}
  \end{equation}
(Here, for a Kuranishi space $\mathcal M$ and strongly continuous map $f:\mathcal M\to X$, $f_*[\mathcal M]$ denotes the virtual fundamental class of $\mathcal M$, which is a singular chain in $X$.
See the appendix for details.)

With these operations,  $(CF(L),m)$ is the $A_\infty$ algebra associated to the Lagrangian $L$.

  \subsection{The bimodule associated to a Lagrangian pair}\label{section:a_infinity-floer_cohomology}
In this section we review the $A_\infty$-bimodule $CF(L_0,L_1)$ associated to a pair of cleanly intersecting graded Lagrangians $(L_0,L_1)$.
Assume that $L_0$ and $L_1$ both have fixed spin structures (see Section \ref{section:local_systems} for more on this issue) and fixed gradings.

For each connected component $R_h$ of $L_0\cap L_1$, let $C(R_h)$ denote a suitable subcomplex of the singular chains on $R_h$, with coefficients in a local system $\Det(TR_h)\otimes\Theta^-_{R_h}$ and coefficient ring $\Lambda_{0,nov}$.
(The construction of the suitable subcomplex is similar to the construction of $\mathcal X_1(L)$ in the previous section.)
The local system will be described in more detail in Section \ref{section:local_systems}.
For now, it suffices to note that the fiber products that will be used to define the bimodule operations can be taken so that they define chains in the local system, see Section \ref{section:local_systems-definition}.

The grading on $CF(L_0,L_1)$ is
\begin{equation}\label{eq:a_infinity-grading_2}
CF^k(L_0,L_1)=\oplus_{R_h}C^{k-\tilde\mu(L_0,L_1;R_h)}(R_h).
\end{equation}
Let $\vec P^{0}=(P^0_1,\ldots,P^0_{k_0})$ and $\vec P^{1}=(P^1_1,\ldots,P^1_{k_1})$ be tuples of chains in $L_0$ and $L_1$ and let $S$ be a chain in (the local system on) $R_h$.
Let $$\mathcal M_{k_0,k_1}(L_0,L_1:R_h,R_{h'}:\vec P^{0},S,\vec P^{1})$$ be the fiber product
\begin{equation}\label{eq:a_infinity-bimodule_fiber_product}
\mathcal M_{k_0,k_1}(L_0,L_1:R_h,R_{h'})\times_{X}(P^0_1\times\cdots\times P^0_{k_0}\times S\times P^1_1\times\cdots\times P^1_{k_1}).
\end{equation}
It is the moduli space of holomorphic discs with $k_0+k_1+2$ marked points passing through the given chains and connecting $R_h$ to $R_{h'}$.
Note that the $+\infty$ marked point is free, hence there is an evaluation map
\begin{displaymath}
ev_{+\infty}:\mathcal M_{k_0,k_1}(L_0,L_1:R_h,R_{h'}:\vec P^0,S,\vec P^1)\to R_{h'}.
\end{displaymath}

%  Then
%  \begin{eqnarray}
%  \vd(\mathcal M_{k_1,k_0,(+\infty)}(L_1,L_0:[h,w],[h',w']:\vec P^{(1)},S,\vec P^{(0)}))=\\
%  \nonumber\dim S+d_{h'}-d_h+\mu([h',w'])-\mu([h,w])-\sum\deg P+k_0+k_1-1.
%  \end{eqnarray}
%  This gives a chain on the local system on $R_{h'}$.
The operators $n_{k_0,k_1}$ are defined by
  \begin{eqnarray}\label{eq:a_infinity-bimodule_operation}
&  n_{k_0,k_1}(T^{\lambda_1^0}P_1^0\otimes\cdots \otimes T^{\lambda_{k_0}^0}P_{k_0}^{0}\otimes T^\lambda S\otimes T^{\lambda_1^1}P_1^1\otimes\cdots\otimes T^{\lambda_{k_1}^1}P_{k_1}^1)\\
  \nonumber&=\\
&\nonumber\sum_{R_{h'},\beta}T^{\lambda'}ev_{+\infty*}[\mathcal M_{k_0,k_1}(L_0,L_1:R_h,R_{h'}:\vec P^{0},S,\vec P^{1}:\beta)],
  \end{eqnarray}
where $\lambda'=\sum \lambda_i^0+\lambda+\sum\lambda_i^1+\omega(\beta)$.

The grading convention here is different than that in \cite{fooo} (i.e. graded Lagrangians are used in lieu of a grading parameter in the Novikov ring), so it should be checked that $n_{k_0,k_1}$ is a map of degree $1-k_0-k_1$:
\begin{lemma}
$n_{k_0,k_1}$ is a map of degree $1-k_0-k_1$.
\end{lemma}
\begin{proof}
Let the input of $n_{k_0,k_1}$ be as in (\ref{eq:a_infinity-bimodule_operation}).
Using the grading conventions given in (\ref{eq:a_infinity-grading_1}) and (\ref{eq:a_infinity-grading_2}), the degree of the input is 
\begin{displaymath}
k_0n-\sum\dim(P_i^0)+k_1n-\sum\dim(P_i^1)+\dim(R_h)-\dim(S)+\tilde\mu(L_0,L_1;R_{h}).
\end{displaymath}
The output chain $Q$ can be written as $Q=\sum_{R_{h'}}Q_{h'}$.
The dimension of $Q_{h'}$ is the dimension of the fiber product (\ref{eq:a_infinity-bimodule_operation}), and by Proposition \ref{prop:moduli_spaces-dimension} and Lemma \ref{lemma:gradings-formula} this is
\begin{eqnarray*}
&&\sum\dim(P_i^0)+\sum\dim(P_i^1)+\dim(S)+k_0+k_1+2+\tilde\mu(L_0,L_1;R_{h'})-\\&&\tilde\mu(L_0,L_1;R_h)+\dim(R_{h'})-\dim R_h-k_0n-k_1n-3.
\end{eqnarray*}
Therefore the degree of $Q_{h'}$ is 
\begin{eqnarray*}
&&\tilde\mu(L_0,L_1;R_{h'})+\dim(R_{h'})-(\sum\dim(P_i^0)+\sum\dim(P_i^1)+\dim(S)+k_0+k_1\\&&+2+\tilde\mu(L_0,L_1;R_{h'})-\tilde\mu(L_0,L_1;R_h)+\dim(R_{h'})-\dim R_h-k_0n-k_1n-3)\\&&=\\
&&-\sum\dim(P_i^0)-\sum\dim(P_i^1)-\dim(S)-k_0-k_1\\&&+\tilde\mu(L_0,L_1;R_h)+\dim R_h+k_0n+k_1n+1.
\end{eqnarray*}
Therefore $n_{k_0,k_1}$ has degree $1-k_0-k_1$.
\end{proof}

The operations $n_{k_0,k_1}$ make $CF(L_0,L_1)$ into a $(CF(L_0),CF(L_1))$ $A_\infty$-bimodule.

\section{Floer cohomology}\label{section:floer_cohomology}
Finally, having dispensed with the necessary definitions and constructions, we can define Floer cohomology.
We begin this section by giving the defintion for a Lagrangian and a pair of Lagrangians.
Then, some remarks on the invariance properties under Hamiltonian isotopy are made, and we mention Poincar\'e duality.
Finally, the spectral sequence given in \cite{fooo} is presented.

\subsection{Definitions}
Let $L,L_0,L_1$ be graded Lagrangian submanifolds.
Furthermore, assume that they are all orientable and spin, and fix an orientation and spin structure on each of them.
(Strictly speaking, the Floer cohomology depends on the choice of spin structures and gradings, and hence these structures should be included in the notation.
However, in this paper we fix these structures on each Lagrangian, and hence no confusion should arise from not explicitly mentioning them.)

By the discussion in Section \ref{section:a_infinity-algebras_bimodules}, if $m_0(1)=0$ then $CF(L)$ is a cochain complex with differential $m_1$.
\begin{definition}
If $m_0(1)=0$ then the Floer cohomology $$HF(L;\Lambda_{0,nov})$$ of $L$ is the cohomology of the complex $(CF(L),m_1)$.

More generally, if $b$ is a bounding cochain then the Floer cohomology $$HF((L,b);\Lambda_{0,nov})$$ of $(L,b)$ is the cohomology of the complex $(CF(L),m_1^b)$.
Recall that a bounding cochain $b$ is an element $b\in CF^1(L)$ such that $b\equiv 0$ mod $\Lambda_{0,nov}^+$ and $b$ satisfies the Maurer-Cartan equation:
\begin{displaymath}
m_0(1)+m_1(b)+\cdots=0.
\end{displaymath}
\end{definition}

Let $m^0$ and $m^1$ be the operators for the $A_\infty$ algebras $CF(L_0)$ and $CF(L_1)$.
\begin{definition}
If $m^0_0(1)=0$ and $m^1_0(1)=0$ then $CF(L_0,L_1)$ is a cochain complex with differential $n_{0,0}$ (see Section \ref{section:a_infinity-algebras_bimodules}.)
The Floer cohomology $$HF(L_0,L_1;\Lambda_{0,nov})$$ of the pair $(L_0,L_1)$ is the cohomology of the comnplex $(CF(L_0,L_1),n_{0,0})$.

More generally, if $b_0$ and $b_1$ are bounding cochains in $CF(L_0)$ and $CF(L_1)$, then the Floer cohomology $$HF((L_0,b_0),(L_1,b_1);\Lambda_{0,nov})$$ is the cohomology of the complex $CF(L_0,L_1)$ with differential $n_{0,0}^{b_0,b_1}$.
\end{definition}

$\Lambda_{0,nov}$ is not a field, so the spectral sequence for Floer cohomology (see Section \ref{section:floer_cohomology-spectral_sequences}) cannot be used to determine the torsion parts of the cohomology with $\Lambda_{0,nov}$ coefficients.
Therefore, it will be more convenient to work with $\Lambda_{nov}$ coefficients.
Note that $\Lambda_{nov}$ is the field of fractions of $\Lambda_{0,nov}$.
In particular, $\Lambda_{nov}$ is a flat $\Lambda_{0,nov}$-module, so the cohomology of the complex
\begin{displaymath}
(CF(L)\otimes \Lambda_{nov},m_1^b\otimes id)
\end{displaymath}
is simply $HF(L;\Lambda_{0,nov})\otimes \Lambda_{nov}$.
Therefore cohomology with $\Lambda_{nov}$ coefficients can simply be defined to be
\begin{displaymath}
HF((L,b);\Lambda_{nov})=HF((L,b);\Lambda_{0,nov})\otimes\Lambda_{nov}.
\end{displaymath}
Similar remarks hold for the Floer cohomology of the pair $((L_0,b_0),(L_1,b_1))$.

The more general notion of Floer cohomology involving bounding cochains will not be needed for the examples dealt with in this paper (see Section \ref{section:floer_cohomology-unobstructed}).
That is, we will always take $b=0$.

\subsection{Spectral sequences}\label{section:floer_cohomology-spectral_sequences}

First, some preliminaries on the Novikov ring $\Lambda_{0,nov}$ are needed.
A grading on $\Lambda_{0,nov}$ can be defined as follows:
Fix $\lambda_0>0$ (to be determined later).
The filtration 
$$F^{\lambda}\Lambda_{0,nov}=\set{\sum r_iT^{\lambda_i}e^{\mu_i}\in\Lambda_{nov}}{\lambda_i\geq \lambda}$$
gives a $\zz$ filtration by
  $$\mathcal F^p\Lambda_{0,nov}=F^{p\lambda_0}\Lambda_{0,nov}.$$
  This leads to an associated grading
  $$gr_p(\mathcal F \Lambda_{0,nov})=\mathcal F^p\Lambda_{0,nov}/\mathcal F^{p+1}\Lambda_{0,nov}.$$

%Along with the formula for the Maslov index from Lemma ??, they allow the immediate calculation of some Floer cohomologies.
The following theorem from \cite{fooo} is one of the main tools that will be used in this paper.

  \begin{theorem}[\cite{fooo} Theorem 24.4] \label{thm:floer_cohomology-spectral_sequence}
  There exists a spectral sequence associated to $CF(L)$ such that
  $$E_2^p=\oplus_{k,p} H^k(L;\qq)\otimes gr_p(\mathcal F\Lambda_{0,nov}).$$
  %\mathcal F^q HF^p((L,b_1),(L,b_0);\Lambda_{0,nov})/\mathcal F^{q+1} HF^p((L,b_1),(L,b_0);\Lambda_{0,nov}).$$
  There exists a spectral sequence associated to $CF(L_0,L_1)$ such that
  $$E_2^p=\bigoplus_{k,R_h,p}H^{k-\tilde\mu(L_0,L_1;R_h)}(R_h,\Det(TR_h)\otimes\Theta_{R_h}^{-})\otimes gr_p (\mathcal F\Lambda_{0,nov}).$$
  In both case the coboundary map $\delta_r$ goes from $E^{p}_r\to E^{p+r-1}_r.$
  \end{theorem}

In particular, if $\lambda_0$ is chosen correctly, the map on the $E_2$ term of the spectral sequence is determined by the lowest energy moduli spaces.

Notice that the spectral sequence is stated for $\Lambda_{0,nov}$ coefficients.
$\Lambda_{0,nov}$-modules can have torsion, and the spectral sequence cannot be used to determine the torsion parts.
However, it can be used to determine the rank, and hence also the rank with $\Lambda_{nov}$ coefficients.

\subsection{Invariance under Hamiltonian isotopy}
If $\psi_0$ and $\psi_1$ are Hamiltonian isotopies, then they induce homotopy equivalences
\begin{displaymath}
\psi_{i*}:CF(L_i)\to CF(\psi_i(L_i)).
\end{displaymath}
In particular, this means that if $b_0\in CF(L_0)$ and $b_1\in CF(L_1)$ are bounding cochains, then $\psi_*(b_0)\in CF(\psi_0(L_0)) $ and $\psi_*(b_1)\in CF(\psi_1(L_1))$ are boundng cochains.
Furthermore, it means that
\begin{eqnarray*}
HF((L_i,b_i);\Lambda_{nov})&\cong& HF((\psi_i(L_i),\psi_{i*}(b_i));\Lambda_{nov}),\\
HF((L_0,b_0),(L_1,b_1);\Lambda_{nov})&\cong & HF((\psi_0(L_0),\psi_{0*}(b_0)),(\psi_1(L_1),\psi_{1*}(b_1));\Lambda_{nov}).
\end{eqnarray*}
Note that the invariance property is not true with $\Lambda_{0,nov}$ coefficients; i.e. the Floer cohomologies can have different torsion parts.

\subsection{Poincar\'e duality}
In \cite{seidel1}, it is shown that one of the benefits of using graded Lagrangians is that a form of Poincar\'e duality holds:
If $L_0,L_1$ are graded Lagrangians of dimension $n$ then
\begin{displaymath}
HF^k(L_0,L_1)\cong HF^{n-k}(L_1,L_0)^*.
\end{displaymath}
If as in \cite{seidel1} the classical notion of Floer cohomology is being used (say $\zz_2$ coefficients, $L_0$ and $L_1$ are transverse, and some assumption that precludes serious difficulties arising from disc bubbling), this isomorphism follows from the simple fact that if $p\in L_0\cap L_1$ then
\begin{displaymath}
\tilde\mu(L_0,L_1;p)=n-\tilde\mu(L_0,L_1;p).
\end{displaymath}

One can ask whether Poincar\'e duality holds in the Bott-Morse case.
Rather than investigating this question in detail, we simply point out here that the grading convention in the Bott-Morse case supports the notion that there should be Poincar\'e duality.
Indeed, suppose $[\alpha]\in HF^k(L_0,L_1)$, where $\alpha$ is a singular chain of homogeneous degree contained in the component $R$ of $L_0\cap L_1$.
Then
\begin{displaymath}
\alpha\in C^{k-\tilde\mu(L_0,L_1;R)}(L_0,L_1).
\end{displaymath}
Let $\phi$ be the singular cochain that is Poincar\'e dual to $\alpha$.
Then
\begin{displaymath}
\phi\in C^{\dim R-k+\tilde\mu(L_0,L_1;R)}(L_0,L_1)^*.
\end{displaymath}
Now $\tilde\mu(L_0,L_1;R)=n-\dim R-\tilde\mu(L_1,L_0;R)$, so
\begin{displaymath}
\dim R-k+\tilde\mu(L_0,L_1;R)=\dim R-k+n-\dim R-\tilde\mu(L_1,L_0;R)=n-k-\tilde\mu(L_1,L_0;R).
\end{displaymath}
Thus $\phi$ can be viewed as an element of $CF^{n-k}(L_1,L_0)^*$.
Hence, this suggests that Poincar\'e duality gives an isomorphism
\begin{displaymath}
HF^k(L_0,L_1)\cong HF^{n-k}(L_1,L_0)^*.
\end{displaymath}

\section{Local systems}\label{section:local_systems}
We now finish defining the bimodule operators $n_{k_0,k_1}$.
In particular, we describe the local systems which are used, and then explain how the operators $n_{k_0,k_1}$ take values in these local systems.

\subsection{Definition of the local systems}\label{section:local_systems-definition}
First, we need to discuss the notion of a spin structure.
Let $P_{SO}(L_i)$ be principal $\textrm{SO}(n)$ bundles on $TL_i$.
A spin structure is a principal $\textrm{Spin}(n)$ bundle $P_{Spin}(L_i)\to L_i$ along with a commutative diagram
\begin{displaymath}
\begin{array}{ccc}
P_{Spin}(L_i) & \longrightarrow & P_{SO}(L_i) \\
\searrow &&\swarrow\\
& L_i &
\end{array}
\end{displaymath}
such that the horizontal arrow induces the nontrivial double cover $\textrm{Spin}(n)\to \textrm{SO}(n)$ on each fiber.
Two spin structures $P_{Spin}(L_i)$ and $P_{Spin}'(L_i)$ are isomorphic if there exists a bundle isomorphism covering the identity map on $P_{SO}(L_i)$.
(Note that $P_{Spin}(L_i)$ and $P_{Spin}'(L_i)$ may be isomorphic as bundles over $L_i$ but not as spin structures on $L_i$.)
A spin structure and orientation on $L_i$ determines a homotopy class of a trivialization of $TL_i$ over any loop in $L_i$. 
Spin structures play a role in orientation issues because of the next lemma.
\begin{lemma}[\cite{fooo} Proposition 34.4]\label{lemma:local_systems-main_lemma}
Let $E$ be a symplectic vector bundle over $D^2$ and let $F\to\partial D^2$ be a Lagrangian subbundle over the boundary.
The choice of a trivialization of $F$ determines an orientation of $\Det(\dbar_{E,F})$ in a canonical way.
Therefore, an orientation of $\Det(\dbar_{E,F})$ is determined by a Spin structure and orientation of $F$.
\end{lemma}

From now on assume that $L_0$ and $L_1$ have fixed orientations $or(L_i)$, fixed $\textrm{SO}(n)$ bundles $P_{SO}(L_i)$, and fixed spin structures $P_{Spin}(L_i)$.
Let $N=(TL_0+TL_1)/T(L_0\cap L_1)$, and view $N$ as a symplectic subspace of $TX$.
Let $\overline{TL_0}$ and $\overline{TL_1}$ be the images of $TL_0$ and $TL_1$ in $N$.
Recall the definition of $Z_\pm$ in (\ref{eq:index_theory-caps}) and the operators $\dbar_{\tilde \lambda,Z_{\pm}}$ in (\ref{eq:index_theory-cap_operators}).
For $p\in R_h$, let $\tilde\lambda$ denote a path of the form $\tilde \lambda=\lambda\oplus T_pR_h$, where $\lambda$ is a path in $\Lambda(N)$ from $\overline{T_pL_0}$ to $\overline{T_pL_1}$.
Moreover, assume that the orientation of $T_pL_0$ is continued by $\tilde\lambda$ to the orientation of $T_pL_1$.
%Let $\mathcal P_p(L_0,L_1)$ denote the totality of the paths $\tilde\lambda$ and let $\mathcal P_{R_h}(L_0,L_1)=\amalg_{p\in R_h}\mathcal P_p(L_0,L_1)$.
%The one-dimensional vector spaces $\Det(\dbar_{\tilde\lambda,Z_{\pm}})$ fit together to give lines bundles
%\begin{displaymath}
%\mathcal E^{\pm}\to \mathcal P_{R_h}(L_0,L_1).
%\end{displaymath}
%It can be shown that $\mathcal E^{\pm}$ restricted to $\mathcal P_p(L_0,L_1)$ is non-trivial for every $p\in R_h$.
%(Essentially because $\pi_1(\mathcal P_p(L_0,L_1))=\pi_2(\Lambda^{ori}(N_p))=\zz_2$.)
%Let $\widetilde{\mathcal P_p}(L_0,L_1)$ denote the universal cover (i.e. non-trivial double cover) of $\mathcal P_p(L_0,L_1)$.
%The spaces $\widetilde{\mathcal P_p}(L_0,L_1)$ fit together to give a double cover $\widetilde {\mathcal P_{R_h}}(L_0,L_1)\to\mathcal P_{R_h}(L_0,L_1)$.

%In Section 51 of \cite{fooo}, local systems $\Theta^\pm_{R_h}$ are constructed on each component $R_h$ using the determinant bundles $\Det(\dbar_{\lambda\oplus TR_h,Z_\pm})$.
%Here is the construction:
%For each component $R_h$, fix once and for all the following data:
%\begin{itemize}
%\item a point $p_h\in R_h$,
%\item a path $\tilde\lambda_h=\lambda_h\oplus T_{p_h}R_h$,
%\item a lift $\tilde\lambda_h'\in\widetilde{\mathcal P_{p_h}}(L_0,L_1)$ of $\tilde\lambda_h$,
%\item a trivialization $\sigma_h$ of the bundle $\tilde\lambda_h\to[0,1]$, and
%\item trivializations $P_{Spin}(L_i)_{p_h}\cong\textrm{Spin(n)}$.
%\end{itemize}

To define the local systems $\Theta_{R_h}^{\pm}$, it suffices to define homomorphisms $\pi_\pm:\pi_1(R_h)\to \zz_2$.
The definitions of $\pi_\pm$ are as follows:
Fix a basepoint $p\in R_h$, and let $\gamma$ be a loop based at $p$. 
The spin structures on $L_0$ and $L_1$ give (homotopy classes of) trivializations of $TL_0$ and $TL_1$ over $\gamma$; that is there exists maps 
\begin{displaymath}
Fr_i:S^1\to\gamma^*Fr(TL_i),
\end{displaymath}
where $Fr(TL_i)$ is the frame bundle of $TL_i$.
Think of the points of $S^1$ as being of the form $e^{2\pi i\theta}$ where $0\leq\theta\leq1$.
Let $\tilde\lambda=\lambda\oplus T_pR_h:[0,1]\to\Lambda(T_pX)$ be the path of Lagrangian subspaces described above.
%$\overline{T_pL_0}\oplus T_pR_h=T_pL_0$ to $\overline{T_pL_1}\oplus T_pR_h=T_pL_1$.
Let $t$ denote the parameter in $[0,1]$.
Then choose a one paramenter family $\tilde\lambda_\theta=\lambda_\theta\oplus T_{\gamma(e^{2\pi i \theta})}R_h$ of Lagrangian paths that continues $\tilde\lambda$ around the loop $\gamma$ such that
\begin{itemize}
\item $\tilde\lambda_0=\tilde\lambda$,
\item $\tilde\lambda_\theta:[0,1]\to \Lambda(T_{\gamma(e^{2\pi i\theta})}X)$, 
\item $\tilde\lambda_\theta(0)=T_{\gamma(e^{2\pi i\theta})}L_0$, and
\item $\tilde\lambda_\theta(1)=T_{\gamma(e^{2\pi i\theta})}L_1$.
\end{itemize}
It may be assumed that $\tilde\lambda_0=\tilde\lambda_1$.
(Indeed, thinking of $\theta$ as fixed and $\tilde\lambda_\theta$ as being a path in $\Lambda(T_{\gamma(e^{2\pi i\theta})}X)$, the Maslov index $\mu(\tilde\lambda_\theta,T_{\gamma(e^{2\pi i\theta})}L_0)$ is half-integer valued. 
Moreover, the stratum of the Maslov cycle of $T_{\gamma(e^{2\pi i\theta})}L_0$ in which $\tilde\lambda_\theta(0)$ lies is the continuation of the stratum in which $\tilde\lambda_0(0)$ lies. 
Similarly, the stratum in which $\tilde\lambda_\theta(1)$ lies is the continuation of the stratum in which $\tilde\lambda_0(1)$ lies.
Therefore, $\mu$ depends continuously on $\theta$, and since it is half-integer valued it follows that it is constant.
Thus
\begin{displaymath}
\mu(\tilde\lambda_0,T_pL_0)=\mu(\tilde\lambda_1,T_pL_0).
\end{displaymath}
Since $\tilde\lambda_0$ and $\tilde\lambda_1$ have the same endpoints, it follows that they are homotopic.
Thus it may be assumed that $\tilde\lambda_0=\tilde\lambda_1$.
See \cite{rs}, in particular Theorem 2.4 for more details about stratum homotopies.)
Next, let $Tr$ be a trivialization of $\tilde\lambda$.
That is, for $0\leq t\leq 1$,
\begin{displaymath}
Tr(t):\tilde\lambda(t)\to\rr^n
\end{displaymath}
is an isomorphism.
Assume that the bases of $\tilde\lambda(0)=T_pL_0$ and $\tilde\lambda(1)=T_pL_1$ given by $Tr(0)$ and $Tr(1)$ agree with the bases given by $Fr_0(e^{2\pi i 0})$ and $Fr_1(e^{2\pi i 0})$.
This can be done because $\tilde\lambda$ was chosen such that it continues the orientation of $T_pL_0$ to the orientation of $T_pL_1$.
Now let $Tr_\theta$ be a one-parameter family of trivializations that continues the trivialization $Tr$ of $\tilde\lambda$ around the loop of paths $\tilde\lambda_\theta$.
That is, 
\begin{itemize}
\item $Tr_0=Tr$,
\item $Tr_\theta$ is a trivialization of $\tilde\lambda_\theta$, and
\item the basis of $T_{\gamma(e^{2\pi i\theta})}L_i$ given by $Tr_\theta(i)$ agrees with that given by $Fr_i(e^{2\pi i\theta})$.
\end{itemize}
By changing the loop $\tilde\lambda_\theta$ if necessary, it may be assumed that $Tr_0=Tr_1$ (because if $(V,\omega)$ is a symplectic vector space, then a trivialization of a loop in $\Lambda(V)$ continued around a non-trivial loop in the based loop space of $\Lambda(V)$ switches the homotopy class of the trivialization).
Finally, consider the line bundle $\gamma^*\Det(\dbar_{\tilde\lambda_\theta,Z_\pm})$.
If this bundle is trivial, define $\pi_\pm(\gamma)$ to be $1$, otherwise define $\pi_\pm(\gamma)$ to be $-1$.
\begin{lemma}
As real line bundles over $R_h$, 
\begin{displaymath}
\Theta^-_{R_h}\otimes\Det(TR_h)\cong \Theta^+_{R_h}.
\end{displaymath}
\end{lemma}
\begin{proof}
This statement is given in Section 41 of \cite{fooo}.
However, the description of $\Theta^\pm$ given there is slightly different from that given here, so a proof is necessary.

%It suffices to show that $\Theta_{R_h}^-\otimes\Theta_{R_{h'}}^+\cong \Det(TR_h)$.
Let $\gamma$ be a loop in $R_h$ based at $p$, and let $\tilde\lambda_\theta$ and $Tr_\theta$ be as above.
%It can be assumed that a trivialization $Tr_0$ of $\tilde\lambda_0$ can be extended to a one-parameter family of trivializations $Tr_\theta$ of $\tilde\lambda_\theta$ in such a way that $Tr_0=Tr_1$.
By a gluing lemma, there exists a canonical isomorphism
\begin{equation}\label{eq:gluing-canonical-iso}
\Det(\dbar_{\tilde\lambda_\theta\#\tilde\lambda_\theta})\cong \Det(\dbar_{\tilde\lambda_\theta,Z_-}\oplus TR_h)\times_{T_{\gamma(\theta)}R_h}\Det(TR_h\oplus\dbar_{\tilde\lambda_\theta,Z_+}).
\end{equation}
Here, $\tilde\lambda_\theta\#\tilde\lambda_\theta$ denotes the Lagrangian subbundle over the boundary of $D^2$ defined by
\begin{displaymath}
\tilde\lambda_\theta\#\tilde\lambda_\theta(e^{i(\pi\phi-\pi/2)})=\left\{
\begin{array}{cl}
\tilde\lambda_\theta(\phi),& 0\leq\phi\leq 1\\
\tilde\lambda_\theta(-\phi),&-1\leq\phi\leq 0.
\end{array}
\right.
\end{displaymath}

As line bundles over $\gamma$, Lemma \ref{lemma:local_systems-main_lemma} and the fact that $Tr_0=Tr_1$ imply that the left hand side of (\ref{eq:gluing-canonical-iso}) is trivial.
$\Det(TR_h)$ appears three times in the right hand side, therefore the equation implies
\begin{displaymath}
 \Det(\dbar_{\tilde\lambda_\theta,Z_-})\otimes \Det(TR_h) \cong \Det(\dbar_{\tilde\lambda_\theta,Z_+}).
\end{displaymath}
As line bundles over $\gamma$, by definition $\Det(\dbar_{\tilde\lambda_\theta,Z_\pm})\cong \Theta_{R_h}^\pm$.
\end{proof}

Let $p\in R_h$ and $\tilde\lambda_p$ be a path from $T_pL_0$ to $T_pL_1$ as above.
An explicit isomorphism (up to homotopy) between $\Theta_{R_h,p}^\pm$ and $\Det(\dbar_{\tilde\lambda_p,Z_\pm})$ can be constructed.
The isomorphism will depend on additional data, namely pick (once and for all) for each component $R_h$ the following data:
\begin{itemize}
\item a point $p_h\in R_h$,
\item a path $\tilde\lambda_h=\lambda_h\oplus T_{p_h}R_h$ from $T_{p_h}L_0$ to $T_{p_h}L_1$, and
\item a trivialization $Tr_h$ of the path $\tilde\lambda_h$.
\end{itemize}
Let the fiber of $\Theta_{R_h}^\pm$ over $p_h$ be $\Det(\dbar_{\tilde\lambda_h,Z_\pm})$.

Then, given a trivialization $Tr_p$ of the path $\tilde\lambda_p$, $\Theta_{R_h,p}^\pm$ can be related to $\Det(\dbar_{\tilde\lambda_p,Z_\pm})$ as follows:
Choose a path $\gamma$ from $p_h$ to $p$.
The spin structures and orientations of $TL_i$ give (homotopy classes of) frames $Fr_i:[0,1]\to \gamma^*Fr(TL_i)$ such that
\begin{itemize}
\item the basis of $T_{p_h}L_i$ given by $Tr_h(i)$ agrees with the basis given by $Fr_i(0)$, and
\item the basis of $T_{p}L_i$ given by $Tr_p(i)$ agrees with the basis given by $Fr_i(1)$.
\end{itemize}
Choose a one-parameter family of paths $\tilde\lambda_\theta$ that covers $\gamma$ such that $\tilde\lambda_0=\tilde\lambda_h$ and $\tilde\lambda_1=\tilde\lambda_p$.
(This may not always be possible, since the Maslov indices of $\tilde\lambda_h$ and $\tilde\lambda_p$ may be different. However, there is a canonical way to compare determinant lines with different Maslov indices, by gluing a sphere onto $Z_\pm$ and then putting a non-trivial vector bundle on the sphere. See \cite{fooo} for details.)
Moreover, assume that the path is chosen such that the trivialization $Tr_h$ extends to a one-parameter family of trivializations $Tr_\theta$ such that $Tr_1=Tr_p$.
Two line bundles are now defined over $\gamma$, namely $\gamma^*\Det(\dbar_{\tilde\lambda_\theta,Z_\pm})$ and $\gamma^*\Theta_{R_h}^\pm$.
The fibers of these bundles are the same over $0$, hence orientations of the fibers over $1$ can be compared by continuity.
This gives an isomorphism (well-defined up to homotopy)
\begin{equation}\label{eq:local_systems-compare}
\Theta_{R_h,p}^\pm\cong \Det(\dbar_{\tilde\lambda_p,Z_\pm}).
\end{equation}

\subsection{Definition of the operators}\label{section:local_systems-definition_operators}
Recall from Section \ref{section:a_infinity-floer_cohomology} that the Floer cochain complex in the Bott-Morse setting is
\begin{displaymath}
CF(L_0,L_1)=\oplus_{R_h}C(R_h),
\end{displaymath}
where $C(R_h)$ is a subcomplex of the singular chain complex (with cohomological grading) with coefficients in a local system.
The local system used is $\Det(TR_h)\otimes \Theta^-_{R_h}$.
It is shown in \cite{fooo} that with this convention orientations can be worked out in such a way that the $A_\infty$-bimodule relations hold.

In the case of $n_{0,0}$, which is the only operator considered in this paper, here is the construction.
Let $S$ be a chain in $C(R_h)$.
Without loss of generality, assume that the domain of $S$ consists of a single $k$-simplex $\Delta^k$ with the standard orientation.
Explicitly, $S=(\phi,s)$ where $\phi:\Delta^k\to R_h$ is a smooth map and $s$ is a flat section of $\phi^*(\Det(TR_h)\otimes\Theta^-_{R_h})$.
Equivalently, the section $s$ along with the orientation can instead be thought of as a flat section of $\phi^*(\Det(TR_h)\otimes\Theta^-_{R_h})\otimes\Det(T\Delta^k)$.
Without loss of generality assume that $s$ is non-zero (otherwise the chain $S$ is $0$).
Recall that $n_{0,0}$ can be decomposed as
\begin{displaymath}
n_{0,0}=\sum_{h',\beta}\delta_{\mathcal M(L_0,L_1:R_h,R_{h'}:\beta)}\otimes T^{\omega(\beta)},
\end{displaymath}
where $\mathcal M(L_0,L_1:R_h,R_{h'}:\beta)$ is the moduli space of holomorphic strips from $R_h$ to $R_{h'}$ of homotopy class $\beta$ (recall that $\mathcal M$ is shorthand for $\mathcal M_{0,0}$) and
\begin{displaymath}
\delta_{\mathcal M(L_0,L_1:R_h,R_{h'}:\beta)}S=ev_{+\infty*}[\mathcal M(L_0,L_1:R_h,R_{h'}:\beta)\times_{R_h}S].
\end{displaymath}
To describe what $\mathcal M(L_0,L_1:R_h,R_{h'})\times_{R_h}S$ is a lemma is first needed.
\begin{lemma}[\cite{fooo} Proposition 41.6]\label{lemma:local_systems-identification}
Spin structures on $L_0$ and $L_1$ induce a canonical isomorphism
\begin{displaymath}
\Det(\mathcal M(L_0,L_1:R_h,R_{h'}))\cong ev_{+\infty}^*\Theta^+_{R_{h'}}\otimes ev_{-\infty}^*\Theta^-_{R_h}.
\end{displaymath}
\end{lemma}

According to the lemma and the conventions in \cite{fooo}, 
\begin{displaymath}
\Det(\mathcal M(L_0,L_1:R_h,R_{h'})\times_{R_h} \Delta^k))\cong ev_{+\infty}^*(\Theta^+_{R_{h'}})\otimes ev_{-\infty}^*(\Theta_{R_h}^-\otimes\Det(TR_h))\otimes\Det(T\Delta^k)
\end{displaymath}
as bundles over $\mathcal M(L_0,L_1:R_h,R_{h'})\times_{R_h}\Delta^k$.
Therefore
\begin{displaymath}
ev_{+\infty}^*(\Theta^+_{R_{h'}})\otimes \Det(\mathcal M(L_0,L_1:R_h,R_{h'})\times_{R_h}\Delta^k)\cong ev_{-\infty}^*(\Theta_{R_h}^-\otimes\Det(TR_h))\otimes\Det(T\Delta^k).
\end{displaymath}
$\phi=ev_{-\infty}$ on $\mathcal M(L_0,L_1:R_h,R_{h'})\times_{R_h}\Delta^k$, therefore $s$ induces a section of the right-hand side of the above equation, and hence also a section of the left-hand side, call this section $s'$.
Then $ev_{+\infty*}(\mathcal M(L_0,L_1:R_h,R_{h'})\times_{R_h}S)$ is defined to be the chain $(ev_{+\infty*}(\mathcal M(L_0,L_1R_h,R_{h'})\times_{R_h}\Delta^k),s')$ in the local system $\Theta_{R_{h'}}^+\cong\Det(TR_{h'})\otimes \Theta^-_{R_{h'}}$.

\part{Properties of real Lagrangians}
In Part I we reviewed the definitions and constructions used in Floer theory.
In this part, we turn to investigating general properties of the real Lagrangians that will then be used to calculate the Floer cohomology in Part III.

  \section{Real Lagrangians}\label{section:real_lagrangians-intersection}
Recall that real Lagrangians are of the form $gL$, where $L$ is the set of real points of $X$ and 
$$g=(1,\gamma^r,\gamma^s,\gamma^t,\gamma^u)=g_1^rg_2^sg_3^tg_4^u$$
is the automorphism of $X$ given by $$g:[X_0:X_1:X_2:X_3:X_4]\mapsto [X_0:\gamma^rX_1:\gamma^sX_2:\gamma^tX_3:\gamma^uX_4]$$
where $\gamma=e^{2\pi i/5}$.

In this section, the possibilities for $L\cap gL$ are enumerated, it is shown that the real Lagrangians are graded, and a formula for the Maslov index is developed.

  \subsection{Intersections of real Lagrangians}
  \begin{lemma} \label{lemma:real_lagrangians-fixed}
  $L\cap gL=Fix(g)\cap L$.
  \end{lemma}
  \begin{proof}
  Let $gx\in L\cap gL$, with $x\in L$.
  Then $\overline{gx}=gx$, since $L$ consists of real points.
  On the other hand, $\overline{gx}=g^{-1}x$.
  Thus $x=g^2x$, and since $g$ has order $5$ it follows that $gx=x$.
  \end{proof}

  \begin{lemma}\label{lemma:real_lagrangians-cases}
  Let $S=L\cap gL$. 
  Then the possibilities for $S$ are (assuming $i,j,k$ are distinct and $1\leq r,s,t,u\leq 4$):
  \begin{enumerate}
  \item $S\cong \rp^3$ if $g=1$,
  \item $S\cong \rp^2$ if $g=g_i^r,$
  \item $S\cong \rp^1$ if $g=g_i^rg_j^s,r\neq s$,
  \item $S\cong \rp^1\coprod\rp^0$ if $g=g_i^rg_j^s, r=s$,
  \item $S\cong \rp^0$ if $g=g_i^rg_j^sg_k^t, r\neq s,r\neq t,s\neq t$,
  \item $S\cong \rp^0\coprod\rp^0$ if $g=g_i^rg_j^sg_k^t, r=s,r\neq t$,
  \item $S\cong \rp^1\coprod\rp^0$ if $g=g_i^rg_j^sg_k^t, r=s=t$,
  \item $S\cong \emptyset$ if $g=g_1^rg_2^sg_3^tg_4^u, r\neq s,r\neq t,r\neq u,s\neq t,s\neq u,t\neq u$,
  \item $S\cong \rp^0$ if $g=g_1^rg_2^sg_3^tg_4^u, r\neq s,r\neq t,r\neq u,s\neq t,s\neq u,t= u$,
  \item $S\cong \rp^0\coprod\rp^0$ if $g=g_1^rg_2^sg_3^tg_4^u, r=s,s\neq t,t=u$,
  \item $S\cong \rp^1$ if $g=g_1^rg_2^sg_3^tg_4^u, r\neq s,s=t=u$,
  \item $S\cong \rp^2$ if $g=g_1^rg_2^sg_3^tg_4^u, r=s=t=u$.
  \end{enumerate}
  \end{lemma}
  \begin{proof}
  The proofs are all similar and follow easily from Lemma \ref{lemma:real_lagrangians-fixed}.
  For example, to prove (2) with say $i=1$, note that $Fix(g)=\sett{[Z_0:0:Z_2:Z_3:Z_4]\in X}$, so $L\cap gL=Fix(g)\cap L=\sett{[X_0:0:X_2:X_3:X_4]\in L}$.
  This is easily seen to be diffeomorphic to $\rp^2$.
  \end{proof}

  There is also an $S_5$ action given by permuting coordinates.
  Using this action, only cases (1)-(6) need to be considered.
  For example, in case (12), $g=[(1,\gamma^r,\gamma^r,\gamma^r,\gamma^r)]=[(\gamma^{5-r},1,1,1,1)]$.
  The permuatation that switches the first two homeogenous coordinates acts on $X$.
  This permutation fixes $L$, but changes $gL$ into case (2).

\subsection{Real Lagrangians have constant phase}

The Poincar\'e residue map is the map of sheaves $\Omega^{4}_{\cp^4}(X)\to\Omega^3_X$
that in local coordinates $(W_1,\ldots,W_4)$ on $\cp^4$ is the map
\begin{displaymath}
\frac{dW_1\wedge dW_2\wedge dW_3\wedge dW_4}{f}\mapsto (-1)^{i-1}\frac{dW_1\wedge\cdots \wedge\widehat{dW_i}\wedge\cdots\wedge dW_4}{\partial f/\partial W_i},
\end{displaymath}
provided that $\frac{\partial f}{\partial W_i}\neq 0$ (see \cite{gh} p. 147).
This induces a map $H^0(\cp^4,\Omega^4_{\cp^4}(X))\to H^0(X,\Omega^3_X)$.
$1/f$ can be interpreted as a global section of $H^0(\cp^4,\Omega^4(X))$; let $\Omega$ denote the image of $1/f$ under the Poincar\'e residue map.
Then $\Omega$ is a non-vanishing holomorphic section of $\Omega^3_X$.
Locally, $\Omega$ is simply 
\begin{displaymath}
(-1)^{i-1}\frac{dW_1\wedge\cdots \wedge\widehat{dW_i}\wedge\cdots\wedge dW_4}{\partial f/\partial W_i}.
\end{displaymath}

%Let $L'$ be a Lagrangian submanifold of $X$.
%$L'$ is called graded if there exists a map $\theta:L'\to\rr$ such that
%\begin{displaymath}
%\textrm{Im}(e^{-i\theta}\Omega)|_{L'}=0.
%\end{displaymath}
%If $\theta$ is a constant function then $L'$ is said to have constant phase $e^{i\theta}$.
%The phase is only well-defined up to a $\pm$ sign.
%If $L'$ is oriented, then the additional condition $\textrm{Re}(e^{-i\theta}\Omega)=c\cdot vol_{L'}$ for $c>0$ uniquely specifies the phase.

%If $L'$ has constant phase then $\textrm{Re}(e^{-i\theta}\Omega)|_{L'}=c\cdot vol_{L'}$ for some real valued function $c:L'\to (0,\infty)$.
%If $c$ is constant then $L'$ is calibrated with respect to the closed form $\frac{1}{c}\textrm{Re}(e^{-i\theta}\Omega)$; i.e. $L'$ is a special Lagrangian.
%If $\nabla \Omega=0$ then $c$ is constant, and $\nabla\Omega=0$ is equivalent to $\omega$ being Ricci-flat.
%A Ricci-flat metric exists by the solution of the Calabi conjecture; however, the standard symplectic structure is not Ricci-flat.
Recall from Section \ref{section:gradings-graded} that if a non-vanishing section $\Omega$ of the canonical bundle is given, then the notion of a graded Lagrangians can be defined.
\begin{lemma}
$L$ is graded and has constant phase $1\in S^1$.
\end{lemma}
\begin{proof}
Without loss of generality, assume that $p\in L$ is a point in affine coordinates $(W_1,W_2,W_3,W_4)\leftrightarrow[1:W_1:W_2:W_3:W_4]$ and $\partial f/\partial W_4(p)\neq 0$.
Let $W_i=x_i+\sqrt{-1}y_i$, and note that the tangent space of $L$ at $p$ is spanned by $\rr$-linear combinations of the $\partial/\partial x_i$'s.
Therefore $dy_i|_L=0$ and it follows that
\begin{displaymath}
\textrm{Im}(\Omega(p))|_L=\textrm{Im}\biggl(-\frac{dW_1\wedge dW_2\wedge dW_3}{\partial f/\partial W_4}(p)\biggr)\bigg|_L=0.
\end{displaymath}
This uses the fact that $\frac{\partial f}{\partial W_4}(p)$ is a real number.
\end{proof}
Recall that $\gamma=e^{2\pi i/5}$.
\begin{lemma}
Let $g:X\to X$ be the diffeomorphism 
\begin{displaymath}
[X_0:X_1:X_2:X_3:X_4]\mapsto [X_0:\gamma^aX_1:\gamma^bX_2:\gamma^cX_3:\gamma^dX_4].
\end{displaymath}
Then $gL$ is graded and has constant phase $\gamma^{2(a+b+c+d)}$.
\end{lemma}
\begin{proof}
$g$ maps $L$ to $gL$ and $\Omega$ to $g^{-1*}\Omega$.
$g^{-1*}\Omega$ can be calculated using the local description of $\Omega$ given above (without loss of generality assume $i=1$):
\begin{displaymath}
g^{-1*}\Omega=g^{-1*}\frac{dW_2\wedge dW_3\wedge dW_4}{\partial f/\partial W_1}=\frac{\gamma^{-b-c-d}dW_2\wedge dW_3\wedge dW_4}{\gamma^{-4a}\partial f/\partial W_1}=\gamma^{-a-b-c-d}\Omega.
\end{displaymath}
Therefore
\begin{displaymath}
\textrm{Im}(\gamma^{-a-b-c-d}\Omega)|_{gL}=\textrm{Im}(g^{-1*}\Omega)|_{gL}=g^{-1*}(\textrm{Im}(\Omega)|_L)=0.
\end{displaymath}
\end{proof}

To simplify the amount of work that needs to be done, a fixed grading will be used on each Lagrangian $gL$:
\begin{definition}\label{dfn:real_lagrangians-grading_convention}
Let $0\leq\theta_g<2\pi$ be the unique real number such that $e^{2\pi i\theta_g}=\gamma^{2(a+b+c+d)}$.
Then the fixed grading on $gL$ is defined to be
\begin{displaymath}
\theta_{gL}\equiv \theta_g.
\end{displaymath}
\end{definition}

  \subsection{Angles between the real Lagrangians}
  Let $p\in L\cap gL=Fix(g)\cap L$. 
  Then $dg$ acts on the vector space $T_pX$ and $dg\cdot T_pL=T_pgL$ (for ease of notation, the $dg$ action will often be denoted by $g$ as well, for example $g\cdot T_pL=T_pgL$).
If $R_h$ is the connected component of $L\cap gL$ containing $p$ then $\aangle{L}{gL}{R_h}$ is defined as in Section \ref{section:gradings-bott_morse} because $L$ and $gL$ have constant phase.
Since $\tau\circ g=g^{-1}\circ \tau$, it follows by Lemmas \ref{lemma:index_theory-linear_algebra_2} and \ref{lemma:index_theory-linear_algebra_3} that $dg$ is almost of the form given in Lemma \ref{lemma:index_theory-matrix}; the only problem is that the eigenvalues of $dg$ may not lie in the upper half-plane.
To correct for this, note that the eigenvalues of $g$ are of the form $\gamma^a,\gamma^b,\gamma^c$ where $0\leq a,b,c\leq4$.
  Let
  \begin{displaymath}
  a'=\left\{
  \begin{array}{lc}
  a, & 0\leq a\leq 2 \\
  a-5/2, & 3\leq a \leq 4
  \end{array}\right.,
  \end{displaymath}
  and similary for $b',c'$.
Then $\gamma^{a'}$, $\gamma^{b'}$, and $\gamma^{c'}$ lie in the upper half-plane.
Notice that $a,b,c$ are constant on connected components of $L\cap gL$, that is they depend only on $R_h$ and not on $p$.

\begin{lemma}\label{lemma:real_lagrangians-angles}
Consider the cases for $L\cap gL$ as enumerated in Lemma \ref{lemma:real_lagrangians-cases}
Then the angles between $L$ and $gL$ are 
\begin{enumerate}
\item $0$,
\item $r'/5$,
\item $(r'+s')/5$,
\item $2r'/5$ for the $\rp^1$ component and $3(5-r)'/5$ for the $\rp^0$ component,
\item $(r'+s'+t')/5$,
\item $(2r'+t')/5$ for the component where the $i,j,k$ coordinates are $0$ and $(2(5-r)'+(t-r)')/5$ for the other component.
\end{enumerate}
\end{lemma}
\begin{proof}
Recall that the angle is defined to be 
\begin{displaymath}
\aangle{L}{gL}{R}=\alpha_1+\alpha_2+\alpha_3,
\end{displaymath}
where $0<\alpha_i<\frac{1}{2}$ and the $e^{2\pi i\alpha_i}$'s are the eigenvalues of the unitary matrix $U$ such that $U\cdot T_pL_0=T_pL_1$ and $T_pL_0$ has a basis consisting of eigenvalues of $U$.
By the discussion above, the angles are $a'/5+b'/5+c'/5$ where $\gamma^a,\gamma^b,\gamma^c$ are the eigenvalues of $dg$.
Thus it remains to calculate the eigenvalues of $dg$.
Here are the computations:

Case (1) is obvious because $g=1$.

For case (2), assume $g=g_1^r$.
Then $L\cap gL=\sett{[X_0:0:X_2:X_3:X_4]\in X}$.
Near the point $p=[1:0:-1:0:0]$ let $W_i$'s be the affine coordinates
$$[1:W_1:W_2:W_3:W_4]\longleftrightarrow (W_1,W_2,W_3,W_4).$$
The $g$ action in affine coordinates is
$$g\cdot(W_1,W_2,W_3,W_4)=(\gamma^rW_1,W_2,W_3,W_4).$$
At the point $p$, 
$$T_pX=\ker(df(p))=\ker(dW_2)=\textrm{span}\sett{\frac{\partial}{\partial W_1},\frac{\partial}{\partial W_3},\frac{\partial}{\partial W_4}}.$$
Therefore the eigenvalues of $g$ are $\gamma^r,1,1$.

For case (3), assume $g=g_1^rg_2^s$.
Then $L\cap gL=\sett{[X_0:0:0:X_3:X_4]\in X}$.
Near the point $x=[1:0:0:-1:0]$ let the affine coordinates be
$$[1:W_1:W_2:W_3:W_4]\longleftrightarrow (W_1,W_2,W_3,W_4).$$
The $g$ action in affine coordinates is
$$g\cdot(W_1,W_2,W_3,W_4)=(\gamma^rW_1,\gamma^sW_2,W_3,W_4).$$
At the point $x$, 
$$T_xX=\ker(df(x))=\ker(dW_3)=\textrm{span}\sett{\frac{\partial}{\partial W_1},\frac{\partial}{\partial W_2},\frac{\partial}{\partial W_4}}.$$
Therefore the eigenvalues are $\gamma^r,\gamma^s,1$.

For case (4), assume $g=g_1^rg_2^r$.
Then 
$$L\cap gL=\sett{[X_0:0:0:X_3:X_4]}\amalg\sett{[0:X_1:X_2:0:0]}\cong\rp^1\amalg\rp^0.$$
Near the point $x=[1:0:0:-1:0]$ use affine coordinates
$$[1:W_1:W_2:W_3:W_4]\longleftrightarrow (W_1,W_2,W_3,W_4).$$
The $g$ action in affine coordinates is
$$g\cdot(W_1,W_2,W_3,W_4)=(\gamma^rW_1,\gamma^rW_2,W_3,W_4).$$
At the point $x$, 
$$T_xX=\ker(df(x))=\ker(dW_3)=\textrm{span}\sett{\frac{\partial}{\partial W_1},\frac{\partial}{\partial W_3},\frac{\partial}{\partial W_4}}.$$
Therefore the eigenvalues are $\gamma^r,\gamma^r,1$ for the $\rp^1$ component.

For the $\rp^0$ component, let $x=[0:1:-1:0:0]$ and use affine coordinates
$$[W_1:1:W_2:W_3:W_4]\longleftrightarrow (W_1,W_2,W_3,W_4).$$
The $g$ action in affine coordinates is
$$g\cdot(W_1,W_2,W_3,W_4)=(\gamma^{-r}W_1,W_2,\gamma^{-r}W_3,\gamma^{-r}W_4).$$
At the point $x$, 
$$T_xX=\ker(df(x))=\ker(dW_2)=\textrm{span}\sett{\frac{\partial}{\partial W_1},\frac{\partial}{\partial W_3},\frac{\partial}{\partial W_4}}.$$
Therefore the eigenvalues are $\gamma^{5-r},\gamma^{5-r},\gamma^{5-r}$ for the $\rp^1$ component.

For case (5), assume $g=g_1^rg_2^sg_3^t$.
Then $L\cap gL=\sett{[X_0:0:0:0:X_4]\in X}$.
Near the point $x=[1:0:0:0:-1]$ use affine coordinates
$$[1:W_1:W_2:W_3:W_4]\longleftrightarrow (W_1,W_2,W_3,W_4).$$
The $g$ action in affine coordinates is
$$g\cdot(W_1,W_2,W_3,W_4)=(\gamma^rW_1,\gamma^sW_2,\gamma^tW_3,W_4).$$
At the point $x$, 
$$T_xX=\ker(df(x))=\ker(dW_4)=\textrm{span}\sett{\frac{\partial}{\partial W_1},\frac{\partial}{\partial W_2},\frac{\partial}{\partial W_3}}.$$
Therefore the eigenvalues are $\gamma^r,\gamma^s,\gamma^t$.

For case (6), assume $g=g_1^rg_2^rg_3^t$.
Then $L\cap gL=[1:0:0:0:-1]\amalg[0:1:-1:0:0]$.
Near the point $x=[1:0:0:0:-1]$ use affine coordinates
$$[1:W_1:W_2:W_3:W_4]\longleftrightarrow (W_1,W_2,W_3,W_4).$$
The $g$ action in affine coordinates is
$$g\cdot(W_1,W_2,W_3,W_4)=(\gamma^rW_1,\gamma^rW_2,\gamma^sW_3,W_4).$$
At the point $x$, 
$$T_xX=\ker(df(x))=\ker(dW_4)=\textrm{span}\sett{\frac{\partial}{\partial W_1},\frac{\partial}{\partial W_2},\frac{\partial}{\partial W_3}}.$$
Therefore the eigenvalues are $\gamma^r,\gamma^r,\gamma^s$ for the $[1:0:0:0:-1]$ component.

Near the point $x=[0:1:-1:0:0]$ use affine coordinates
$$[W_1:1:W_2:W_3:W_4]\longleftrightarrow (W_1,W_2,W_3,W_4).$$
The $g$ action in affine coordinates is
$$g\cdot(W_1,W_2,W_3,W_4)=(\gamma^{-r}W_1,W_2,\gamma^{s-r}W_3,\gamma^{-r}W_4).$$
At the point $x$, 
$$T_xX=\ker(df(x))=\ker(dW_2)=\textrm{span}\sett{\frac{\partial}{\partial W_1},\frac{\partial}{\partial W_3},\frac{\partial}{\partial W_4}}.$$
Therefore the eigenvalues are $\gamma^{-r},\gamma^{-r},\gamma^{s-r}$.
\end{proof}

%  \begin{lemma}\label{prop:compindices}
%  Consider the possibilities for $L\cap gL$, as listed in Lemma \ref{lemma:intersections}.
%  Recall that only cases (1)-(6) need to be considered.
%  The weights of the components of $L\cap gL$ are
%  \begin{enumerate}
%  \item $wt(\rp^3)=0$,
%  \item $wt(\rp^2)=r'$,
%  \item $wt(\rp^1)=r'+s'$,
%  \item $wt(\rp^1)=2r'$, $wt(\rp^0)=3(5-r)'$,
%  \item $wt(\rp^0)=r'+s'+t'$,
%  \item $wt(\rp^0)=2r'+t'$, $wt(\rp^0)=2(5-r)'+(t-r)'$ or $2(5-r)'+(5+t-r)'$ if $t<r$.
%  \end{enumerate}
%  \end{lemma}
%  \begin{proof}
%  The proofs follow from the definition of $wt$ and the previous lemma.

%  \end{proof}

%The weights can be generalized in the following way.
%  Let $L_1,L_2\in\mathcal L$. 
%  Then $L_2=gL_1$ for some $g\in G$.
%  Let $x\in L_1\cap L_2$.
%  Then $x\in Fix(g)$, so $g$ acts on $T_xX$, and the weight can be defined as above.
%  Denote the weight by 
%  $$wt(x;L_1,L_2).$$
%  Again, $wt$ is constant on connected components of $L_1\cap L_2$. 

  \section{Floer cohomology of real Lagrangians}
  In Section \ref{section:a_infinity}, we explained how to construct an $A_\infty$-algebra $(CF(L),m)$ for each Lagrangian $L$, and an $A_\infty$-bimodule $(CF(L_0,L_1),n)$ for each pair of Lagrangians $(L_0,L_1)$.
In Section \ref{section:floer_cohomology} we showed that if the Lagrangians are unobstructed then the Floer cohomologies $HF(L)$ and $HF(L_0,L_1)$ are defined.
In this section we apply these ideas to the real Lagrangians.
  \subsection{Real Lagrangians are unobstructed}\label{section:floer_cohomology-unobstructed}
  A Lagrangian $L$ is unobstructed if there exists $b\in CF(L)^1$ such that $b\equiv 0 \textrm{ mod } \Lambda_{0,nov}^+$ and
  $$m_0(1)+m_1(b)+m_2(b,b)+m_3(b,b,b)+\cdots=0.$$
  Such $b$ are called bounding cochains, and they can be used to deform the operators $m_k$ into new operators $m_k^b$ such that $m_1^b\circ m_1^b=0$.
  The Floer cohomology of $(L,b)$ is the cohomology with respect to the differential $m_1^b$.

  In \cite{fooo}, it is shown that bounding cochains exist if certain (ordinary) cohomology classes vanish.
  In the case of real Lagrangians, the relevant cohomology classes lies in $H^1(L,\qq)=H^1(\rp^3,\qq)=0$, so the real Lagrangians are unobstructed.
  Moreover, $b=0$ can be taken as a bounding cochain by Theorem 1.5 in \cite{foooasi}.
Essentially, the reason is that any disc $u$ with boundary on $L$ can be conjugated to get a disc $\tilde u$ with boundary on $L$.
With one marked point added, the orientations of $u$ and $\tilde u$ are opposite of each other, and therefore cancel each other out.
It follows that $m_0(1)=0$ and $b=0$ is a bounding cochain.
Since $b=0$, the $A_\infty$ operators do not need to be deformed to carry out the calculations in this paper.

  \section{Holomorphic strips and spheres}
%In the previus section, it was shown that some Floer cohomologies could be calculated purely from degree considerations.
%The remaining examples turn out to be more difficult, and the calculations rely on several properties.
%In preparation for calculating the Floer cohomologies of the real Lagrangians, some properties of holomorphic strips on these Lagrangians need to be developed.
In this section we show that there is a bijection between holomorphic strips and holomorphic spheres with certain symmetries.
%The remaining properties are developed in the next two sections.
%The other needed properties are developed in subsequent sections, and then in Sections ??-?? the theory is applied to calculate some Floer cohomologies.

  \subsection{Reflecting strips to get a sphere}\label{section:strips-spheres}
  Let $\sigma=e^{2\pi i/10}$.
  Consider the sector $S_0$ in $\cc$ that is bounded by the rays $\rr_{\geq0}$ and $\sigma\rr_{\geq0}$.
  $S_0$ is biholomorphic to the strip $\rr\times[0,1]$ via the map
\begin{displaymath}
z\in\rr\times[0,1]\mapsto e^{\pi z/5}\in S_0.
\end{displaymath}
  Let $L_0=L$ and $L_1=gL$ for some $g\in G$.
  Let $u_0:S_0\to X$ be a holomorphic strip that maps $\rr_{\geq 0}$ into $L_0$ and $\sigma\rr_{\geq0}$ into $L_1$.
  Let $S_k=\sigma^kS_0$, so $\cc\setminus\sett{0}$ is covered by $S_0\cup\cdots\cup S_9$.
  Using the Schwarz reflection principle, the map $u_0$ can be reflected to get a map $u_k:S_k\to X$.
  Explicitly, for $1\leq k\leq 9$ define
  \begin{displaymath}
  u_k(z)=\left\{
  \begin{array}{ll}
  g^{k+1}\overline{u_0(\sigma^{k+1}\bar z)} & \textrm{ $k$ odd},\\
  g^ku_0(\sigma^{-k}z) & \textrm{ $k$ even} .
  \end{array}
  \right.
  \end{displaymath}
  It is not hard to check that $u_k$ agrees with $u_{k+1}$ along $S_k\cap S_{k+1}$.
  For example, if $z \in \rr_+$, then
  $$u_9(z)=g^{10}\overline{u_0(\sigma^{10}\bar z)}=\overline{u_0(\bar z)}=u_0(z).$$
Gluing the 
  Gluing the $u_k$'s together gives a holomorphic map $\cc\setminus\sett{0}\to X$.
If $u_0$ is assumed to have finite energy, then this map extends to a holomorphic sphere
  $$u:\cc\cup\sett{\infty}=\cp^1\to X.$$

  \begin{lemma} \label{lemma:strips-symmetries}
  $u$ has the following symmetries:
  \begin{enumerate}
  \item $u(z)=\overline{u(\bar z)}$, and
  \item $u(\gamma z)=g^2u(z)$.
  \end{enumerate}
  \end{lemma}
  \begin{proof}
  $u$ maps $\rr_{\geq 0}$ into $L$, so (1) is true for $z\in\rr_{\geq 0}$.
  Since both sides of the equation are holomorphic, equality must hold everywhere.

  (2) can be checked directly from the definition:
  Suppose $z\in S_k$ and $k$ is odd.
  Then 
  $$g^2u(z)=g^2u_k(z)=g^{k+3}\overline{u_0(\sigma^{k+1}\bar z)}=g^{k+3}\overline{u_0(\sigma^{k+3}\bar\sigma^2\bar z)}=$$
  $$=u_{k+2}(\sigma^2z)=u(\sigma^2z)=u(\gamma z).$$
  The calculation for $k$ even is similar.
  \end{proof}

In fact, any holomorphic sphere with the symmetries of Lemma \ref{lemma:strips-symmetries} comes from a holomorphic strip.
\begin{lemma}\label{lemma:strips-symmetries_converse}
Suppose $u:\cp^1\to X$ is a holomorphic sphere that satisfies properties (1) and (2) in Lemma \ref{lemma:strips-symmetries}.
Then there exists a holomorphic strip $u_0:S_0\to X$ that gives $u$ using the construction above.
\end{lemma}
\begin{proof}
Taking $u_0=u|S_0$ will give the desired strip if it can be shown that this $u_0$ has the correct boundary values.
Property (1) implies that $u(\rr_{\geq0})\subset L$.
It remains only to show that $u(\sigma\rr_{\geq0})\subset gL$.
Let $t\in\rr_{\geq0}$.
Using (2) gives
$$g^{-1}u(\sigma t)=g^{-1}u(\sigma^2\sigma^{-1} t)=gu(\sigma^{-1}t).$$
Then applying (1) gives
$$\overline{g^{-1}u(\sigma t)}=\overline{gu(\sigma^{-1}t)}=g^{-1}u(\sigma t).$$
Thus $g^{-1}u(\sigma t)\in L$, so $u(\sigma t)\in gL$.
\end{proof}

Any such $u$ is of the form
\begin{equation}\label{eq:polynomial_form_of_sphere}
u(z)=[u_0(z):u_1(z):u_2(z):u_3(z):u_4(z)]
\end{equation}
where each $u_i(z)$ is a polynomial.
(Thinking of the domain of $u$ as $\cc\cup\infty$, so the polynomials do not have to be homogeneous.)
\begin{lemma}\label{lemma:strips-polynomials}
Suppose $g=g_1^{a_1}g_2^{a_2}g_3^{a_3}g_4^{a_4}$.
Fix $j$ such that $u_j(0)\neq 0$, and for each $i$ let $k_i$ be such that $\gamma^{k_i}=\gamma^{-2a_j}$ and $5>k_i+2a_i\geq 0$.
Then $u_i$ can be taken to be of the form
\begin{displaymath}
u_i(z)=z^{k_i+2a_i}p_i(z^5)
\end{displaymath}
for some polynomial $p_i$ with real coefficients.
\end{lemma}
\begin{proof}
Since $u(\gamma z)=g^2u(z)$ there exists, for every $z$, a constant $c(z)\in\cc^*$ such that $u_i(\gamma z)=c(z)\gamma^{2a_i}u_i(z)$ for each $i$.
Then
\begin{displaymath}
c(z)=\frac{u_j(\gamma z)}{\gamma^{2a_j}u_j(z)}
\end{displaymath}
is a degree 0 rational function.
If $u_j(\gamma z_0)=0$ but $u_j(z_0)\neq 0$ then $c(z_0)=0$, a contradiction.
Likewise if $u_j(z_0)=0$ then $u_j(\gamma z_0)=0$.
Thus $u_j(z)$ and $u_j(\gamma z)$ have the same roots, so $c(z)$ is actually a constant, call it $c$.
Since $u_j(0)\neq 0$, it follows that $c=c(0)=\gamma^{-2a_j}$.
Thus
\begin{displaymath}
u_i(\gamma z)=\gamma^{2a_i+k_i}u_i(z).
\end{displaymath}

Let $0\leq k<5$ be an integer and suppose that $p(z)$ is a polynomial such that
\begin{displaymath}
p(\gamma z)=\gamma^{k}p(z).
\end{displaymath}
Let $m$ be such that $p(z)=z^mq(z)$ and $q(0)\neq 0$.
Then
\begin{displaymath}
\gamma^m z^mq(\gamma z)=p(\gamma z)=\gamma^kp(z)=\gamma^k z^m q(z).
\end{displaymath}
Dividing by $z^m$ and plugging in $z=0$ gives $k\equiv m$ mod 5.
It follows that $q(\gamma z)=q(z)$ and so $q(z)=q_1(z^5)$ for some polynomial $q_1$.
Thus $p$ is of the form 
\begin{displaymath}
p(z)=z^kq_2(z^5).
\end{displaymath}

Therefore $u_i$ is of the form
\begin{displaymath}
u_i(z)=z^{2a_i+k_i}p_i(z^5).
\end{displaymath}
The $p_i$'s can be taken to have real coefficients because $u$ maps $\rr$ into $L$ and $L$ consists of the real points of $X$.
\end{proof}

Let $E_0=\int_{\cp^1}\omega$ for $\cp^1$ a line in $\cp^4$.
The next lemma provides a formula for the energy of holomorphic strips.
\begin{lemma}\label{lemma:strips-energy}
Suppose $u_0:S\to X$ is a holomorphic strip such that the corresponding sphere $u:\cc\to X$ is a polynomial in $z\in\cc$ of degree $k$.
Then the energy of $u_0$ is 
\begin{displaymath}
E(u_0)=kE_0/10.
\end{displaymath}
\end{lemma}
\begin{proof}
$u$ is a holomorphic sphere of degree $k$ so $E(u)=kE_0$.
From the formula for $u_k$ in terms of $u_0$ it is clear that $E(u_k)=E(u_0)$.
Therefore $E(u)=10E(u_0)$.
\end{proof}

Finally, recall that if $u_0:\rr\times[0,1]\to X$ is a holomorphic strip then $u_0$ can be translated by $s_0\in\rr$ to get another holomorphic strip $u_0^{s_0}$.
The formula for $u_0^{s_0}$ is $u_0^{s_0}(s,t)=u_0(s+s_0,t)$.
\begin{lemma}\label{lemma:strips-r_translation}
If the domain of $u_0$ is though of as $S_0$ instead of $S$, then the formula for $u_0^{s_0}$ is
\begin{displaymath}
u_0^{s_0}(z)=u(e^{\pi s_0/5}z).
\end{displaymath}
\end{lemma}
\begin{proof}
This follows from the fact that the biholomorphism $S\to S_0$ is $s+it\mapsto e^{\pi(s+it)/5}$,
\end{proof}

\section{Cokernels}
Suppose $u\in W^{1,p}(L,gL)$ is a holomorphic map.
The domain of $u$ is $S=\rr\times[0,1]$ with coordinate $s+it$.
The linearization of $\dbar$ is the operator
\begin{equation}\label{eq:cokernel-linearized_operator}
D_u\dbar=D\dbar:T_uW^{1,p}(L,gL)\to\mathcal E^{0,p}_u.
\end{equation}
(The notation is as in Section \ref{section:moduli_spaces-bott_morse}.)
In this section we develop a method that can be used to calculate the cokernel of $D\dbar$.

Recall that $X$ is equipped with the standard symplectic form $\omega$ and the standard complex structure $J$.
Furthermore, the Levi-Civita connection $\nabla$ is used to define all parallel translations and linearizations.
Also, $W^{1,p}(L,gL)=W^{1,p}_{0,0}(L,gL)$ denotes the space of strips with bottom boundary on $L$ and top boundary on $gL$ and no marked points (other than $\pm\infty$).

\subsection{Calculation of the cokernel}
To start with, consider the operator
\begin{eqnarray}
\nonumber&(u^*\nabla)^{0,1}=\nabla^{0,1}:W^{1,p;\delta}_\lambda(S,u^*TX)\to L^{p;\delta}(S,\Lambda^{0,1}\otimes u^*TX),&\\
\nonumber&\xi\mapsto (\nabla_s\xi+J(u)\nabla_t\xi)\otimes (ds-idt).&
\end{eqnarray}

\begin{lemma}\label{lemma:dual}
The dual of $L^{p;\delta}$ can be identified with $L^{q;\delta}$ ($\frac{1}{p}+\frac{1}{q}=1$).
The pairing is
\begin{eqnarray*}
&\langle\cdot,\cdot\rangle:L^{p;\delta}\times L^{q;\delta}\to \rr,&\\
&\xi,\eta\mapsto \int \xi \eta e^{\delta}.&
\end{eqnarray*}
\end{lemma}
\begin{proof}
The maps 
\begin{eqnarray*}
&&L^{p;\delta}\to L^{p},\,\xi\mapsto\xi e^{\delta/p}\\
&&L^{q;\delta}\to L^{q},\,\eta\mapsto\eta e^{\delta/q},\\
\end{eqnarray*}
are isomorphisms.
$L^q$ is the dual of $L^p$, and the pairing is $(\xi,\eta)\mapsto \int \xi\eta$.
Pulling back the pairing by the above isomorphisms proves the lemma.
\end{proof}

\begin{corollary}
The dual of $L^{p;\delta}(S,\Lambda^{0,1}\otimes u^*TX)$ can be identified with $L^{q;\delta}(S,\Lambda^{0,1}\otimes u^*TX)$.
The pairing is
\begin{equation}\label{eq:cokernel-pairing}
(\xi\otimes (ds-idt),\eta\otimes (ds-idt))\in L^{p;\delta}\times L^{q;\delta}\mapsto \int_S g(\xi,\eta) e^\delta,
\end{equation}
where $g$ is the K\"ahler metric on $TX$.
\end{corollary}
\begin{proof}
$u$ extends continuously to $\overline{\rr}\times[0,1]$.
Choose a trivialization $u^*TX\cong \rr\times[0,1]\times \cc^3$ that extends to $\overline{\rr}\times[0,1]$.
Then the metric $g$ on $u^*TX$ is equivalent to the euclidean metric.
The corollary then follows from the previous lemma.
\end{proof}

\begin{lemma}\label{lemma:cokernel-eta}
Suppose $\eta\otimes (ds-idt)\in L^{q;\delta}(S,\Lambda_S^{0,1}\otimes u^*TX)$ vanishes on $\textrm{\Image }\nabla^{0,1}$.
Then
$$(\nabla_s-J(u)\nabla_t)(\eta e^\delta)=0$$
and $\eta(\cdot,0)\in TL,\eta(\cdot,1)\in TgL$.
\end{lemma}

\begin{proof}
First, suppose that $\xi\in W^{1,p;\delta}_\lambda$ is smooth and has compact support contained in the interior of $S$.
Then
\begin{eqnarray*}
0&=&\int_S \del_s(g(\xi,\eta )e^\delta)+\del_t(g(J\xi,\eta )e^\delta)=\int_S g(\nabla_s\xi,\eta )e^\delta+g(\xi,\nabla_s(\eta e^\delta))\\
&&+g(J\nabla_t\xi,\eta)e^\delta+g(J\xi,\nabla_t(\eta e^\delta))\\
&=&\int_S g((\nabla_s+J\nabla_t)\xi,\eta)e^\delta+\int_Sg(\xi,(\nabla_s-J\nabla_t)(\eta e^\delta))\\
&=&0+\int_Sg(\xi,(\nabla_s-J\nabla_t)(\eta e^\delta)).
\end{eqnarray*}
Thus $(\nabla_s-J(u)\nabla_t)(\eta e^\delta)=0$.

Now, suppose that $\xi$ is smooth and has compact support (i.e. $\xi$ does not necessarily vanish along the top and bottom boundaries of $S$).
Then the first integral above is not $0$ but instead is
\begin{displaymath}
\int_\rr g(J\xi(s,1),\eta(s,1))e^\delta-g(J\xi(s,0),\eta(s,0))e^\delta.
\end{displaymath}
Thus $\eta(s,1)\in TL_1$ and $\eta(s,0)\in TL_0$.
\end{proof}

Let $h=g-i\omega$ be the Hermitian metric on $(TX,J)$.
Let $\eta$ be as in Lemma \ref{lemma:cokernel-eta} and define $\check\eta\in C^\infty(S,u^*T^*X)$ by
$$\check\eta(\xi)=h(\xi,\eta).$$
$\check\eta$ is a section of $\Hom_\cc(u^*TX,\cc)=:u^*T_{\cc}^*X$ (i.e. the complex dual of $u^*TX$ with respect to the complex structure $J$).
\begin{lemma}
$$(\nabla_s+i\nabla_t)(e^\delta\check\eta)=0.$$
\end{lemma}
\begin{proof}
Let $\xi$ be a section of $u^*TX$.
Then
\begin{eqnarray*}
&&(\nabla_s+i\nabla_t)(e^\delta\check\eta)(\xi)=(\del_s+i\del_t)(e^\delta\check\eta(\xi))-e^\delta\check\eta((\nabla_s+J\nabla_t)\xi)=\\
&&h((\nabla_s+J\nabla_t)\xi,e^\delta\eta)+h(\xi,(\nabla_s-J\nabla_t)(e^\delta\eta))-h((\nabla_s+J\nabla_t)\xi,e^\delta\eta)=0.
\end{eqnarray*}
\end{proof}

Since $u$ is holomorphic and $X$ is K\"ahler, the Dolbeault operator on $u^*T_\cc^*X$ is $\nabla^{0,1}$.
Thus the previous lemma says that $e^\delta\check\eta$ is a holomorphic section of $u^*T_\cc^*X$.
Moreover, $\check\eta\in L^{q;\delta}$.

Recall that $u$ can be extended to a map $\cp^1\to X$.
To make the notation clearer, let $u_0$ take the place of $u$, and let $u$ be $u_0$ extended to $\cc P^1$.
So $u_0$ is a map $S=\rr\times[0,1]\to X$ and $u$ is a map $\cp^1\to X$.
Also, let $\eta_0$ take the place of $\eta$, so $\eta_0$ is a section of $u_0^*TX$, and $e^\delta\check\eta_0$ is a section of $u_0^*T_\cc^*X$.

Now let 
$$C=\rr\times[0,10]/(s,0)\sim(s,10)=\rr\times(\rr/10\z).$$
Note that $S$ is contained in $C$ in the obvious way.
The same biholomorphism that takes $S$ to $S_0\subset\cc$ takes $C$ to $\cc^*\subset \cc$.
Thus the map $u$ can be restricted to $C$, and is simply the extension of $u_0$ to $C$.
$e^\delta\check\eta_0$ can be extended to a holomorphic section $e^\delta\check\eta$ of $u^*T_\cc^* X$ over $C$:
For $1\leq k\leq 9$, $k\leq t\leq k+1$ and $\xi$ a tangent vector, define
\begin{equation}
\check\eta_k(s,t)(\xi)=\left\{
\begin{array}{ll}
\check\eta_k(s,t)(\xi)=\overline{\check\eta_0(s,k+1-t)(d\tau^{-1}(dg^{k+1})^{-1}\xi)}&,\textrm{ $k$ odd,}\\
\check\eta_k(s,t)(\xi)=\check\eta_0(s,t-k)((dg^k)^{-1}\xi)&,\textrm{ $k$ even.}
\end{array}
\right.
\end{equation}

\begin{lemma}\label{lemma:cokernel-symmetries}
$\check\eta$ has the following properties:
\begin{enumerate}
\item $\check\eta(s,t)(\xi)=\overline{\check\eta(s,-t)(d\tau \eta)}$, and
\item $\check\eta(s,t+2)(\xi)=\check\eta(s,t)((dg^2)^{-1}\xi)$.
\end{enumerate}
\end{lemma}
\begin{proof}
Properties (1) and (2) are easily seen to hold for $t=0$.
Since $e^\delta\check\eta$ is holomorphic, they must hold for all $t$.
For property (2), suppose $k\leq t\leq k+1$ and $k$ is odd, $0\leq k\leq 9$.
Then $-t=10-t$ mod $10$, and $j\leq 10-t\leq j+1$ with $j=9-k$ even, $0\leq j\leq 9$.
Thus
$$
\overline{\check\eta(s,-t)(d\tau \eta)}=\overline{\check\eta_j(s,10-t)(d\tau\eta)}=\overline{\check\eta_0(s,10-t-j)((dg^j)^{-1}d\tau\eta)}=
$$
$$
\overline{\check\eta_0(s,k+1-t)(dg^{k+1}d\tau\eta)}=\overline{\check\eta_0(s,k+1-t)(d(g^{k+1}\tau)\eta)}=
$$
$$
\overline{\check\eta_0(s,k+1-t)(d(\tau g^{-k-1})\eta)}=\check\eta(s,t)(\eta).
$$
\end{proof}

Let $\zeta=e^\delta\check\eta$.
Then $e^{-\delta}\zeta\in L^{q;\delta}$, that is
\begin{displaymath}
\int_C|\zeta|^qe^{(1-q)\delta}<\infty.
\end{displaymath}
Here the norm is with respect to the Hermitian metric induced on $T_\cc^*X$ by the metric $h$ on $TX$, i.e. $|h(\cdot,v)|:=|v|$.
The next step is to show that $\zeta$ extends to a section on all of $\cp^1$.
First, a lemma is needed.
(The lemma is not obvious because $q<2$.)
\begin{lemma}
Let $D^2$ be the closed unit disc in $\cc$ and let the coordinate be $z=x+iy$.
Let $h:D^2\setminus\sett{0}\to\cc$ be a holomorphic function.
Suppose that
\begin{displaymath}
\int_{D^2}|h(z)|^q|z|^{-2+\epsilon}<\infty
\end{displaymath}
where $\epsilon=\frac{5}{\pi}(q-1)\delta_0$ (i.e. $\epsilon>0$ is small).
Then $h$ extends holomorphically to $D^2$.
\end{lemma}
\begin{proof}
Let $g(z)=zh(z)$.
Then
\begin{displaymath}
\int |g|^q|z|^{-2-q+\epsilon}=\int |h|^q|z|^{-2+\epsilon}<\infty.
\end{displaymath}
Let $r>0$ be small, and $z_0\in D^2$ such that $1>|z_0|>3r$.
By the Cauchy integral formula,
\begin{displaymath}
g(z_0)=\frac{1}{2\pi i}\int_{|z|=1}\frac{g(z)}{z-z_0}dw-\frac{1}{2\pi ir}\int_A\frac{g(z)}{z-z_0}e^{i\theta}dxdy,
\end{displaymath}
where $A=\sett{r<|z|<2r}$.
If the second term approaches $0$ as $r$ approaches $0$ then $g$ is holomorphic.
The second term does indeed approach $0$ because
\begin{eqnarray*}
&&\biggl|\frac{1}{r}\int_A\frac{g(z)}{z-z_0}e^{i\theta}dxdy\biggr|\leq\frac{1}{r}\int_Ar^{-1}|g|=\frac{1}{r}\int_Ar^{-\alpha}|g|r^{\alpha-1}\leq\\
&&\frac{c}{r}\biggl(\int_A r^{-\alpha q}|g|^q\biggr)^{1/q} r^{\alpha-1} r^{2/p}\leq c\biggl(\int_A|z|^{-2-q+\epsilon}|g|^q\biggr)^{1/q} r^{\alpha+\frac{2}{p}-2}\leq c r^{1-\epsilon/q},
\end{eqnarray*}
where $-\alpha q=-2-q+\epsilon$.

It follows that $h$ has at worst a simple pole at the origin.
However,
\begin{displaymath}
\int_{D^2}|z^{-1}|^q|z|^{-2+\epsilon}=\infty
\end{displaymath}
so $h$ cannot have a pole at the origin.
Thus $h$ is holomorphic at $0$.
\end{proof}

\begin{lemma}
$\zeta$ extends to a section on all of $\cp^1$.
\end{lemma}
\begin{proof}
Let $f:C\to \cc\setminus\sett{0}$ be the biholomorphism $f(s+it)=e^{\pi (s+it)/5}=z$.
If $u$ is a function on $C$ then, by the change of variables formula,
\begin{displaymath}
\int_\cc u\circ f^{-1}=\int_Cu|f'|^2.
\end{displaymath}

Let $C_1=f^{-1}(D^2\setminus\sett{0})$.
$C_1$ consists of the set of all $s+it$ with $s\leq 0$.
Note that $|z|=e^{\pi s/5}$ and $|f'(s+it)|^2=\frac{\pi}{5}|z|^2$.
Recall that $\epsilon=\frac{5}{\pi}(q-1)\delta_0$.
Then
\begin{displaymath}
\int_{D^2}|\zeta\circ f^{-1}|^q|z|^{-2+\epsilon}=c\int_{C_1}|\zeta|^qe^{(1-q)\delta_0|s|}<\infty.
\end{displaymath}
By the previous lemma it follows that $\zeta$ is holomorphic at $0$.
A similar argument shows that $\zeta$ is holomorphic at $\infty$.
\end{proof}

Note that $\zeta$ does not have to vanish at $0$, because $|z|^{-2+\epsilon}$ is in $L^1$ on $D^2$.
However, if the full linearized operator (\ref{eq:cokernel-linearized_operator}) is considered, then values of $\zeta$ at $0$ and $\infty$ are imposed.
Indeed, the full linearized operator can be identified with
\begin{eqnarray}\label{eq:cokernel-linearized_operator_2}
&D\dbar:T_pR_h\oplus T_qR_{h'}\oplus W^{1,p;\delta}_\lambda(S,u_0^*TX)\to L^{p;\delta}(S,\Lambda^{0,1}\otimes u_0^*TX),&\\
\nonumber&(U,V,\xi)\mapsto (\nabla_s+J\nabla_t)(\chi_-U+\chi_+V+\xi)\otimes (ds-idt).
\end{eqnarray}
Here $\chi_-$ and $\chi_+$ are cutoff functions such that $\chi_-=1$ near $-\infty$ and 0 away from $-\infty$, and $\chi_+=1$ near $\infty$ and 0 away from $\infty$.
Suppose that $\eta\otimes(ds-idt)\in L^{q;\delta}(S,\Lambda^{0,1}\otimes u_0^*TX)$ is such that $\eta\otimes (ds-idt)$ vanishes on the image of $D\dbar$, where the pairing is as in (\ref{eq:cokernel-pairing}).
$\eta\otimes (ds-idt)$ also vanishes on the image of $\nabla^{0,1}$, so $\eta$ satisfies the same properties as $\eta$ above.
\begin{lemma}
\begin{displaymath}
\lim_{s\to\pm\infty}|e^\delta\eta|=0.
\end{displaymath}
\end{lemma}
\begin{proof}
As before, for $\xi$ a section of $u_0^*TX$, a straightforward calculation shows that
\begin{displaymath}
\int_S\del_s(g(\xi,\eta)e^\delta)+\del_t(g(J\xi,\eta)e^\delta)=\int_Sg((\nabla_s+J\nabla_t)\xi,\eta)e^\delta+g(\xi,(\nabla_s-J\nabla_t)(e^\delta\eta))=0.
\end{displaymath}
Moreover, by the boundary conditions of $\xi$ and $\eta$, it follows that
\begin{displaymath}
\int_S \del_t(g(J\xi,\eta)e^\delta)=\int_\rr g(J\xi(s,1),\eta(s,1))e^\delta-\int_\rr g(J\xi(s,0),\eta(s,0))e^\delta=0.
\end{displaymath}
Let $W_\pm=(e^\delta\eta)(\pm\infty,\cdot)$, which exists by the previous lemma.
Note that $W_-\in T_pR_h=T_pL\cap T_pgL$ and $W_+\in T_qL\cap T_qgL$.
Then
\begin{eqnarray*}
\int_S \del_s(g(\xi+\chi_-U+\chi_+V,\eta)e^\delta)+\del_t(g(J(\xi+\chi_-U-+\chi_+V),\eta)e^\delta)&=&\\
\int_S g((\nabla_s+J\nabla_t)(\xi+\chi_-U+\chi_+V),\eta)e^\delta+\int_Sg(\xi+\chi_-U+\chi_+V,(\nabla_s-J\nabla_t)(e^\delta\eta)).
\end{eqnarray*}
The first integral on the right hand side is $0$ by assumption, and the second integral is $0$ because $e^\delta\eta$ satisfies the same properties as before.
On the other hand, the integral on the left hand side can be rewritten using Stoke's theorem.
Combining these facts leads to the equation
\begin{displaymath}
0=\int_0^1g(V,W_+)dt-\int_0^1g(U,W_-)dt.
\end{displaymath}
(The top and bottom boundary terms in Stoke's theorem drop out because again $e^\delta\eta$ satisfies the same properties as before.)
Since $U\in T_pR_h$ and $V\in T_qR_{h'}$ are arbitrary, it follows that $W_\pm=0$.
Thus $(e^\delta\eta)(\pm\infty,\cdot)=0$.
\end{proof}

\begin{corollary}
Suppose $\eta$ vanishes on the image of $D\dbar$ and $\zeta=e^\delta\check\eta$.
Then, viewing the domain of $\eta$ as $\cc\cup\infty$, $\zeta(0)=\zeta(\infty)=0$.
\end{corollary}

\begin{prop}\label{prop:cokernel-main_prop}
The cokernel of the operator $D\dbar$ in (\ref{eq:cokernel-linearized_operator}) can be identified with the elements $\zeta$ of $H^0(\cc P^1,u^*T^*X)$ that satisfy the following properties:
\begin{enumerate}
\item $\zeta(0)=\zeta(\infty)=0$,
\item $\zeta(z)(\xi)=\overline{\zeta(\bar z)(d\tau \xi)},$ and
\item $\zeta(\gamma z)(\xi)=\zeta(z)(dg^{-2}\xi)$.
\end{enumerate}
\end{prop}
\begin{proof}
Given $\eta\otimes (ds-idt)$ in the cokernel of $D\dbar$, let $\zeta=e^\delta\check\eta$.
The first item is the previous lemma, and the other two items are Lemma \ref{lemma:cokernel-symmetries} (the converse of Lemma \ref{lemma:cokernel-symmetries} can be proven in the same way as Lemma \ref{lemma:strips-symmetries_converse}).
\end{proof}

Properties (2) and (3) in the proposition can be interpreted in another way:
Namely, consider the $\z_5$ action on $H^0(\cp^1,u^*T^*X)$ given by $\zeta\mapsto\tilde\zeta$ where
\begin{displaymath}
\langle \tilde\zeta(z),\xi\rangle =\langle \zeta(\gamma z),dg^2\cdot \xi\rangle.
\end{displaymath}
Then $\zeta$ satisfies property (3) if and only if $\zeta=\tilde\zeta$, i.e. if and only if $\zeta$ is fixed by the $\z_5$ action.
Similarly, $\zeta$ satisifies property (2) if and only if it is fixed by the obvious $\z_2$ action.

\subsection{Remark}
The obstruction bundle of the holomorphic sphere $u:\cc P^1\to X$ is $H^1(\cc P^1,u^*T_X)$. 
By Serre duality
$$H^1(\cc P^1,u^*T_X)\cong H^0(\cc P^1,\Omega^1\otimes u^*T^*_ X)^*.$$
Since $\Omega^1=\mathcal O(-2)$, the latter group can be identified with the holomorphic sections of $u^*T_X^*$ that vanish at $0$ and $\infty$.
Thus elements of the obstruction bundle of the strip are in bijective correspondence to elements of the obstruction bundle of the sphere with the symmetries given above.

\section{Conjugation of strips and orientations}
In this section we construct an involution (which we call conjugation) on the space of strips.
The effect of this operation on the orientation of the strips is then worked out.

\subsection{The conjugation construction}
Let $W^{1,p}(L,gL)=W^{1,p}_{0,0}(L,gL)$.
Let $u_0(z)$ be a strip, not necessarily holomorphic, and view the domain of $u_0$ as $S_0\subset\cc$.
Define the conjugate $\tilde u_0$ of $u_0$ as
$$\tilde u_0(z)=g\tau u_0(\sigma\bar z).$$
If $u_0$ is holomorphic then $\tilde u_0$ is also holomorphic.
$\tilde u_0$ satisfies the same boundary conditions as $u_0$.
Let $\rho:W^{1,p}(L,gL)\to W^{1,p}(L,gL)$ denote this map.
Note that $\rho$ is an involution, because
\begin{displaymath}
\rho^2u_0(z)=\rho g\tau u_0(\sigma\bar z)=g\tau g\tau u_0(\sigma\bar\sigma z)=gg^{-1}\tau\tau u_0(z)=u_0(z).
\end{displaymath}
If the domain of $u_0$ is viewed as $S=\rr\times[0,1]$ with coordinates $(s,t)$ then $\tilde u_0(s,t)=g\tau u_0(s,1-t)$.
Note that $\tilde u_0(\pm\infty,\cdot)=u_0(\pm\infty,\cdot)$.
The remainder of this section is devoted to showing that $\rho$ induces a map on the moduli space of strips.

Let $p=u_0(-\infty,\cdot)\in R_h$ and $q=u_0(+\infty,\cdot)\in R_{h'}$.
Differentiating $\rho$ gives a map
$$
\rho_*:T_{u_0}W^{1,p}(L,gL)\to T_{\tilde u_0}W^{1,p}(L,gL).
$$
Recall that 
\begin{equation}\label{eq:conjugation-tangent_space}
T_pR_h\oplus T_qR_{h'}\oplus W^{1,p;\delta}_\lambda(S,u_0^*TX)\cong T_{u_0}W^{1,p}(L,gL),
\end{equation}
and similarly for $T_{\tilde u_0}$.
Recall that this is because $W^{1,p;\delta}$ consists of sections of class $W^{1,p}$ with weight $\delta$.
The weight forces the sections to decay to $0$ at $\pm\infty$, so $T_pR_h\oplus T_qR_{h'}$ has to be added to $W^{1,p;\delta}$ to allow the endpoints of the maps to vary in $R_h$ and $R_{h'}$.

The isomorphism (\ref{eq:conjugation-tangent_space}) can be described explicitly.
Let $\chi_p:\rr\to [0,1]$ be a smooth decreasing function such that $\chi_p(s)=1$ for $s\leq -R-1$ and $\chi_p(s)=0$ for $s\geq -R$, where $R$ is a large positive constant.
Let $\chi_q:\rr\to [0,1]$ be a smooth increasing function such that $\chi_q(s)=1$ for $s\geq R+1$ and $\chi_q(s)=0$ for $s\leq R$.
%Also, assume that $|\dot\chi_p|,|\dot\chi_q|\leq 2$.
The isomorphism is then
\begin{displaymath}
(U,V,\xi)\mapsto \chi_p(s)\Par_{u_0} U+\chi_q(s)\Par_{u_0} V+\xi,
\end{displaymath}
where $\Par_{u_0} U(s,t)$ is the parallel translate of $U$ from $T_{u_0(-\infty,t)}X$ to $T_{u_0(s,t)}X$ along the obvious horizontal path.

\begin{lemma}
The map $\rho_*$ is
\begin{displaymath}\label{eq:rho_push_forward_on_tangent_space_of_maps}
\rho_*:T_pR_h\oplus T_qR_{h'}\oplus W^{1,p;\delta}_\lambda(S,u_0^*TX)\to T_pR_h\oplus T_qR_{h'}\oplus W^{1,p;\delta}_\lambda(S,\tilde u_0^*TX),
\end{displaymath}
$$
(U,V,\xi(s,t))\mapsto (U,V,dg\cdot d\tau\cdot \xi(s,1-t)).
$$
\end{lemma}
\begin{proof}
Since $g$ and $\tau$ are isometries, $g\tau \exp_xY=\exp_{g\tau x}(dg\cdot d\tau \cdot Y)$ for $x\in X$, $Y\in T_xX$.
Likewise $g\tau(\Par_\gamma Y)=\Par_{g\tau\gamma} (dg\cdot d\tau\cdot Y)$ for any path $\gamma$.
Since $R_h\subset L\cap gL=L\cap \textrm{Fix}(g)$, $dg$ and $d\tau$ act trivially on $TR_h$.
The same is true for $TR_{h'}$.
It follows that
\begin{eqnarray*}
 &\rho(\exp_{u_0}( \chi_p\Par_{u_0} U+ \chi_q\Par_{u_0} V+ \xi))(s,t)=&\\
&\exp_{\tilde u_0(s,t)}( dg\cdot d\tau\cdot \chi_p(s)\Par_{u_0} U(s,1-t)+ dg\cdot d\tau \cdot \chi_q(s)\Par_{u_0} V(s,1-t)+ &\\
&dg\cdot d\tau\cdot \xi(s,1-t))=&\\
&\exp_{\tilde u_0(s,t)}( \chi_p(s)\Par_{\tilde u_0} U(s,t)+ \chi_q(s)\Par_{\tilde u_0} V(s,t)+ dg\cdot d\tau\cdot \xi(s,1-t)).
\end{eqnarray*}

A simple calculation now completes the proof:
\begin{eqnarray*}
\rho_*(U,V,\xi)(s,t)&=&\frac{\partial}{\partial\lambda}\biggl|_{\lambda=0}\rho \exp_{u_0}(\lambda(U,V,\xi))(s,t)\\
&=&\frac{\partial}{\partial\lambda}\biggl|_{\lambda=0}\exp_{\tilde u_0(s,t)}( \lambda\chi_p(s)\Par_{\tilde u_0} U(s,t)+ \lambda\chi_q(s)\Par_{\tilde u_0} V(s,t)\\&&+ \lambda dg\cdot d\tau\cdot \xi(s,1-t))\\
&=& \chi_p(s)\Par_{\tilde u_0} U(s,t)+ \chi_q(s)\Par_{\tilde u_0} V(s,t)+  dg\cdot d\tau\cdot \xi(s,1-t))\\
&=&(U,V,dg\cdot d\tau\cdot\xi(s,1-t)).
\end{eqnarray*}
\end{proof}

$\rho$ lifts to a bundle map of $\mathcal E^p$ (also denoted $\rho_*$)
$$
\rho_*:\mathcal E^{p}_{u_0}\to \mathcal E^{p}_{\tilde u_0},
$$
$$
\eta(s,t)\otimes (ds-idt) \mapsto dg\cdot d\tau \cdot \eta(s,1-t)\otimes (ds-idt).
$$

Consider the linearized $\dbar$ operator
$$
D_{u_0}\dbar:T_{u_0}W^{1,p}(L,gL)\to \mathcal E^p_{u_0}.
$$
By Proposition 3.1.1 in \cite{ms}, the formula for $D_{u_0}\dbar$ is
\begin{displaymath}
D_{u_0}\dbar:T_pR_h\oplus T_qR_{h'}\oplus W^{1,p;\delta}_\lambda(S,u_0^*TX)\to L^p(S,\Lambda^{0,1}\otimes u_0^*TX),
\end{displaymath}
$$
(U,V,\xi)\mapsto (\frac{1}{2}(\nabla_s+J(u_0)\nabla_t)(\xi+\chi_p\Par_{u_0}U+\chi_q\Par_{u_0}V))\otimes (ds-idt).
$$
%Here we view $u,v$ as constant sections defined near $s=\pm\infty$.
Similarly for $D_{\tilde u_0}$.

\begin{lemma}\label{lemma:conjugation-det}
The following diagram commutes:
$$
\begin{array}{ccc}
T_{u_0}W^{1,p}(L,gL) & \xrightarrow{D_{u_0}\dbar} & \mathcal E^p_{u_0} \\
\biggr\downarrow \rho_* & & \biggr\downarrow \rho_*\\
T_{\tilde u_0} W^{1,p}(L,gL) & \xrightarrow{D_{\tilde u_0}\dbar} & \mathcal E^p_{\tilde u_0}.
\end{array}
$$
Hence $\rho_*$ induces an isomorphism
\begin{equation}\label{eq:det_rho}
\Det(\rho_*):\Det(D_{u_0}\dbar)\to \Det(D_{\tilde u_0}\dbar).
\end{equation}
\end{lemma}
\begin{proof}
The commutativity is a straightforward calculation, similar to the proof of the previous lemma, using the fact that $g$ and $\tau$ are isometries and $\nabla$ is the Levi-Civita connection.
The second statement follows from commutativity of the diagram.
\end{proof}

\begin{prop}\label{lemma:conjugation-involution_kuranishi_spaces}
$\rho_*$ induces an involution of Kuranishi spaces $$\mathcal M(L,gL)\to \mathcal M(L,gL).$$
\end{prop}
\begin{proof}
See Section \ref{section:appendix-involution} for the definition of an involution of a Kuranishi space.
Recall that $\mathcal M(L,gL)$ is the (Gromov) compactification of $\mathcal M^{reg}(L,gL)$ (actually, it is only compact when restricted to strips with energy bounded by any fixed constant $E$).
Recall also that
\begin{displaymath}
\mathcal M^{reg}(L,gL)=\dbar^{-1}(0)/\rr,
\end{displaymath}
where the $\rr$ action is the natural translation of strips.
Since $\rho_*:W^{1,p}(L,gL)\to W^{1,p}(L,gL)$ commutes with the $\rr$ action, the additional detail of needing to mod out by the $\rr$ action can be ignored.

First, $\rho_*:\mathcal M(L,gL)\to\mathcal M(L,gL)$ needs to be defined as a homeomorphism such that $\rho_*^2=0$.
If $u_0\in\mathcal M^{reg}(L,gL)$, then the definition is the same as before, namely
\begin{displaymath}
(\rho_*u_0)(s,t)=\tilde u_0(s,t)=g\tau u_0(s,1-t).
\end{displaymath}

Now assume $u_0$ is in the boundary of $\mathcal (L,gL)$, so the domain of $u_0$ is a semi-stable curve $\Sigma$ with at least two components.
One of the components must be a strip, call it $S$, and let the union of the other components be $\Sigma_1$.
Then $\Sigma=S\cup \Sigma_1$.
The map $\rho_*$ will be constructed by induction on the number of domain components.
Start the induction with two domain components, $S$ and $\Sigma_1$.
There are three cases to consider, depending on whether $\Sigma_1$ is a sphere, a disc, or a strip.

Case 1: Assume $\Sigma_1$ is a sphere.
$\Sigma_1$ is attached to $S$; assume the point $0=[1:0]\in\cp^1=\Sigma_1$ attaches to the point $(s_0,t_0)\in \textrm{Int}(S)$.
Let $u_{\cp^1}=u_0|_{\Sigma_1}$ and let $u_S=u_0|_S$.
Let $\tilde\Sigma$ be the stable curve with $0\in\cp^1$ attached to $(s_0,1-t_0)\in S$.
Define $\tilde u_0:\tilde\Sigma\to X$ by 
\begin{displaymath}
\tilde u_0|_S=\tilde u_S,\, \tilde u_0|_{\cp^1}(z)=g\tau u_{\cp^1}(\bar z).
\end{displaymath}
Then $\tilde u_0\in \mathcal (L,gL)$.
Define $\rho_*(u_0)$ to be $\tilde u_0$.
Note that $g\tau g\tau=id$, so $\rho_*(\tilde u_0)=u_0$.

Case 2: Assume $\Sigma_1$ is a disc. $\Sigma_1$ is attached to $S$; assume without loss of generality that $1\in D^2=\Sigma_1$ attaches to $(s_0,0)\in S$.
Let $u_{D^2}=u_0|_{\Sigma_1}$ and let $u_S=u_0|_S$.
Let $\tilde\Sigma$ be the stable curve with $1\in D^2$ attached to $(s_0,1)\in S$.
Define $\tilde u_0:\tilde \Sigma\to X$ by
\begin{displaymath}
\tilde u_0|_S=\tilde u_S,\, \tilde u_0|_{D^2}(z)=g\tau u_{D^2}(\bar z).
\end{displaymath}
Then $\tilde u_0\in\mathcal (L,gL)$.
Define $\rho_*(u_0)$ to be $\tilde u_0$.

Case 3: Assume $\Sigma_1$ is a strip. The way to define $\rho_*$ is obvious in this case: simply conjugate both strips.

To see that $\rho_*$ is continuous, the bubbling off procedure needs to be examined.
The details will be given for Case 1; the other cases are similar.
Let $u_n\in\mathcal M^{reg}(L,gL)$ be a sequence of strips that converges to the boundary point $u\in\mathcal M(L,gL)$ that is of the form given in Case 1.
For appropriate $\epsilon_n\to 0$ and $R_n\to \infty$ , let $\phi_n$ be the obvious biholomorphim from an $\epsilon_n$ ball $B_{\epsilon_n}(s_0,t_0)$ about $(s_0,t_0)$ to the ball $B_{R_n}(0)$ of radius $R_n$ centered at $0\in\cc$, possibly also composed with a rotation.
Then, on the subset $\cc\subset\cp^1$, the sphere bubble $u_{\cp^1}$ is
\begin{displaymath}
u_{\cp^1}(z)=\lim_{n\to\infty} u_n\circ\phi_n^{-1}(z).
\end{displaymath}
Then $\tilde u_n=\rho_*u_n$ forms the sphere bubble $\tilde u_{\cp^1}=\rho_* u_{\cp^1}$.
%For points $(s,t)$ in $S$ away from $(s_0,t_0)$, it is clear that $\rho_*u_n(s,t)$ converges to $\rho_*u(s,t)$.
It follows that $\rho_*$ is continuous.
Since $\rho_*$ is its own inverse, it follows that $\rho_*$ is a homeomorphism.

To complete the induction, it needs to be shown how to define $\rho_*$ when another component $\Sigma_2$ is added to the domain $\Sigma$.
The details of this will be skipped, as it is fairly obvious how to proceed by mimicing the constructions above.
The one subtle point is this: If a disc bubble has other discs bubbling out of it, then conjugating the disc will reverse the cyclic order of the points where the other discs bubble out.
To correct for this, the conjugation has to be composed with a permutation to correct the order of the marked points.
This will have an effect on the sign of $\rho_*$; however, this effect on the sign only comes into play on codimension two and greater boundaries.
The sign that is calculated in the next section will be on the interior of the moduli space only, hence this correction will have no effect on it.

Next, it needs to be shown that $\rho_*$ lifts to an involution of the Kuranishi structure on $\mathcal M(L,gL)$.
First, consider the interior of the moduli space, namely $\mathcal M^{reg}(L,gL)$.
Let $u$ be a strip.

Case 1: $u$ is not a fixed point of $\rho_*$. 
Recall from Section \ref{section:moduli_spaces-kuranishi_structure_interior} that a Kuranishi chart centered at $u$ is of the form
\begin{displaymath}
(V,E,1,i,\dbar|V).
\end{displaymath}
Here, $E_u$ is a complement of $D_u\dbar$ and $E$ is the bundle obtained by parallel translating $E_u$ to a neighborhood of $u$, and $V$ is a neighborhood of $u$ inside $\dbar^{-1}(E)$.
Since $\rho_*u\neq u$, it may be assumed that $\rho_*V\cap V=\emptyset$.
Then $(\rho_*V,\rho_*E,1,i,\dbar|\rho_*V)$ is a Kuranishi chart about $\rho_*u$.
This uses the fact that $\rho_*\circ\dbar=\dbar\circ\rho_*$ (by Lemma \ref{lemma:conjugation-det}), and $\rho_*E$ is $E_{\rho_*u}$ parallel translated to a neighborhood of $\rho_*u$ (this is true because the formula for $\rho_*$ involves $g$ and $\tau$, which are isometries).

Case 2: $u$ is a fixed point $\rho_*$.
Choose a complement $E_u$ of $D_u\dbar$ such that $\rho_*E_u=E_u$ (for example, replace $E_u$ with $E_u+\rho_* E_u$ if necessary).
Since $\dbar$ commutes with $\rho_*$, the neighborhood $V$ of $u$ in $\dbar^{-1}(E)$ may be assumed to be $\rho_*$-invariant.
Then the Kuranishi chart is mapped to itself by the $\rho_*$ action.

This completes the proof that $\rho_*$ induces an involution of the Kuranishi structure on $\mathcal M^{reg}(L,gL)$.
To see that $\rho_*$ induces an involution of Kuranishi structure on $\mathcal M(L,gL)$ is more complicated; the additional complication is that Kuranishi charts are defined near boundary points of the moduli space by resolving nodal points of the domains and then using a gluing process. 
It needs to be checked that conjugation is compatible with these operations.
It is stated in \cite{foooasi} Theorem 4.9 that this is indeed the case.
See also \cite{sol} for similar arguments.

\end{proof}
\subsection{The effect on orientation}\label{section:conjugation-orientation}
The goal of this section is to determine if $\Det(\rho_*)$ in Lemma \ref{lemma:conjugation-det} preserves or reverses orientation.
$\mathcal M(L,gL)$ is not necessarily orientable, so more precisely the following is what is meant:
By Lemma \ref{lemma:local_systems-identification}, if $u$ is a strip that connects $p\in R_h$ to $q\in R_{h'}$, then there is a canonical isomorphism of oriented vector spaces
\begin{equation}\label{eq:conjugation-orientation}
\Det(D_u\dbar)\cong ev_{-\infty}^*\Theta_{R_h}^-\otimes ev_{+\infty}^*\Theta_{R_{h'}}^+.
\end{equation}
Therefore, since $u$ and $\tilde u$ both begin and end at the same points, $\Det(\rho_*)$ can be viewed as an automorphism of $ev_{-\infty}^*\Theta_{R_h}^-\otimes ev_{+\infty}^*\Theta_{R_{h'}}^+.$
The goal then is to determine if this automorphism is homotopic to $1$ or $-1$.

$L\cong\rp^3$, so the tangent bundle of $L$ is trivial.
Fix a trivialization $TL\to L\times \rr^3$.
$TL$ inherits an orientation and spin structure from this trivialization (namely, the pullbacks of the standard orientation and spin structure of $L\times\rr^3$).
Also, the trivialization gives a section
\begin{equation}\label{eq:conjugation-frame}
Fr:L\to Fr(TL)
\end{equation}
of the frame bundle $Fr(TL)$ of $L$.
Equip $gL$ with the push forward trivialization, spin structure, and orientation.
Then $gL$ is equipped with a section of its frame bundle as well, namely the frame at $x\in gL$ is $dg\cdot Fr(g^{-1}x)$.

The isomorphism (\ref{eq:conjugation-orientation}) is defined as follows:
First, by a gluing theorem there is a canonical isomorphism (for all paths $\lambda_p$ and $\lambda_q$)
\begin{eqnarray}\label{eq:glueing_isomorphism}
\Det(\dbar_{\lambda(\lambda_p,u,\lambda_q)})&\cong& (\Det(\dbar_{\lambda_p\oplus TR_{h}})\otimes \Det(T_pR_{h}))\times_{T_pR_{h}}\\\nonumber&&\Det(D_{u}\dbar)\times_{T_qR_{h'}}(\Det(T_qR_{h'})\otimes \Det(\dbar_{\lambda_q\oplus TR_{h'}}))
\end{eqnarray}
Likewise for $\tilde u$.
The notation is
\begin{itemize}
\item $\lambda_p:[0,1]\to \Lambda^{ori}(T_pX)$ with $\lambda_p(0)=T_pL_0$ and $\lambda_p(1)=T_pgL_0$,
\item $\lambda_q:[0,1]\to \Lambda^{ori}(T_qX)$ with $\lambda_q(0)=T_qL_0$ and $\lambda_q(1)=T_qgL_0$,
\item $\dbar_{\lambda(\lambda_p,u,\lambda_q)}$ is the Cauchy-Riemann operator on the glued domain $$Z_-\#(\rr\times[0,1])\#Z_+$$ (which is isomorphic to the disc $D^2$),
\item $\lambda(\lambda_p,u,\lambda_q)$ denotes the Lagrangian boundary condition on the domain (the left semi-circle boundary is $\lambda_p$, the right semi-circle boundary is $\lambda_q$, and the top and bottom boundaries are $u^*TL_0$ and $u^*TL_1$, respectively),
\item similarly for $\dbar_{\lambda(\lambda_p,\tilde u,\lambda_q)}$,
\item $\dbar_{\lambda_p\oplus TR_{h'}}:W^{1,p;\delta}_\lambda(Z_-,T_pX)\to L^{p;\delta}(Z_-,\Lambda^{0,1}\otimes T_pX)$ is the Cauchy-Riemann operator with weights, and
\item similarly for $\dbar_{\lambda_p\oplus TR_h}$.
\end{itemize}
Therefore an orientation of $\Det(D_{\tilde{u}}\dbar)$ is determined by orientations of $\Det(\dbar_{\lambda_p\oplus TR_h})$, $\Det(\dbar_{\lambda_q\oplus TR_{h'}})$, and $\Det(\dbar_{\lambda(\lambda_q,u,\lambda_p)})$. 
(The orientation does not depend on $T_pR_h$ or $T_qR_{h'}$ because both of these vector spaces appear twice in the formula.)
Choose trivializations of the paths $\lambda_p$ and $\lambda_q$ that extend the given frames at $p$ and $q$ of $TL_0$ and $TL_1$.
This data gives canonical isomorphisms (see the discussion at the end of Section 8.1; also see \cite{fooo} Section 41):
\begin{eqnarray}
\nonumber\Det(\dbar_{\lambda(\lambda_q,u,\lambda_p)})&\cong& \rr,\\
\label{eq:conjugation-canonical_isomorphisms}\Det(\dbar_{\lambda_p,Z_-})&\cong & \Theta_{R_h,p}^-,\\
\nonumber\Det(\dbar_{\lambda_q,Z_+})&\cong & \Theta_{R_{h'},q}^+.
\end{eqnarray}
Therefore an orientation of $\Theta_{R_h,p}^-\otimes\Theta_{R_{h'},q}^+$ determines an orientation of $\Det(D_u\dbar)$, and hence determines the isomorphism (\ref{eq:conjugation-orientation}) (up to homotopy).

%Therefore, to determine the sign of $\Det(\rho_*)$, it is equivalent to compute the sign of 
%$$\Det(\rho_*):\Det(D_{u}\dbar)\to\Det(D_{\tilde{u}}\dbar)$$
%when the vector spaces $\Det(D_{u}\dbar)$ and $\Det(D_{\tilde{u}}\dbar)$ are given orientations by choosing the same paths $\lambda_p$, $\lambda_q$ for both of them and the same orientations of $\Det(\dbar_{\lambda_p\oplus TR_h})$, $\Det(\dbar_{\lambda_q\oplus TR_{h'}})$ for both of them and using the isomorphism (\ref{eq:conjugation-orientation}).

To make calculations easier, special paths $\lambda_p$ and $\lambda_q$ will be used.
Let $x$ be $p$ or $q$.
$dg$ and $d\tau$ act on $T_xX$, and $dg$ is a unitary transformation.
Furthermore, $dg\circ d\tau=d\tau\circ dg^{-1}$, so by Lemma \ref{lemma:index_theory-linear_algebra_3} there exists a basis of $T_xL$ such that $dg$ is diagonal.
Let $dg_x$ be the action on $T_xX$, and let the eigenvalues be $\gamma^{a_x},\gamma^{b_x},\gamma^{c_x}$.
Then, with respect to an appropriate basis of $T_xL$, 
\begin{displaymath}
dg_x=\left[
\begin{array}{ccc}
e^{2\pi i a_x/5}&0&0\\
0&e^{2\pi i b_x/5}&0\\
0&0&e^{2\pi i c_x/5}
\end{array}
\right].
\end{displaymath}
Let $dg_x^t$ be the $t^{th}$ power of $dg_x$, defined in the obvious way.
Then let
\begin{eqnarray*}
\lambda_p(t)&=&dg_p^t\cdot T_pL,\\
\lambda_q(t)&=&dg_q^t\cdot T_qL.
\end{eqnarray*}
Note that 
\begin{eqnarray}\label{eq:conjugation-end_cap_paths_preserved}
dg\cdot d\tau\ \cdot\lambda_x(1-t)&=&dg\cdot d\tau \cdot dg_x^{1-t}\cdot T_xL=dg \cdot dg_x^{t-1}\cdot T_xL\\
\nonumber&=&dg_x^t\cdot T_xL=\lambda_x(t).
\end{eqnarray}
Note also that if $T_xL$ is viewed as an oriented Lagrangian subspace then $\lambda_x$ can be interpreted as a path of oriented Lagranian subspaces.
Moreover, recall that a basis $Fr(x)$ of $T_xL$ is given by (\ref{eq:conjugation-frame}), so a canonical trivialization of $\lambda_x$ is given by 
\begin{equation}\label{eq:end_cap_triv}
dg_x^t\cdot Fr(x).
\end{equation}
The basis given to $\lambda_x(1)=T_xgL$ in this way agrees with the basis given by $g_*Fr(x)$.
%It may be assumed that the trivialization of $\lambda_x(1)=T_xL_1$ given in this way agrees with the trivialization of $T_xL_1$ given by the spin structure on $L_1$ (if it does not, homotope the trivialization given by the spin structure so that it does agree).

Define the operator
\begin{equation}\label{eq:left_cap_operator}
D_-:T_pR_h\oplus W^{1,p;\delta}_{\lambda_p\oplus TR_h}(Z_-,T_pX)\to L^{p;\delta}(Z_-,T_pX),
\end{equation}
$$
(V,\xi)\mapsto\dbar_{\lambda_p\oplus TR_h}(\chi_-\cdot V+\xi),
$$
where $\chi_-$ is a cutoff function such that
\begin{itemize}
\item $\chi_-$ depends only on $s$ and is increasing,
\item $\chi_-(s)=0$ for $s\leq R-1$, and
\item $\chi_-(s)=1$ for $s\geq R$.
\end{itemize}
There is an evaluation map
$$
ev_{+\infty}:T_pR_h\oplus W^{1,p;\delta}_{\lambda_p\oplus TR_h}(Z_-,T_pX)\to T_pR_h,
$$
$$(V,\xi)\mapsto V.$$
Furthermore, there is a canonical isomorphism
$$\Det(D_-)\cong \Det(T_pR_h)\otimes \Det(\dbar_{\lambda_p\oplus TR_h}).$$

Similarly, define the operator
\begin{equation}\label{eq:right_cap_operator}
D_+:W^{1,q;\delta}_{\lambda_q\oplus TR_{h'}}(Z_+,T_qX)\oplus T_qR_{h'}\to L^{p;\delta}(Z_+,T_qX),
\end{equation}
$$
(\xi,V)\mapsto\dbar_{\lambda_q\oplus TR_{h'}}(\xi+\chi_+V),
$$
and evaluation map
$$
ev_{-\infty}:W^{1,p;\delta}_{\lambda_q\oplus TR_{h'}}(Z_+,T_qX)\oplus T_qR_{h'}\to T_qR_{h'},
$$
$$(\xi,V)\mapsto V.$$
Furthermore, there is a canonical isomorphism
$$\Det(D_+)\cong \Det(\dbar_{\lambda_q\oplus TR_{h'}})\otimes \Det(T_qR_{h'}).$$

Now consider the following commutative diagrams:
\begin{equation}\label{eq:conjugation-involution_1}
\begin{array}{ccc}
 W^{1,p}_\lambda(Z_-\# S \# Z_+,u^*TX) & \xrightarrow{\dbar_{\lambda(\lambda_q,u,\lambda_p)}} & L^p(Z_-\# S \# Z_+,\Lambda^{0,1}\otimes u^*TX) \\
 \downarrow & & \downarrow \\
 W^{1,p}_\lambda(Z_-\# S \# Z_+,\tilde u^*TX) & \xrightarrow{\dbar_{\lambda(\lambda_q,\tilde u,\lambda_p)}} & L^p(Z_-\# S \# Z_+,\Lambda^{0,1}\otimes \tilde u^*TX) ,
\end{array}
\end{equation}

\begin{equation}\label{eq:conjugation-involution_2}
\begin{array}{ccc}
 T_qR_{h'}\oplus W^{1,p;\delta}_{\lambda_q\oplus TR_{h'}}(Z_+,T_qX) & \xrightarrow{D_+} & L^{p;\delta}(Z_+,\Lambda^{0,1}\otimes T_qX) \\
 \downarrow & & \downarrow \\
 T_qR_{h'}\oplus W^{1,p;\delta}_{\lambda_q\oplus TR_{h'}}(Z_+,T_qX) & \xrightarrow{D_+} & L^{p;\delta}(Z_+,\Lambda^{0,1}\otimes T_qX) ,
\end{array}
\end{equation}

\begin{equation}\label{eq:conjugation-involution_3}
\begin{array}{ccc}
T_pR_h\oplus W^{1,p;\delta}_{\lambda_p\oplus TR_{h}}(Z_-,T_pX) & \xrightarrow{D_-} & L^{p;\delta}(Z_-,\Lambda^{0,1}\otimes T_pX) \\
 \downarrow & & \downarrow \\
T_pR_h\oplus W^{1,p;\delta}_{\lambda_p\oplus TR_{h}}(Z_-,T_pX) & \xrightarrow{D_-} & L^{p;\delta}(Z_-,\Lambda^{0,1}\otimes T_pX) ,
\end{array}
\end{equation}

\begin{equation}\label{eq:conjugation-involution_4}
\begin{array}{ccc}
 T_pR_h\oplus T_qR_{h'}\oplus W^{1,p;\delta}_\lambda(S,u^*TX) & \xrightarrow{D_{u}\dbar} & L^{p;\delta}(S,\Lambda^{0,1}\otimes u^*TX) \\
 \downarrow & & \downarrow \\
 T_pR_h\oplus T_qR_{h'}\oplus W^{1,p;\delta}_\lambda(S,\tilde u^*TX) & \xrightarrow{D_{u}\dbar} & L^{p;\delta}(S,\Lambda^{0,1}\otimes \tilde u^*TX).
\end{array}
\end{equation}
In each case the left vertical map is $\xi(s,t)\mapsto dg\cdot d\tau\cdot\xi(s,1-t)$ and the right vertical map is $\eta(s,t)\otimes (ds-idt)\mapsto dg\cdot d\tau \cdot\eta(s,1-t)\otimes (ds-idt)$.
Note that the left vertical isomorphisms preserve the boundary conditions by (\ref{eq:conjugation-end_cap_paths_preserved}).
On the $T_pR_h$ and $T_qR_{h'}$ summands the vertical maps are trivial because $dg$ and $d\tau$ act trivially on these spaces.
\begin{lemma}\label{lemma:conjugation-gluing_isomorphism}
All of the vertical isomorphisms are compatible with the gluing isomorphism.
That is, the following diagram commutes (up to homotopy):
\begin{displaymath}
\begin{array}{ccc}
\Det(\dbar_{\lambda_p,u,\lambda_q})&\to &\Det(\dbar_{\lambda_p,Z_-}\oplus T_pR_h)\times_{T_pR_h}\Det(D_u\dbar)\\&&\times_{T_qR_{h'}}\Det(T_qR_{h'}\oplus\dbar_{\lambda_q,Z_+})\\
\downarrow &&\downarrow\\
\Det(\dbar_{\lambda_p,\tilde u,\lambda_q})&\to &\Det(\dbar_{\lambda_p,Z_-}\oplus T_pR_h)\times_{T_pR_h}\Det(D_{\tilde u}\dbar)\\&&\times_{T_qR_{h'}}\Det(T_qR_{h'}\oplus\dbar_{\lambda_q,Z_+}).
\end{array}
\end{displaymath}
The vertical maps are the ones induced by the involutions above and the horizontal maps are the gluing isomorphism.

%Therefore, to determine the sign of (\ref{eq:conjugation-involution_4}), it suffices to compute the signs of the other involutions (\ref{eq:conjugation-involution_1}), (\ref{eq:conjugation-involution_2}), and (\ref{eq:conjugation-involution_3}).
\end{lemma}
\begin{proof}
The procedure for gluing linearized operators is well-known and follows the method detailed in \cite{fh}, so only an outline of the proof will be given.
First, as shown in \cite{fh}, it may be assumed that all the operators are surjective.
Thus, the $\Det$ spaces appearing above are nothing more than the top exterior powers of the kernels of the operators, so to describe the gluing isomorphisms it suffices to describe the maps between the kernels, namely the maps
\begin{eqnarray}\label{eq:conjugation-gluing_kernels_1}
(\Ker(\dbar_{\lambda_p,Z_-})\oplus T_pR_h )\times_{T_pR_h}\Ker(D_{u}\dbar)\times_{T_qR_{h'}} \\\nonumber(T_qR_{h'}\oplus\Ker(\dbar_{\lambda_q,Z_+}))&\to&\Ker(\dbar_{\lambda_p,u,\lambda_q})
\end{eqnarray}
and
\begin{eqnarray}\label{eq:conjugation-gluing_kernels_2}
(\Ker(\dbar_{\lambda_p,Z_-})\oplus T_pR_h )\times_{T_pR_h}\Ker(D_{\tilde u}\dbar)\times_{T_qR_{h'}} \\\nonumber(T_qR_{h'}\oplus\Ker(\dbar_{\lambda_q,Z_+}))&\to&\Ker(\dbar_{\lambda_p,\tilde u,\lambda_q}).
\end{eqnarray}

%Recall from (\ref{eq:cokernel-linearized_operator_2}) that elements in the domain of $D_{u}\dbar$ can be identified (in a canonical way) with triples $(U,V,\xi)$, where $U\in T_pR_h$, $V\in T_qR_{h'}$, and $\xi$ is a section over the strip that decays to $0$ at $\pm\infty$.
An arbitrary element in the left-hand side of (\ref{eq:conjugation-gluing_kernels_1}) is of the form 
\begin{displaymath}
((\xi_-,U),\xi_0,(\xi_+,V))
\end{displaymath}
where $\xi_-,\xi_0,\xi_+$ are holomorphic and $\xi_0$ satisfies  $$\xi_0(-\infty,\cdot)=U\in T_pR_h$$ and $$\xi_0(+\infty,\cdot)=V\in T_qR_{h'}.$$
For appropriate cutoff functions $\chi_-(s)$, $\chi_0(s)$, and $\chi_+(s)$ these sections can be glued to get a section
\begin{equation}\label{eq:conjugation-gluing_pre}
\chi_-\xi_-+(1-\chi_-)U+\chi_0\xi_0+(1-\chi_+)V+\chi_+\xi_+
\end{equation}
over $Z_-\# S\# Z_+$.
This section is not holomorphic, so it has to be projected to the kernel of $\dbar_{\lambda_p,u,\lambda_q}$.
To do this, let $E$ be a fixed complement in $W^{1,p}_\lambda$ of the kernel, and let
\begin{displaymath}
(\chi_-\xi_-+(1-\chi_-)U+\chi_0\xi_0+(1-\chi_+)V+\chi_+\xi_+)_K
\end{displaymath}
be the projection of (\ref{eq:conjugation-gluing_pre}) to $\Ker(\dbar_{\lambda_p,u,\lambda_q})$ via the splitting $W^{1,p}_\lambda=E\oplus\Ker$.
Then the gluing isomorphism is defined to be the map
\begin{displaymath}
((\xi_-,U),\xi_0,(V,\xi_+))\mapsto (\chi_-\xi_-+(1-\chi_-)U+\chi_0\xi_0+(1-\chi_+)V+\chi_+\xi_+)_K.
\end{displaymath}

Now let $\rho_*$ be any one of the vertical maps in the diagram, so $\rho_*$ is of the form
\begin{displaymath}
(\rho_*\xi)(s,t)=dg\cdot d\tau\cdot\xi(s,1-t).
\end{displaymath}
Let $\tilde E=\rho_* (E)$.
Note that $\tilde E$ is a complement of $\Ker(\dbar_{\lambda_p,\tilde u,\lambda_q})$, and also $$\rho_*(\Ker(\dbar_{\lambda_p,u,\lambda_q})=\Ker(\dbar_{\lambda_p,\tilde u,\lambda_q}),$$
by Lemma \ref{lemma:conjugation-det}.
If $\tilde E$ is used to define the projection in the gluing isomorphism (\ref{eq:conjugation-gluing_kernels_2}), it is easy to check that gluing commutes with $\rho_*$.
Since it does not matter (up to homotopy) which complements are used to define the gluing isomorphism, the lemma follows.

\end{proof}

\begin{corollary}
The sign of $\Det(\rho_*)$ (viewed as an automorphism of $ev_{-\infty}^*\Theta_{R_h}^-\otimes ev_{+\infty}^*\Theta^+_{R_{h'}}$) is equal to the product of the signs of the involutions in (\ref{eq:conjugation-involution_1}), (\ref{eq:conjugation-involution_2}), and (\ref{eq:conjugation-involution_3}).
\end{corollary}
\begin{proof}
The fixed paths $\lambda_p,\lambda_q$ and fixed trivialization $Tr_p,Tr_q$ of these paths give the canonical isomorphisms in (\ref{eq:conjugation-canonical_isomorphisms}).
Therefore the commutative diagram in the previous lemma gives rise to the commutative diagram
\begin{displaymath}
\begin{array}{ccc}
\rr&\to &ev_{-\infty}^*\Theta_{R_h}^-\otimes\Det( T_pR_h)\times_{T_pR_h}\Det(D_u\dbar)\times_{T_qR_{h'}}\Det(T_qR_{h'})\otimes ev^*_{+\infty}\Theta^+_{R_{h'}}\\
\downarrow &&\downarrow\\
\rr&\to &ev_{-\infty}^*\Theta_{R_h}^-\otimes\Det( T_pR_h)\times_{T_pR_h}\Det(D_{\tilde u}\dbar)\times_{T_qR_{h'}}\Det(T_qR_{h'})\otimes ev^*_{+\infty}\Theta^+_{R_{h'}}.\\
\end{array}
\end{displaymath}
The horizontal isomorphisms in this diagram determine the isomorphisms (\ref{eq:conjugation-orientation}) for $u$ and $\tilde u$.
\end{proof}
Consider first the involution (\ref{eq:conjugation-involution_1}).
Let $u'(s,t)=\tau u(s,1-t)$.
Then $u'$ is a strip with top boundary on $L$ and bottom boundary on $g^{-1}L$.
(\ref{eq:conjugation-involution_1}) can be factored as the composition
\begin{displaymath}\label{eq:involution_composition}
\begin{array}{ccc}
 W^{1,p}_\lambda(Z_-\# S \# Z_+, u^*TX) & \xrightarrow{\dbar_{\lambda(\lambda_q,u,\lambda_p)}} & L^p(Z_-\# S\# Z_+, \Lambda^{0,1}\otimes u^*TX) \\
\downarrow\Psi_1 && \downarrow\Psi_1 \\
W^{1,p}_\lambda(Z_-\# S \# Z_+, u'^*TX) & \xrightarrow{\dbar_{\lambda(\bar\lambda_q^{-1},u',\bar\lambda_p^{-1})}} & L^p(Z_-\# S\# Z_+, \Lambda^{0,1}\otimes u'^*TX) \\
\downarrow \Psi_2&& \downarrow\Psi_2 \\
W^{1,p}_\lambda(Z_-\# S \# Z_+, \tilde u^*TX) & \xrightarrow{\dbar_{\lambda(\lambda_q,\tilde u,\lambda_p)}} & L^p(Z_-\# S\# Z_+,\Lambda^{0,1}\otimes \tilde u^*TX) ,
\end{array}
\end{displaymath}
where $\Psi_1$ maps $\xi(s,t)$ to $d\tau \cdot\xi(s,1-t)$ and $\Psi_2$ maps $\xi(s,t)$ to $dg\cdot\xi(s,t)$.
Here $\bar\lambda_x^{-1}(t)=\overline{\lambda_x(1-t)}$ for $x=p,q$.
%Call the first map $\Psi_1$ and the second map $\Psi_2$, and call the rows A,B,C.
Call the rows in the diagram above A, B, and C.
The Lagrangian boundary conditions in A, B, and C have spin structures in the obvious way, call them $\textrm{Spin}_\textrm{A}$, $\textrm{Spin}_\textrm{B}$, and $\textrm{Spin}_\textrm{C}$.
The spin structures give canonical orientations to the determinant line of the operators in A, B, and C.

\begin{lemma}\label{lemma:conjugation-spin}
$$\Psi_{1*}\mathrm{Spin}_\mathrm{A}=\mathrm{Spin}_\mathrm{B},$$
$$\Psi_{2*}\mathrm{Spin}_\mathrm{B}=\mathrm{Spin}_\mathrm{C}.$$
\end{lemma}
\begin{proof}
For the first statement, it needs to be checked that the trivialization of the Lagrangian bundle along the boundary in A given by $\textrm{Spin}_\textrm{A}$ pushes forward (up to homotopy) to the trivialization of the Lagrangian frame bundle around the boundary in B given by $\mathrm{Spin}_\textrm{B}$.

For the two end caps, the two trivializations agree by definition (the trivialization over the end caps is specified by (\ref{eq:end_cap_triv})).

For the top, the pushforward trivialization is $Fr(u(s,0))$, and the trivialization coming from $\textrm{Spin}_\textrm{B}$ is also $Fr(u(s,0))$.

For the bottom, the pushforward trivialization is 
$$d\tau\cdot dg Fr(g^{-1}u(s,1))=dg^{-1}\cdot d\tau Fr(g^{-1}u(s,1))=dg^{-1} Fr(g^{-1}u(s,1)).$$
The trivialization coming from $\textrm{Spin}_\textrm{B}$ is 
$$dg^{-1}Fr(gu'(s,0))=dg^{-1}Fr(g\tau u(s,1))=dg^{-1}Fr(g^{-1}u(s,1)).$$

Thus the spin structures agree (in fact exactly, not just up to homotopy). The proof of the second statement is similar.
\end{proof}

\begin{lemma}
$\Det(\Psi_1):\Det(\dbar_{\lambda(\lambda_q,u,\lambda_p)})\to\Det(\dbar_{\lambda(\lambda_q,u',\lambda_p)})$ has sign $(-1)^{\mu/2}$ where $\mu=\mu(\lambda(\lambda_p,u,\lambda_q))$.
\end{lemma}
\begin{proof}
By the previous lemma, $\Psi_1$ preserves spin structure.
All of the data can be homotoped so that $\Psi_1$ is simply conjugation.
The lemma then follows from Theorem 1.3 in \cite{foooasi}.
\end{proof}

Let $k$ be an integer such that $0\leq k\leq 4$.
Define
\begin{displaymath}
k'=\left\{
\begin{array}{ll}
k & k=0,1,2,\\
k-\frac{5}{2} & k=3,4.
\end{array}
\right.
\end{displaymath}
and
\begin{displaymath}
k''=\left\{
\begin{array}{ll}
0 & k=0,1,2,\\
1 & k=3,4.
\end{array}
\right.
\end{displaymath}
For example, if $\gamma^{a_x},\gamma^{b_x},\gamma^{c_x}$ are the eigenvalues of the $g$ action on $T_xX$ then $a_x''$, $b_x''$, and $c_x''$ are defined.

\begin{lemma}\label{lemma:conjugation-change_1}
$$\mu=\mu(u)+(a_q''+b_q''+c_q''-a_p''-b_p''-c_p'').$$
\end{lemma}
\begin{proof}
$\mu(u)$ was defined by taking positive definite paths in the unoriented Lagrangian grassmannian from $TL$ to $TgL$ at $p$ and $q$.
$\lambda_p$ and $\lambda_q$ were defined by taking paths in the oriented Lagrangian grassmannian.
%By Lemma \ref{lemma:index_theory-matrix} and the proof of Lemma \ref{lemma:real_lagrangians-angles}, there exists a basis of $T_xL$ such that 
With respect to an appropriate basis of $T_xL$, the positive definite path is
\begin{displaymath}
t\mapsto\left[
\begin{array}{ccc}
\gamma^{ta_x'} & 0 & 0\\
0 & \gamma^{tb_x'} & 0 \\
0 & 0 & \gamma^{tc_x'}
\end{array}
\right]\cdot T_xL
\end{displaymath}
and the path $\lambda_x$ is
\begin{displaymath}
t\mapsto\left[
\begin{array}{ccc}
\gamma^{ta_x} & 0 & 0\\
0 & \gamma^{tb_x} & 0 \\
0 & 0 & \gamma^{tc_x}
\end{array}
\right]\cdot T_xL.
\end{displaymath}
Since $(-1)^{k''}\gamma^{k'}=\gamma^k$, it follows that $\mu-\mu(u)=a_q''+b_q''+c_q''-a_p''-b_p''-c_p''$.
\end{proof}

\begin{lemma}
$\Det(\Psi_2)$ preserves orientation.
\end{lemma}
\begin{proof}
$\Psi_2$ maps the section $\xi(s,t)$ to $dg\cdot \xi(s,t)$, so $\Psi_2$ covers the identity map on $(D^2,\partial D^2)$.
The lemma then follows from Lemmas \ref{lemma:local_systems-main_lemma} and \ref{lemma:conjugation-spin}.
\end{proof}

Next, consider the involutions on the end caps (\ref{eq:conjugation-involution_2}) and (\ref{eq:conjugation-involution_3}).
To determine how the sign changes, the change of sign in the canonical end cap problem needs to be determined.
To this end, let
$$\dbar_-:W^{1,p}_\lambda(Z_-,\cc)\to L^p(Z_-,\cc),$$
$$\dbar_+:W^{1,p}_\lambda(Z_+,\cc)\to L^p(Z_+,\cc)$$
denote the standard $\dbar$ operators.
For these operators it will be more convenient to think of $Z_-$ as $Z_-=[0,\infty)\times[0,1]$ and $Z_+$ as $Z_+=(-\infty,0]\times[0,1]$.
$\lambda$ denotes the boundary conditions
\begin{itemize}
\item $\rr$ along $t=0$,
\item $i\rr$ along $t=1$, and
\item $e^{i\pi(\frac{1}{2}-k)t}\cdot\rr$ along $s=0$.
\end{itemize}

In \cite{ohgaf2} it is shown that $\Ind(\dbar_-)=k$ and $\Ind(\dbar_+)=-k+1$.
Moreover, the kernel of $\dbar_-$ for $k\leq 0$ is $0$ and for $k>0$ is spanned by
\begin{displaymath}
e^{\pi(\frac{1}{2}-k)z},e^{\pi(\frac{1}{2}-k+1)z}+e^{\pi(\frac{1}{2}-k-1)z},\ldots,e^{\pi(\frac{1}{2}-1)z}+e^{\pi(\frac{1}{2}-2k+1)z}.
\end{displaymath}
The cokernel is $0$ for $k\geq0$ and for $k<0$ is spanned by
\begin{displaymath}
e^{-\pi(\frac{1}{2}-k+1)\bar z}-e^{-\pi(\frac{1}{2}-k-1)\bar z},\ldots,e^{-\pi(\frac{1}{2}-k+k)\bar z}-e^{-\pi(\frac{1}{2}-k-k)\bar z}.
\end{displaymath}
The kernel of $\dbar_+$ for $k\geq 1$ is $0$ and for $k\leq 0$ is spanned by
\begin{displaymath}
e^{\pi(\frac{1}{2}-k)z},e^{\pi(\frac{1}{2}-k+1)z}+e^{\pi(\frac{1}{2}-k-1)z},\ldots,e^{\pi(\frac{1}{2})z}+e^{\pi(\frac{1}{2}-2k)z}.
\end{displaymath}
The cokernel is $0$ for $k\leq0$ and for $k\geq1$ is spanned by
\begin{displaymath}
e^{-\pi(\frac{1}{2}-k+1)\bar z}-e^{-\pi(\frac{1}{2}-k-1)\bar z},\ldots,e^{-\pi(\frac{1}{2}-k+k)\bar z}-e^{-\pi(\frac{1}{2}-k-k)\bar z}.
\end{displaymath}

Let $\Theta_\pm$ be the map that takes a section $\xi(s,t)$ of $\cc$ over $Z_\pm$ to $i\overline{\xi(s,1-t)}$.
Notice that $\Theta_\pm e^{i\pi(\frac{1}{2}-k)t}=ie^{-i\pi(\frac{1}{2}-k)(1-t)}=e^{i\pi k}e^{i\pi(\frac{1}{2}-k)t}$.
Thus $\Theta_\pm$ fits into a commutative diagram with the operators $\dbar_\pm$:
\begin{displaymath}
\begin{array}{ccc}
W^{1,p}_\lambda(Z_\pm,\cc) & \xrightarrow{\dbar_\pm} & L^p(Z_\pm,\cc)\\
\downarrow\Theta_\pm&&\downarrow\Theta_\pm\\
W^{1,p}_\lambda(Z_\pm,\cc)&\xrightarrow{\dbar_\pm} & L^p(Z_\pm,\cc)
\end{array}.
\end{displaymath}
The vertical maps $\Theta_\pm$ cover the map $(s,t)\mapsto(s,1-t)$ on $Z_\pm$.
The commutativity of the diagram implies that $\Theta_\pm$ induces maps $\Det(\Theta_\pm):\Det(\dbar_\pm)\to\Det(\dbar_\pm)$.
\begin{lemma}\label{lemma:conjugation-canonical}
$\Det(\Theta_-)$ changes orientation by $(-1)^{k(k+1)/2}$.
$\Det(\Theta_+)$ changes orientation by $(-1)^{k(k-1)/2}$.
\end{lemma}
\begin{proof}
First consider $\Det(\Theta_-)$.
Suppose $k>0$.
Then $\Theta_-$ maps the basis element 
$$e^{\pi(\frac{1}{2}-k+j)(s+it)}+e^{\pi(\frac{1}{2}-k-j)(s+it)}$$ 
of the kernel to
$$
ie^{\pi(\frac{1}{2}-k+j)(s-i(1-t))}+ie^{\pi(\frac{1}{2}-k-j)(s-i(1-t))}=
e^{i\pi (k-j)}(e^{\pi(\frac{1}{2}-k+j)(s+it))}+e^{\pi(\frac{1}{2}-k-j)(s+it)}).
$$
Thus the sign is changed by
$$
(-1)^{\sum_{j=0}^{k-1}(k-j)}=(-1)^{k(k+1)/2}.
$$

If $k<0$, the cokernel basis element
$$e^{-\pi(\frac{1}{2}-k+j)(s-it)}+e^{-\pi(\frac{1}{2}-k-j)(s+it)}$$ 
is mapped to
$$ie^{-\pi(\frac{1}{2}-k+j)(s+i(1-t))}+ie^{-\pi(\frac{1}{2}-k-j)(s+i(1-t))}=$$ 
$$e^{i\pi(k-j)}(e^{-\pi(\frac{1}{2}-k+j)(s-it)}+e^{-\pi(\frac{1}{2}-k-j)(s-it)}).$$ 
Thus the sign changes by
$$(-1)^{\sum_{j=1}^{|k|}(k-j)}=(-1)^{k-1+\cdots+2k}=(-1)^{k^2-1-\cdots+k}=(-1)^{k^2+k(-k+1)/2}=(-1)^{k(k+1)/2}.$$

Now consider the case of $\Theta_+$.
If $k\leq 0$ then the kernel basis element
\begin{displaymath}
e^{\pi(\frac{1}{2}-k+j)(s+it)}+e^{\pi(\frac{1}{2}-k-j)(s+it)}
\end{displaymath}
is mapped to
\begin{displaymath}
i(e^{\pi(\frac{1}{2}-k+j)(s-i(1-t))}+e^{\pi(\frac{1}{2}-k-j)(s-i(1-t))})=e^{\pi(k-j)i}(e^{\pi(\frac{1}{2}-k+j)(s+it)}+e^{\pi(\frac{1}{2}-k-j)(s+it)}).
\end{displaymath}
Thus $\Det(\Theta_+)$ changes orientation by
\begin{displaymath}
(-1)^{\sum_{j=0}^{|k|} (k-j)}=(-1)^{k(k-1)/2}.
\end{displaymath}

If $k>0$ then the cokernel basis element
\begin{displaymath}
e^{-\pi(\frac{1}{2}-k+j)(s-it)}-e^{-\pi(\frac{1}{2}-k-j)(s-it)}
\end{displaymath}
is mapped to
\begin{eqnarray*}
&i(e^{-\pi(\frac{1}{2}-k+j)(s+i(1-t))}-e^{-\pi(\frac{1}{2}-k-j)(s+i(1-t))})=&\\
&e^{\pi(k-j)i}(e^{-\pi(\frac{1}{2}-k+j)(s-it)}-e^{-\pi(\frac{1}{2}-k-j)(s-it)}).
\end{eqnarray*}
Hence $\Det(\Theta_+)$ changes orientation by
\begin{displaymath}
(-1)^{\sum_{j=1}^k(k-j)}=(-1)^{1+\cdots+(k-1)}=(-1)^{k(k-1)/2}.
\end{displaymath}
\end{proof}

Notice that the domain of $\dbar_-$ can be glued to the domain of $\dbar_+$ to get a bundle pair with Maslov index $0$ defined over the disc.
The operators $\Theta_-$ and $\Theta_+$ can also be glued together to get an involution of the glued problem over the disc.
The glued operator is almost conjugation.
The subtle difference is that if $k$ is odd then the glued operator also switches the orientation of the Lagrangian subbundle along the boundary.
Thus, by Lemma \ref{lemma:local_systems-main_lemma}, the glued operator changes orientation by $(-1)^{k}$.
This is consistent with the previous lemma, which says that the glued operator changes sign by
\begin{displaymath}
\textrm{sign}(\Det(\Theta_-))\textrm{sign}(\Det(\Theta_+))=(-1)^{k(k+1)/2+k(k-1)/2}=(-1)^{k^2}=(-1)^k.
\end{displaymath}

Now suppose the top boundary condition is $e^{2\pi i a/5}\cdot \rr$ where $1\leq a\leq 4$ and the boundary condition along $s=0$ is $e^{2\pi i a t/5}\cdot \rr$.
Denote sections with these boundary conditions by $W^{1,p}_{\gamma^a}(Z_\pm,\cc)$.

Consider the map $\xi(s,t)\mapsto \gamma^a\overline{\xi(s,1-t)}$.
Since 
$$\gamma^a\overline{e^{2\pi ia (1-t)/5}}=e^{2\pi i a t/5},$$
this gives a map 
$$
\Theta_\pm^{\gamma^a}:W^{1,p}_{\gamma^a}(Z_\pm,\cc)\to W^{1,p}_{\gamma^a}(Z_\pm.\cc),$$ 
and likewise for $L^p$ sections.
\begin{lemma}\label{lemma:conjugation-change_2}
If $0\leq a\leq 2$ then $\Det(\Theta_\pm^{\gamma^a})$ preserves orientation.
If $3\leq a\leq 4$ then $\Det(\Theta_\pm^{\gamma^a})$ changes orientation.
%$\Det(\Theta_\pm^{\gamma^a})$ preserves orientation.
\end{lemma}
\begin{proof}
If $a=0$ then the kernel and cokernel are $0$, so $\Theta_\pm^{\gamma^a}$ preserves orientation (weighted Sobolev spaces have to be used in this case).

If $1\leq a\leq 2$ then the problem can be homotoped to the one with boundary condition $e^{i\pi \frac{1}{2}t}\cdot \rr$ along $s=0$, and $\Theta_\pm^{\gamma^a}$ becomes the operator $\Theta_\pm$.
Then apply Lemma \ref{lemma:conjugation-canonical} with $k=0$.

If $3\leq a\leq 4$ then the problem can be homotoped to the one with boundary condition $e^{i\pi(\frac{1}{2}+1)t}\cdot \rr$ along $s=0$, and $\Theta_\pm^{\gamma^a}$ becomes the operator $-\Theta_\pm$ (because multiplcation by $\gamma^a$ homotopes to multiplication by $-i$).
$-\Theta_\pm$ changes orientation by $(-1)^{\Ind(\dbar_\pm)+k(k\mp1)/2}$.
Plugging in $k=-1$ then shows that $-\Theta_\pm$ changes orientation.
\end{proof}

Now the sign of the involutions (\ref{eq:conjugation-involution_2}) and (\ref{eq:conjugation-involution_3}) can be calculated.
\begin{lemma}
The involution (\ref{eq:conjugation-involution_2}) changes orientation by $(-1)^{a_q''+b_q''+c_q''}$.
The involution (\ref{eq:conjugation-involution_3}) changes orientation by $(-1)^{a_p''+b_p''+c_p''}$.
%changes oreintation by $(-1)^{a_q''+b_q''+c_q''}$ and $\Det(D_-)$ changes oreintation by $(-1)^{a_p''+b_p''+c_p''}$.
\end{lemma}
\begin{proof}
Let $x$ be $p$ or $q$.
As in the proof of Lemma \ref{lemma:real_lagrangians-angles}, there exists a basis of $T_xL$ such that the involutions split as a direct sum of operators of the form $\Theta_\pm^{\gamma^a}$.
The result then follows from Lemma \ref{lemma:conjugation-change_2}.
\end{proof}

Finally, the main results of this section can be proven.
\begin{prop}\label{prop:conjugation-main_prop}
$\Det(\rho_*)$ changes orientation by
$$
(-1)^{\frac{1}{2}\mu(u)+\frac{3}{2}( +a_q''+b_q''+c_q''-a_p''-b_p''-c_p'')}.
$$
\end{prop}
\begin{proof}
By Lemma \ref{lemma:conjugation-gluing_isomorphism}, the sign of $\Det(\rho_*)$ is the product of the signs of the involutions (\ref{eq:conjugation-involution_1}), (\ref{eq:conjugation-involution_2}), and (\ref{eq:conjugation-involution_3}).
By Lemmas \ref{lemma:conjugation-change_1} and \ref{lemma:conjugation-change_2}, this product is $\frac{1}{2}\mu(u)+\frac{3}{2}(a_q''+b_q''+c_q''-a_p''-b_p''-c_p'')$.
\end{proof}

\begin{prop}\label{prop:conjugation-main_prop_2}
Fix two components $R_h$ and $R_{h'}$.
Suppose that $\rho_*$ has sign $-1$ on $\mathcal M(L,gL:R_h,R_{h'})$.
Furthermore, suppose that the virtual dimension of $\mathcal M(L,gL:R_h,R_{h'})$ is $0$.
Assume that the line bundle $\Det(TR_h)\otimes \Theta_{R_h}^-$ is trivial.
Let $[R_h]$ be the fundamental class of $R_h$ with coefficients in the local system $\Det(TR_h)\otimes \Theta_{R_h}^-$.
Fix an energy $E>0$.
Then
\begin{displaymath}
\#(\mathcal M(L,gL:R_h,R_{h'}:E)\times_{R_h} [R_h])=0.
\end{displaymath}
(Fixing the energy $E$ is necessary to make the cardinality of the set finite.
A priori, the ``cardinality'' is an element of $\qq$, because the fiber product gives a $0$-dimensional chain with $\qq$-coefficients.)
\end{prop}
\begin{proof}
The idea for this proof is taken from \cite{sol}.
By Lemma \ref{lemma:conjugation-involution_kuranishi_spaces}, $\z_2$ acts on the Kuranishi space $\mathcal M(L,gL:R_h,R_{h'}:E)$.
By Lemma A1.49 in \cite{fooo}, there exists a Kuranishi structure on the quotient $\mathcal M(L,gLR_h,R_{h'}:E)/\z_2$.
Choose a transverse multisection on the quotient, and then lift this multisection to $\mathcal M(L,gL:R_h,R_{h'}:E)$.
The lifted multisection will then be invariant under the $\z_2$ action.
Denote this multisection by
\begin{displaymath}
\mathfrak s_\epsilon=\sett{s_{p,\epsilon}'}_{p\in\mathcal P}.
\end{displaymath}

Since the virtual dimension of $\mathcal M(L,gL:R_h,R_{h'}:E)$ is $0$, this means that the zero set of the multisection contains a finite number of points (i.e. strips); call these points $u_1,\ldots,u_n$.
In the fiber product
\begin{displaymath}
\mathcal M(L,gL:R_h,R_{h'}:E)\times_{R_h}[R_h],
\end{displaymath}
each of these points gets assigned a rational weight in the local system $\Theta_{R_{h'}}^+$ on $R_{h'}$.
The weight is defined as follows:
Suppose $u_i$ is in the Kuranishi chart $(V_{p_i},E_{p_i},\Gamma_{p_i},\psi_{p_i},s_{p_i,\epsilon}')$.
Let $y\in V_{p_i}$ such that $\pi(y)=u_i$ under the projection $\pi:V_{p_i}\to V_{p_i}/\Gamma_{p_i}$.
Let
\begin{displaymath}
\tilde s_{p_i,\epsilon}'=(\tilde s_{p_i,\epsilon,1}',\ldots,\tilde s_{p_i,\epsilon,n_i}')
\end{displaymath}
be the branches of the multisection $s_{p_i,\epsilon}'$.
Let
\begin{displaymath}
\tilde s_{p_i,\epsilon,1}',\ldots,\tilde s_{p_i,\epsilon,k_i}'
\end{displaymath}
be the branches that vanish at $y$.

The weight assigned to $u_i$ is then
\begin{displaymath}
mul_i=\frac{\sum_{j=1}^{k_i}\epsilon_j}{n_i},
\end{displaymath}
where $\epsilon_j$ is $\pm1$ according to whether the surjective map
\begin{displaymath}
d_y\tilde s_{p_i,\epsilon,j}':T_yV_{p_i}\to E_{p_i}
\end{displaymath}
preserves or reverses orientation.
More precisely, note that $d_y\tilde s_{p_i,\epsilon,j}'$ can be viewed as an element of 
\begin{displaymath}
\Det(E_{p_i}^*)\otimes \Det(T_yV_{p_i})\cong \Det(D_{u_i}\dbar).
\end{displaymath}
$\Det(D_{u_i}\dbar)$ is canonically isomorphic to $\Theta_{R_h,p}^-\otimes \Theta_{R_{h'},q}^+$, where $p=u_i(-\infty,\cdot)$ and $q=u_i(+\infty,\cdot)$.
The class $[R_h]$ in the local system $\Det(TR_h)\otimes \Theta_{R_h}^-$ gives an orientation to $\Theta_{R_h,p}^-$, hence an orientation of $\Theta_{R_{h'},q}^+$ gives an orientation to $\Det(D_{u_i}\dbar)$.
Hence, it makes sense to say that $\epsilon_j$ is an element in $\Theta_{R_{h'},q}^+$.

The proof is now easy.
By construction, $\rho_*$ permutes the branches of the multisections.
If $\rho_*$ takes the branch $\tilde s_{p_1,\epsilon,j_1}'$ to $\tilde s_{p_2,\epsilon,j_2}'$, then the numbers $\epsilon_{j_1}$ and $\epsilon_{j_2}$ coming from these branches have to be opposites of each other.
Indeed, let $\pi(y_1)=u_{i_1}$ and $\pi(y_2)=u_{i_2}$, with $\rho_*(u_{i_1})=u_{i_2}$ and $\tilde s_{p_1,\epsilon,j_1}'(y_1)=0$, $\tilde s_{p_2,\epsilon,j_2}'(y_2)=0$.
Since $u_{i_1}(\pm\infty,\cdot)=u_{i_2}(\pm\infty,\cdot)$, $\Det(D_u\dbar)$ and $\Det(D_{\tilde u}\dbar)$ are both canonically isomorphic to the same fiber of $\Theta_{R_h}^-\otimes \Theta_{R_{h'}}^+$.
However, by assumption $\Det(\rho_*)$ reverses sign when viewed as an automorphism of $\Theta_{R_h}^-\otimes \Theta_{R_{h'}}^+$.
Therefore, by the way $\epsilon_{j_1}$ and $\epsilon_{j_2}$ are assigned signs as described above, they must be opposite of each other.
It follows that $mul_{u_i}=0$.
(Note that this argument works even if $\rho_*$ has fixed points, because the fact that it reverses signs implies that it cannot fix any of the individual branches of the multisection.)

\end{proof}

\part{Calculations of Floer cohomology}
%\chapter{Calculation of Floer cohomology}
In this last part we finally calculate the Floer cohomology.
We proceed on a case by case basis, following the classification given in Section \ref{section:summary}.
The properties developed in Part II are used to carry out the calculations.

\section{Case (1): $L\cap gL\cong \rp^3$}

The first case to consider is $g=1$, in which case the Floer cohomology is by definition the cohomology of the $A_\infty$-algebra $CF(L)$.

  \begin{theorem}\label{thm:case1-main_thm}
  $$HF^*(L;\Lambda_{nov})\cong H^*(L;\Lambda_{nov}).$$
That is, the Floer cohomology is (graded) isomorphic to the singular cohomology.
  \end{theorem}
  \begin{proof}
The Floer coboundary operator $\delta$ can be written as $\delta=\delta_0+\delta_1$ where $\delta_0$ is the ordinary singular cohomology operator and $\delta_1$ is the part coming from positive energy holomorphic discs.
Because $L$ has constant phase, the Maslov index of any holomorphic disc with boundary on $L$ is $0$.

The spectral sequence (Theorem \ref{thm:floer_cohomology-spectral_sequence}) implies that the Floer cohomology is isomorphic to the cohomology of $\oplus_kH^k(L;\qq)\otimes gr(\mathcal F\Lambda_{0,nov})$ with respect to the differential induced by $\delta_1$.
The theorem will follow if it can be shown that $\delta_1\equiv 0$ on $H^*(L)$.
But this follows from degree considerations:
Indeed, let $[S]\in H^k(L)$ be a chain (so the dimension of $S$ is $3-k$).
Then
\begin{displaymath}
\dim \delta_1 S=\dim S+\dim \mathcal M_2(L)-3=\dim S+3+0+2-3-3=\dim S-1.
\end{displaymath}
It follows that if $[S]$ is non-zero then the cohomology class of $\delta_1 S$ must be zero, because $L$ only has non-zero cohomology in degrees $0$ and $3$.
Thus $\delta_1\equiv 0$ on $H^*(L)$.
  \end{proof}

\section{Case (2): $L\cap gL\cong \rp^2$}
The Floer cohomology in this (and subsequent) sections is the Floer cohomology of the pair $(L,gL)$.
That is, $HF(L,gL;\Lambda_{nov})$ is the cohomology of the $A_\infty$-bimodule $CF(L,gL)$ with respect to the differential $n_{0,0}$.

Without loss of generality, assume $g=(1,\gamma^a,1,1,1)$.
Then $L\cap gL=\textrm{Fix}(g)\cap L=\rp^2$, where $\rp^2$ is the subset of $L$ given by
\begin{displaymath}
\rp^2=\sett{[X_0:0:X_2:X_3:X_4]\in L}.
\end{displaymath}

\begin{lemma}\label{lemma:case2-line_bundle}
The line bundle $\Theta_{\rp^2}^-$ is trivial on $\rp^2$ if $a$ is 1 or 2, and non-trivial if $a$ is 3 or 4.
\end{lemma}
\begin{proof}

Let $\sigma:S^1\to\rp^2\subset L$ represent a non-trivial loop in $\pi_1(\rp^2)$.
Explicitly, $\sigma$ can be taken to be
$$
\sigma(\theta)=[\cos\pi\theta:0:\sin\pi\theta:0:-(cos^5\pi\theta+sin^5\pi\theta)^{1/5}],
$$
where $\theta\in S^1=\rr/\z$.
To determine if $\Theta^-_{\rp^2}$ is trivial or not, it suffices to determine if it is trivial or not over $\gamma$.

For each $\theta\in S^1$, let $dg_\theta$ be the $dg$ action on $T_{\sigma(\theta)}X$.
The eigenvalues of $dg_\theta$ are $1,1,\gamma^a$ by the proof of Lemma \ref{lemma:real_lagrangians-angles}.
Moreover, $T_{\sigma(\theta)}L$ has a basis consisting of eigenvalues of $dg_\theta$.
Let $E_{0,\theta}$ be the 1-eigenspace and $E_{a,\theta}$ the $\gamma^a$ eigenspace.
The vector spaces $E_{0,\theta}$ fit together to give a subbundle $E_0$ of $\sigma^*TL$; likewise the $E_{a,\theta}$ give a subbundle $E_a$.
Thus there is a splitting
\begin{equation}\label{eq:case1-splitting}
\sigma^*TL\cong E_0\oplus E_1.
\end{equation}
Notice that the 1-eigenspace of $dg_\theta$ is precisely $T_{\sigma(\theta)}\rp^2$, so $E_0\cong \sigma^*T\rp^2$.
Since $TL$ is orientable, and $\sigma^*T\rp^2$ is not orientable, it follows that $E_1$ is not orientable.

For each $\theta\in S^1$, and for $0\leq t\leq 1$ let $dg_\theta^t$ be the $t^{th}$ power of $dg_\theta$.
Let $\lambda_\theta:[0,1]\to\Lambda(T_{\sigma(\theta)}X)$ be the path
$$
\lambda_\theta(t)=dg_\theta^t\cdot T_{\sigma(\theta)}L.
$$
Notice that $\lambda_\theta$ can be decomposed as 
\begin{equation}\label{eq:case1-splitting_paths}
\lambda_\theta(t)=(dg_\theta^t\cdot E_0)\oplus (dg_\theta^t\cdot E_a)=:\lambda^0_\theta(t)\oplus\lambda_\theta^a(t).
\end{equation}
using the splitting (\ref{eq:case1-splitting}).

Let $\Det(\dbar_{\lambda,Z_-})$ be the determinant bundle over $S^1$ whose fiber over $\theta\in S^1$ is
$\Det(\dbar_{\lambda_\theta,Z_-}).$
Using (\ref{eq:case1-splitting_paths}), this bundle splits as 
\begin{equation}\label{eq:case1-det_splitting}
\Det(\dbar_{\lambda,Z_-})\cong \Det(\dbar_{\lambda^0,Z_-})\oplus \Det(\dbar_{\lambda^a,Z_-}).
\end{equation}

Now the bundle $\Det(\dbar_{\lambda^0,Z_-})$ is trivial, because each fiber is canonically isomorphic to $\rr$.
Indeed, for $\theta$ fixed, the boundary conditions given by $\lambda^0_\theta$ are constant, namely the Lagrangian plane specified for each boundary point of $Z_-$ is simply $T_{\sigma(\theta)}\rp^2$.
Thus the kernel and cokernel of the operator $\dbar_{\lambda^0,Z_-}$ are both $0$, hence by definition the determinant is $\rr$.

The bundle $\Det(\dbar_{\lambda^a,Z_-})$ is trivial if $a=1$ or $2$, and non-trivial if $a=3$ or $4$.
Indeed, if $a=1$ or $2$ then by the discussion before Lemma \ref{lemma:conjugation-canonical}, the operators $\dbar_{\lambda^a_\theta,Z_-}$ can be homotoped in a consistent way to operators that have 0 kernel and cokernel (they can homotoped to operators where the Lagrangian path is $e^{\pi i t/2}\cdot E_a$).
Thus, for the same reason as in the previous paragraph, $\Det(\dbar_{\lambda^a,Z_-})$ is trivial.
On the other hand, if $a=3$ or $4$, then the operators can be homotoped to ones where the Lagrangian path is $e^{3\pi i t/2}\cdot E_a$, hence they have a 1-dimensional cokernal and 0-dimensional kernel.
At each point $\theta\in S^1$, the cokernal can canonically be identified with the fiber $E_{a,\theta}$.
Thus $\Det(\dbar_{\lambda^a,Z_-})$ is isomorphic to $E_a$, and hence is not orientable.

By the splitting (\ref{eq:case1-det_splitting}), it follows that $\Det(\dbar_{\lambda,Z_-})$ it orientable if $a=1$ or $2$ and not orientable if $a=3$ or $4$.
To complete the proof, it remains to show that $\Det(\dbar_{\lambda,Z_-})\cong \sigma^*\Theta^-_{\rp^2}$ as bundles over $S^1$.
To see that these bundles are isomorphic, recall from Section \ref{section:local_systems} how to determine if $\Theta^-_{\rp^2}$ is trivial over the loop $\sigma$: 
The spin structures and orientations on $L$ and $gL$ determine trivializations of $TL$ and $TgL$ over $\sigma$.
Then a 1-parameter family of paths of frames $Fr_\theta$ is chosen, so that for $0\leq t\leq 1$
\begin{itemize}
\item $Fr_\theta(t)$ is a frame whose span is a Lagrangian subspace of $T_{\sigma(\theta)}X$,
\item $T_{\sigma(\theta)}L\cap T_{\sigma(\theta)}gL\subset \textrm{span }Fr_\theta(t)$, 
\item $Fr_\theta(0)$ is the frame of $T_{\sigma(\theta)}L$ determined by the spin structure and orientation of $L$, and
\item $Fr_\theta(1)$ is the frame of $T_{\sigma(\theta)}gL$ determined by the spin structure and orientation of $gL$.
\end{itemize}
Since the $\rr$ span of each $Fr_\theta(t)$ is a Lagrangian subspace, it defines a 1-parameter family of Lagrangian paths $\lambda_\theta'$.
Then $\sigma^*\Theta_{\rp^2}^-$ is defined to be isomorphic to the bundle whose fiber over $\theta$ is $\Det(\dbar_{\lambda_\theta',Z_-})$.
In the case at hand, the paths of frames can be taken to be $Fr_\theta(t)=dg_\theta^t\cdot Fr(T_{\sigma(\theta)}L)$, where $Fr(T_{\sigma(\theta)}L)$ is the frame of $TL$ over $\sigma$ induced by the spin structure and orientaion of $L$.
Since $gL$ is given the push forward orientation and spin structure, it follows that this path of frames satisfies the properties above.
Moreover, $\lambda_\theta'(t)=\textrm{span }Fr_\theta(t)$ is equal to $\lambda_\theta$ defined in (\ref{eq:case1-splitting_paths}).
Thus $\sigma^*\Theta_{\rp^2}^-\cong \Det(\dbar_{\lambda,Z_-})$, and the proof is complete.
\end{proof}

\begin{theorem}\label{thm:case2-main_thm}
For $g=(1,\gamma^a,1,1,1)$, 
  $$HF^*(L,gL;\Lambda_{nov})\cong H^{*+a''}(\rp^2;\Det(T\rp^2)\otimes\Theta^-_{\rp^2};\Lambda_{nov}),$$
where $a''=0$ if $a=1$ or $2$ and $a''=1$ if $a=3$ or $4$.
That is, the Floer cohomology is isomorphic to the singular cohomology of $\rp^2\cong L\cap gL$ with coefficients in the local system $\Det(T\rp^2)\otimes\Theta^-_{\rp^2}$.

In particular, if $a$ is $1$ or $2$ then the Floer cohomology has rank $2$, and if $a$ is $3$ or $4$ then the Floer cohomology has rank $1$.
  \end{theorem}
  \begin{proof}
It suffices to show that if $[S]\in H^k(\rp^2)$ then $n_{0,0}S=0$ as a cohomology class.
As before, this follows from degree considerations:
First, note that if $u$ is a strip then $\mu(u)=0$ by Lemma \ref{lemma:gradings-formula}.
Then
\begin{eqnarray*}
\dim n_{0,0}S&=&\dim S+\dim \mathcal M_{0,0}(L,gL)-\dim L\cap gL=\\
&&\dim S+0+2-1-2=\dim S-1.
\end{eqnarray*}
It follows that $n_{0,0} S$ must be zero in cohomology.

To see that the isomorphism is a graded isomorphism, recall that the grading in the Bott-Morse setting was defined (see Section \ref{section:a_infinity-floer_cohomology}) to be
\begin{displaymath}
CF^k(L,gL)=\bigoplus_{R_h}C^{k-\tilde\mu(L,gL;R_h)}(R_h)=C^{k-\tilde\mu(L,gL;\rp^2)}(\rp^2).
\end{displaymath}
By Lemma \ref{lemma:gradings-absolute_formula},
\begin{displaymath}
\tilde\mu(L,gL;\rp^2)=\theta_L-\theta_{gL}+2\textrm{angle}(L,gL;\rp^2).
\end{displaymath}
From Definition \ref{dfn:real_lagrangians-grading_convention} and Lemma \ref{lemma:real_lagrangians-angles} it follows that
\begin{displaymath}
\tilde\mu(L,gL;\rp^2)=\frac{2}{5}(a'-a)=-a''.
\end{displaymath}
Therefore
\begin{displaymath}
CF^k(L,gL)=C^{k+a''}(\rp^2).
\end{displaymath}

The last statement follows from the previous lemma and the fact that $\rp^2$ is not orientable, so $H^0(\rp^2;\Det(T\rp^2)\otimes\Theta^-_{\rp^2})$ is non-zero if and only if $\Theta^-_{\rp^2}$ is trivial.
Recall that by definition
\begin{displaymath}
H^0(\rp^2;\Det(T\rp^2)\otimes\Theta^-_{\rp^2})\cong H_2(\rp^2;\Det(T\rp^2)\otimes \Theta^-_{\rp^2}).
\end{displaymath}
\end{proof}

\section{Case (3): $L\cap gL\cong\rp^1$}
In case (3) $g=g_i^rg_j^s$, $r\neq s$ and 
$$L\cap gL=L\cap Fix(g)=\rp^1.$$
$L\cap gL$ consists of the points in $X$ where the $i^{th}$ and $j^{th}$ homogeneous coordinates are both zero.
By permuting coordinates, it suffices to consider the cases where $g=g_1^rg_2^s$ and $1\leq r<s\leq 4$.
Consequently there are six cases to consider and $\rp^1$ is%: $r=1,s=2$; $r=1,s=3$; $r=1,s=4$; $r=2,s=3$; $r=2,s=4$; and $r=3,s=4$.
\begin{displaymath}
\rp^1=\sett{[X_0:0:0:X_3:X_4]\in X}.
\end{displaymath}

In all cases the Maslov index of a strip $u$ is $0$ because $L\cap gL$ has only one component.
If $S\subset L\cap gL$ is a chain then
\begin{displaymath}
\dim n_{0,0}(S)=\dim(S)+1-1+0-1=\dim(S)-1.
\end{displaymath}
$H^*(L\cap gL)=H^*(\rp^1)$ is non-zero only in dimensions 0 and 1.
Thus, by applying the spectral sequence, the only calculation that needs to be done is $n_{0,0}([\rp^1])$: if this is zero, then the Floer cohomology is isomorphic to the singular cohomology; otherwise the Floer cohomology is zero.
In particular, if $n_{0,0}([\rp^1])$ is nonzero in the $E_2$ page of the spectral sequence, then the Floer cohomology is $0$.

The cases where $(-1)^{r''+s''}=-1$ can be calculated for trivial reasons once the orientability of the local systems is determined.
The case $r=1, s=2$ requires more work, but can still be done.
The one remaining case, $r=3, s=4$, cannot be calculated, the difficulty being that it seems impossible or at the least extremely difficult to determine lowest energy strips.
Before proceeding to the case by case analysis, a lemma on the orientability of the local systems is needed.

\begin{lemma}\label{lemma:case3-line_bundle}
The line bundle $\Theta_{\rp^1}^-$ is trivial if and only if $(-1)^{r''+s''}=1$.
(Recall that $r''$ is defined to be $0$ if $r$ is 1 or 2, and $1$ if $r$ is 3 or 4. Similarly for $s''$.)
\end{lemma}
\begin{proof}
The eigenvalues of the $dg$ action on $T_xX$ for $x\in\rp^1$ are $1,\gamma^r,\gamma^s$.
Thus the bundle $TL|\rp^1$ splits as 
\begin{displaymath}
TL|\rp^1\cong E_0\oplus E_r\oplus E_s,
\end{displaymath}
where $E_0$ is the 1-eigenspace intersected with $TL$, $E_r$ is the $\gamma^r$-eigenspace intersected with $TL$, and $E_s$ is the $\gamma^s$-eigenspace intersected with $TL$.
$E_0$ is orientable, being isomorphic to the tangent space of $\rp^1$.
$E_r$ and $E_s$ are not orientable.
This can be seen by direct calculation as follows:
First, note that since $TL$ is orientable, it suffices to show that $E_r$ is not orientable.
Cover $\rp^1=L\cap gL$ with two coordinate charts $U_1=[-1,1]$ and $U_2=[-1,1]$ via the homeomorphisms
\begin{eqnarray*}
&t\in U_1\mapsto [t:0:0:1:-(1+t^5)^{1/5}]\in L\cap gL=\rp^1,&\\
&s\in U_2\mapsto [1:0:0:s:-(1+t^5)^{1/5}]\in L\cap gL=\rp^1.&
\end{eqnarray*}
Consider the affine coordinates
\begin{eqnarray*}
&(Z_1,Z_2,Z_3,Z_4)\in\cc^4\leftrightarrow[1:Z_1:Z_2:Z_3:Z_4]&\\
&(W_1,W_2,W_3,W_4)\in\cc^4\leftrightarrow[W_1:W_2:W_3:1:W_4].
\end{eqnarray*}
Note that the image of $U_1$ is contained in the affine chart determined by the $Z_i$'s and the image of $U_2$ is contained in the affine chart determined by the $W_i$'s.
In this way, $U_1$ and $U_2$ can be thought of as mapping into $\cc^4$.

At $t$, in affine coordinates, the function $f$ defining $X$ has differential
\begin{displaymath}
df(t)=\sum 5Z_i^4dZ_i=5t^4dZ_1+5(1+t^5)^{4/5}dZ_4.
\end{displaymath}
The $g$ action on the affine coordinates is $g\cdot(Z_1,Z_2,Z_3,Z_4)=(Z_1,\gamma^rZ_2,\gamma^sZ_3,Z_4)$.
The tangent space of $X$ at $t$ is the kernel of $df(t)$, hence is spanned by the vectors
\begin{displaymath}
\frac{\partial}{\partial Z_2},\frac{\partial}{\partial Z_3},(1+t^5)^{4/5}\frac{\partial}{\partial Z_1}-t^4 \frac{\partial}{\partial Z_4}.
\end{displaymath}
Thus the $\gamma^r$ eigenspace of $TX$ at $t$ is spanned by $\frac{\partial}{\partial Z_2}$.
Thus $E_r$ is the real span of $\frac{\partial}{\partial Z_2}$.

Likewise, at $s$, in affine coordinates determined by the $W_i$'s, the function $f$ defining $X$ has differential
\begin{displaymath}
df(s)=\sum 5W_i^4dW_i=5s^4dW_3+5(1+t^5)^{4/5}dW_4.
\end{displaymath}
The $g$ action on the affine coordinates is $g\cdot(W_1,W_2,W_3,W_4)=(\gamma^r W_1,\gamma^sW_2,W_3,W_4)$.
The tangent space of $X$ at $s$ is the kernel of $df(s)$, hence is spanned by the vectors
\begin{displaymath}
\frac{\partial}{\partial W_1},\frac{\partial}{\partial W_2},(1+s^5)^{4/5}\frac{\partial}{\partial W_3}-s^4 \frac{\partial}{\partial W_4}.
\end{displaymath}
Thus the $\gamma^r$ eigenspace of $TX$ at $s$ is spanned by $\frac{\partial}{\partial W_1}$, and $E_r$ is the real span of $\frac{\partial}{\partial W_1}$.

Finally, to see that $E_r$ is nontrivial over $\rp^1$ note that $t=1$ corresponds to $s=1$ and $t=-1$ corresponds to $s=-1$.
Furthermore, by a simple change of coordinates calculation, $\frac{\partial}{\partial Z_2}=\frac{1}{Z_1}\frac{\partial}{\partial W_1}$.
Thus, at $t=1$ it follows that $\frac{\partial}{\partial Z_2}=\frac{\partial}{\partial W_1}$, but at $t=-1$ $\frac{\partial}{\partial Z_2}=-\frac{\partial}{\partial W_1}$.
Since $\frac{\partial}{\partial Z_2}$ and $\frac{\partial}{\partial W_1}$ give trivializations of $E_r$ over $U_1$ and $U_2$, it follows that $E_r$ is not orientable.

The proof of the Lemma then goes through in the same way as Lemma \ref{lemma:case2-line_bundle} to show that the index bundle used to define $\Theta_{\rp^1}^-$ is trivial if and only if $(-1)^{r''+s''}=1$.
\end{proof}

\subsection{The case $r=1,s=2$}
The induced map on the $E_2$ term is defined using the lowest energy strips, which can be found by applying Lemma \ref{lemma:strips-polynomials}:
Let
$$u(z)=[u_0(z):u_1(z):u_2(z):u_3(z):u_4(z)].$$
By the lemma there exists polynomials $p_i$ with real coefficients such that
$$u_i(z)=p_i(z^5),\, i=0,3,4$$
$$u_1(z)=z^2p_1(z^5),$$
$$u_2(z)=z^4p_2(z^5).$$
The condition $u(0),u(\infty)\in L\cap gL$ implies
\begin{itemize}
\item at least one $u_i$ (i=0,3,4) has a nonzero constant term, and
\item $\textrm{max}\sett{deg(u_0),deg(u_3),deg(u_4)}>\textrm{max}\sett{deg(u_1),deg(u_2)}.$
\end{itemize}
Therefore the maps of lowest energy must be of the form
\begin{displaymath}
\begin{array}{rcl}
z&\mapsto&[(a_0z^5+1):10^{1/5}c_1z^2:5^{1/5}c_2z^4:(a_3z^5+b_3):(a_4z^5+b_4)],\\
z&\mapsto&[(a_0z^5+b_0):10^{1/5}c_1z^2:5^{1/5}c_2z^4:(a_3z^5+1):(a_4z^5+b_4)],\textrm{ or}\\
z&\mapsto&[(a_0z^5+b_0):10^{1/5}c_1z^2:5^{1/5}c_2z^4:(a_3z^5+b_3):(a_4z^5+1)].\\
\end{array}
\end{displaymath}
where not all of $a_0$, $a_3$, and $a_4$ are zero.
Consider maps of the first form.
(Once these maps are found, by symmetry the maps of the other forms can be written down.)
\begin{lemma}
$c_1$ and $c_2$ cannot both be zero.
\end{lemma}
\begin{proof}
Assume $c_1=c_2=0$.
If $a_0=0$ then
$$1+(a_3w+b_3)^5+(a_4w+b_4)^5=0$$
for all $w=z^5$.
Examining the $w^5$ and $w^4$ terms gives $a_3=-a_4$ and $b_3=-b_4$, but then this gives $1=0$.
So $a_0\neq 0$.
Then let $w=a_0z^5+1$.
Then
$$w^5+(a_3'w+b_3')^5+(a_4'w+b_4')^5=0.$$
Examining the coefficients again shows this is impossible.
\end{proof}

First assume $c_1$ is non-zero, then (up to $\rr$-tranlation, see Lemma \ref{lemma:strips-r_translation}) it may be assumed that $c_1=\pm1$.
The quintic equation implies that the coefficients satisfy the system of equations
\begin{eqnarray}\label{eq:polyequations}
\nonumber a_0^5+a_3^5+a_4^5&=&0,\\
\nonumber a_0^4+c_2+a_3^4b_3+a_4^4b_4&=&0,\\
a_0^3+a_3^3b_3^2+a_4^3b_4^2&=&0,\\
\nonumber a_0^2\pm1+a_3^2b_3^3+a_4^2b_4^3&=&0,\\
\nonumber a_0+a_3b_3^4+a_4b_4^4&=&0,\\
\nonumber 1+b_3^5+b_4^5&=&0.
\end{eqnarray}
\begin{lemma} 
The system (\ref{eq:polyequations}) has no solutions if $a_0,a_3,a_4,b_3,b_4$ are all non-zero.
\end{lemma}
\begin{proof}
$$a_0^6=a_0^5\cdot a_0=(-a_3^5-a_4^5)(-a_3b_3^4-a_4b_4^4)=$$
$$a_3^6b_3^4+a_3^5a_4b_4^4+a_4^5a_3b_3^4+a_4^6b_4^4.$$
Also
$$a_0^6=(a_0^3)^2=(-a_3^3b_3^2-a_4^3b_4^2)^2=a_3^6b_3^4+2a_3^3b_3^2a_4^3b_4^2+a_4^6b_4^4.$$
Subtracting these two expressions gives
$$0=a_3^6b_3^4+a_3^5a_4b_4^4+a_4^5a_3b_3^4+a_4^6b_4^4-a_3^6b_3^4-2a_3^3b_3^2a_4^3b_4^2-a_4^6b_4^4$$
$$=a_3^5a_4b_4^4-a_3^3b_3^2a_4^3b_4^2+a_4^5a_3b_3^4-a_3^3b_3^2a_4^3b_4^2$$
$$=a_3^3a_4b_4^2(a_3^2b_4^2-b_3^2a_4^2)+a_4^3a_3b_3^2(a_4^2b_3^2-a_3^2b_4^2)$$
$$=(a_3^3a_4b_4^2-a_4^3a_3b_3^2)(a_3^2b_4^2-b_3^2a_4^2)$$
$$=a_3a_4(a_3^2b_4^2-a_4^2b_3^2)(a_3^2b_4^2-b_3^2a_4^2)$$
$$=a_3a_4(a_3b_4-a_4b_3)^2(a_3b_4+b_3a_4)^2.$$
Since $a_3,a_4$ are assumed to be non-zero, this implies
$$\frac{a_3}{a_4}=\pm\frac{b_3}{b_4}.$$

First consider the case
$$\frac{a_3}{a_4}=+\frac{b_3}{b_4}.$$
Let $\lambda=a_4/a_3$.
%Then $a_0^5+a_3^5+a_4^5=0$ implies
%$$\frac{a_0^5}{a_3^5}=-(1+\lambda^5),$$
Then $a_0+a_3b_3^4+a_4b_4^4=0$ implies
$$\frac{a_0}{a_3b_3^4}=-(1+\lambda^5),$$
and $a_0^3+a_3^3b_3^2+a_4^3b_4^2=0$ implies
$$\frac{a_0^3}{a_3^3b_3^2}=-(1+\lambda^5).$$
Thus
$$\frac{a_0}{a_3b_3^4}=\frac{a_0^3}{a_3^3b_3^2},$$
and so
$$a_0=\pm\frac{a_3}{b_3}.$$
%Suppose $a_0=\frac{a_3}{b_3}$.
Then $a_0^2\pm1+a_3^2b_3^3+a_4^2b_4^3=0$ implies
$$\frac{a_3^2}{b_3^2}+1+a_3^2b_3^3+a_4^2b_4^3=0$$
$$a_3^2(1+b_3^5+b_4^5)\pm b_3^2=0.$$
Since $1+b_3^5+b_4^5=0$ this implies $b_3=0$, a contradiction.

Therefore there are no solutions if $\frac{a_3}{a_4}=+\frac{b_3}{b_4}$.
The case $\frac{a_3}{a_4}=-\frac{b_3}{b_4}$ is very similar.
\end{proof}

Thus a solution can only occur if one of $a_0,a_3,b_3,a_4,b_4$ is $0$.
The equations are easy to solve under this assumption, and in fact only $a_0=0$ yields real roots.
\begin{lemma}
The only real solutions of the system (\ref{eq:polyequations}) are
$$a_0=0,a_3=-a_4=\pm2^{-1/5},b_3=b_4=-2^{-1/5},c_1=c_2=1.$$
\end{lemma}

The case $c_2$ is non-zero can be solved in the same way and yields the same solutions.
%Permuting the $0,3,4$ coordinates also yields solutions.
This proves the first part of the next lemma.
\begin{lemma}\label{lemma:case3-maps}
There are 6 lowest energy holomorphic strips.
They are:
$$z\mapsto[1:10^{1/5}z^2:5^{1/5}z^4:\pm 2^{-1/5}z^5-2^{-1/5}:\mp 2^{-1/5}z^5-2^{-1/5}],$$
$$z\mapsto[\pm 2^{-1/5}z^5-2^{-1/5}:10^{1/5}z^2:5^{1/5}z^4:1:\mp 2^{-1/5}z^5-2^{-1/5}],$$
$$z\mapsto[\pm 2^{-1/5}z^5-2^{-1/5}:10^{1/5}z^2:5^{1/5}z^4:\mp 2^{-1/5}z^5-2^{-1/5}:1].$$
Moreover, each row represents a pair of conjugate strips.
\end{lemma}
\begin{proof}
It remains to check the last statement.
Let $u(z)=[1:10^{1/5}z^2:^{1/5}z^4:2^{-1/5}z^5-2^{-1/5}:-2^{-1/5}z^5-2^{-1/5}]$.
The conjugate of $u(z)$ is $\tilde u(z)=g\overline{u(\sigma\bar z)}$.
That is,
\begin{eqnarray*}
&&\tilde u(z)=g\cdot[1:10^{1/5}\gamma^{-1} z^2:5^{1/5}\gamma^{-2}z^4:-2^{-1/5}z^5-2^{-1/5}:2^{-1/5}z^5-2^{-1/5}]=\\
&&[1:10^{1/5}z^2:5^{1/5}z^4:-2^{-1/5}z^5-2^{-1/5}:2^{-1/5}z^5-2^{-1/5}].
\end{eqnarray*}
Thus the first row is a pair of conjugate strips.
The proof for the remaining rows is similar.
\end{proof}

The next step is to calculate the cokernel of $D_u\dbar$ for these lowest energy strips.
By symmetry, the dimension of the cokernel is the same for all the strips in Lemma \ref{lemma:case3-maps}, so it suffices to do the calculation for the map
$$u(z)=[1:10^{1/5}z^2:5^{1/5}z^4:2^{-1/5}z^5-2^{-1/5}:-2^{-1/5}z^5-2^{-1/5}].$$

By Proposition \ref{prop:cokernel-main_prop}, the cokernel of $D_u\dbar$ can be identified with the elements of $H^0(\cp^1,u^*T_X^*)$ that vanish at $0$ and $\infty$ and are fixed by the $\z_5$ and $\z_2$ actions.
To calculate $H^0(\cp^1,u^*T_X^*)$, it suffices to calculate how $u^*T_X$ splits as a direct sum of line bundles.

The splitting can be calculated as in \cite{katz}.
Namely, consider the exact sequences of sheaves
\begin{eqnarray}
\label{eq:case3-sequence_1}
&0\to u^*\mathcal O\to u^*\mathcal O(1)^{\oplus 5}\cong\mathcal O(5)^{\oplus 5}\to u^*T_{\cc P^4}\to 0,\\
\label{eq:case3-sequence_2}
&0\to u^*T_X\to u^*T_{\cc P^4}\to u^*N_{X,\cc P^4}\to 0.
\end{eqnarray}
The first map in the first sequence is given by (the homogeneous version of) $u$ and hence is
\begin{equation}\label{eq:case3-map_1}
\left[
\begin{array}{c}
Z_1^5\\
10^{1/5}Z_0^2Z_1^3\\
5^{1/5}Z_0^4Z_1\\
2^{-1/5}Z_0^5-2^{-1/5}Z_1^5\\
-2^{-1/5}Z_0^5-2^{-1/5}Z_1^5
\end{array}
\right].
\end{equation}
$u^*T_{\cp^4}$ splits as a direct sum of line bundles, say
\begin{displaymath}
u^*T_{\cp^4}\cong\mathcal O(d_1)\oplus\mathcal O(d_2)\oplus\mathcal O(d_3)\oplus\mathcal O(d_4).
\end{displaymath}
With this identification, the map
\begin{displaymath}
u^*\mathcal O(1)^{\oplus 5}\cong\mathcal O(5)^{\oplus 5}\to u^*T_{\cp^4}\cong\mathcal O(d_1)\oplus\mathcal O(d_2)\oplus\mathcal O(d_3)\oplus\mathcal O(d_4)
\end{displaymath}
in the sequence (\ref{eq:case3-sequence_1}) can be though of as a $4\times 5$ matrix of homogeneous polynomials in the variables $Z_0$ and $Z_1$.
Since the sequence is exact, the rows must generate the relations among the polynomials in (\ref{eq:case3-map_1}).
That is, if $p_0,\ldots,p_4\in \cc[Z_0,Z_1]$ are polynomials such that
\begin{eqnarray*}
0&=&p_0Z_1^5+p_110^{1/5}Z_0^2Z_1^3+p_25^{1/5}Z_0^4Z_1+p_3(2^{-1/5}Z_0^5-2^{-1/5}Z_1^5)+\\
&&p_4(-2^{-1/5}Z_0^5-2^{-1/5}Z_1^5),
\end{eqnarray*}
then $\left[\begin{array}{ccccc}p_0 & p_1 & p_2 & p_3 & p_4\end{array}\right]$ is in the row space of the matrix.
Therefore the map can be taken to be, for example,  
\begin{equation}\label{eq:firstsequencesecondmap}
\left[
\begin{array}{ccccc}
2^{4/5} &0&0&1&1\\
10^{1/5}Z_0^2& -Z_1^2&0&0&0\\
0&Z_0^2&-2^{1/5}Z_1^2&0&0\\
0&0&2^{4/5}Z_0&5^{1/5}Z_1&-5^{1/5}Z_1
\end{array}
\right].
\end{equation}
(Clearly the rows of this matrix are relations among the polynomials.
The easiest way to see that they generate all the relations is to use a computer.
An example that shows how to do this will be given below.)
The sequence (\ref{eq:case3-sequence_1}) is a graded sequence, so comparing degrees shows that
\begin{equation}
u^*T_{\cc P^4}\cong\mathcal O(5)\oplus\mathcal O(7)\oplus\mathcal O(7)\oplus\mathcal O(6).
\end{equation}

Now consider the second sequence (\ref{eq:case3-sequence_2}).
$u^*N$ is a line bundle, so the map $u^*T_{\cp^4}\to u^*N$ can be identified with a $1\times 4$ matrix of homogeneous polynomials, let this matrix be
$
\left[
\begin{array}{cccc}
a_1 & a_2 & a_3 & a_4
\end{array}
\right]
$.
The composition $$u^*\mathcal O(1)^{\oplus 5}\to u^*T_{\cp^4}\to u^*N$$ is nothing more than $df$, where $f=X_0^5+\cdots+X_4^5$ is the defining equation for $X$.
Thus
\begin{eqnarray*}
&\left[
\begin{array}{cccc}
a_1 & a_2 & a_3 & a_4
\end{array}
\right]
\left[
\begin{array}{ccccc}
2^{4/5} &0&0&1&1\\
10^{1/5}Z_0^2& -Z_1^2&0&0&0\\
0&Z_0^2&-2^{1/5}Z_1^2&0&0\\
0&0&2^{4/5}Z_0&5^{1/5}Z_1&-5^{1/5}Z_1
\end{array}
\right]
=\\
&\left[
\begin{array}{ccccc}
5X_0^4 & 5X_1^4 & 5X_2^4 & 5X_3^4 & 5X_4^4
\end{array}
\right].
\end{eqnarray*}
This leads to the system of equations

\begin{eqnarray*}
2^{4/5}a_1+10^{1/5}Z_0^2a_2 &=& 5Z_1^{20}\\
-Z_1^2a_2+Z_0^2a_3 &=& 5\cdot 10^{4/5}\cdot Z_0^8Z_1^{12}\\
-2^{1/5}Z_1^2a_3+2^{4/5}Z_0a_4 &=& 5^{9/5}\cdot Z_0^{16}Z_1^4\\
a_1+5^{1/5}Z_1a_4 &=& 5(2^{-1/5}Z_0^5-2^{-1/5}Z_1^5)^4\\
a_1-5^{1/5}Z_1a_4 &=& 5(-2^{-1/5}Z_0^5-2^{-1/5}Z_1^5)^4.
\end{eqnarray*}
Thus
\begin{eqnarray*}
a_1&=&5\cdot 2^{-4/5}(Z_1^{20}+6Z_0^{10}Z_1^{10}+Z_0^{20}),\\
a_2 &=& -10^{1/5}(Z_0^8Z_1^{10}+Z_0^{18}),\\
a_3&=&(5\cdot 10^{4/5}-10^{1/5})Z_0^6Z_1^{12}-10^{1/5}Z_0^{16}Z_1^2 ,\\
a_4&=&(5^{9/5}2^{-4/5}-2^{-3/5}10^{1/5})Z_0^{15}Z_1^4+2^{-3/5}(5\cdot 10^{4/5}-10^{1/5})Z_0^5Z_1^{14}.
\end{eqnarray*}
The map $u^*T_X\to u^*T_{\cc P^4}$ is given by the relations among the $a_i$, which with a computer can be calculated to be
\begin{displaymath}
\left[
\begin{array}{ccc}
%0 & \frac{2^{4/5}}{10}Z_0^6 & \frac{2^{4/5}\cdot 3119}{999990}Z_0^5Z_1^2 \\
%-\frac{99808}{499995}Z_1^4 & -\frac{1}{5}Z_0^8 & -\frac{9646876}{1562484375}Z_0^7Z_1^2 \\
%-\frac{191}{499995}Z_0^2Z_1^2 & \frac{99999}{1559500000}Z_1^8 & -\frac{1}{15625}Z_0^9 \\
%-\frac{1}{499995}Z_0^3 & \frac{1}{1559500000}Z_0Z_1^6 & \frac{1}{499995}Z_1^8
0 & dZ_0^6 & hZ_0^5Z_1^2 \\
aZ_1^4 & eZ_0^8 & iZ_0^7Z_1^2 \\
bZ_0^2Z_1^2 & fZ_1^8 & jZ_0^9 \\
cZ_0^3 & gZ_0Z_1^6 & kZ_1^8
\end{array}
\right],
\end{displaymath}
for some constants $a,b,c,d,e,f,g,h,i,j,k$.
For example, in Magma this can be calculated by entering the following commands:
\begin{verbatim}

S<c,d>:=PolynomialRing(RationalField(),2);
F<a,b>:=quo<S|c^5-2,d^5-5>;
R<z,w>:=PolynomialRing(F,2);
a1:=10/a^4*(w^20+6*z^10*w^10+z^20);
a2:=-a*b*(z^8*w^10+z^18);
a3:=(5*a^4*b^4-a*b)*z^6*w^12-a*b*z^16*w^2;
a4:=(b^9/a^4-b/a^2)*z^15*w^4+(5*a^4*b^4-a*b)/a^3*z^5*w^14;
N:=SyzygyModule([a1,a2,a3,a4]);
MinimalBasis(N);

\end{verbatim}
Comparing degrees proves
$$u^*T_X\cong \mathcal O(3)\oplus\mathcal O(-1)\oplus\mathcal O(-2),$$
and therefore
$$ u^*T_X^*\cong\mathcal O(-3)\oplus\mathcal O(1)\oplus\mathcal O(2).$$

Recall that the cokernel consists of those holomorphic sections that vanish at $z=0$ and $z=\infty$ (the affine coordinate $z$ corresponds to the homogeneous coordinate $[z:1]$) and are fixed by the $\z_5$ and $\z_2$ actions.
A basis for the subspace spanned by sections that vanish at $0$ and $\infty$ is $\zeta=(0,0,Z_0Z_1)$.
The $\z_2$ fixed subspace is the $\rr$ span of $\zeta$.

To determine the $\z_5$ fixed subspace, it will be helpful to have a more useful formula for the $\z_5$ action.
Via the exact sequence of sheaves
\begin{displaymath}
0\to N^*\to T_{\cp^4}^*\to T_X^*\to 0,
\end{displaymath}
a holomorphic section of $u^*T_X^*$ pulls back to a meromorphic section of $u^*T_{\cp^4}^*$.
Likewise, a meromorphic section $s(z)$ of $u^*T_{\cp^4}^*$, via the exact sequence of sheaves
\begin{displaymath}
0\to u^*T_{\cp^4}^*\to u^*\mathcal O(1)^{*\oplus 5}\to \mathcal O^*\to 0,
\end{displaymath}
can be identified with a meromorphic section of $u^*\mathcal O(1)^{*\oplus 5}$, which in turn can be thought of as a tuple of rational functions.
Explicitly, if $\ell=u^*\mathcal O(1)^*$ is the pullback of the tautological bundle of $\cp^4$, then the tuple
\begin{equation}\label{eq:case3-form}
s(z)=\biggl(\frac{n_0(z)}{m_0(z)},\ldots,\frac{n_4(z)}{m_4(z)}\biggr)
\end{equation}
represents the section
\begin{displaymath}
\biggl(\frac{n_0(z)}{m_0(z)}u(z),\ldots,\frac{n_4(z)}{m_4(z)}u(z)\biggr)\in \ell^{\oplus 5}=u^*\mathcal  O(1)^*.
\end{displaymath}

Via the sequence (\ref{eq:case3-sequence_1}), the tangent bundle $T_{\cp^4}$ is identified with $(\cc^5/\ell)\otimes \ell^*$, and thus
\begin{displaymath}
T_{\cp^4}^*\cong (\cc^5/\ell)^*\otimes\ell.
\end{displaymath}
If $e_0,\ldots,e_4$ is the standard basis of $\cc^5$ and $e_0^*,\ldots,e_4^*$ is the dual basis, then under this isomorphism the section $s$ represents the section
\begin{displaymath}
s(z)=\biggl(\frac{n_0(z)}{m_0(z)}e_0^*+\cdots+\frac{n_4(z)}{m_4(z)}e_4^*\biggr)\otimes u(z)\in(\cc^5/\ell)^*\otimes\ell.
\end{displaymath}
\begin{lemma}\label{lemma:case3-action}
Let $\lambda_u\in\cc^*$ be such that $\lambda_uu(\gamma z)=g^2u(z)$.
Let $s=(\frac{n_0}{m_0},\ldots,\frac{n_4}{m_4})$ be a section of $u^*T_X^*$, written in the form (\ref{eq:case3-form}).
Let $\tilde s(z)$ be the image of $s(z)$ under the action of the generator of the $\z_5$.
Then
\begin{displaymath}
\tilde s(z)=\biggl(\frac{\lambda_u^{-1}g_0^2n_0(\gamma z)}{m_0(\gamma z)},\ldots,\frac{\lambda_u^{-1}g_4^2n_4(\gamma z)}{m_4(\gamma z)}\biggr).
\end{displaymath}
\end{lemma}
\begin{proof}
Let $v\in T_{\cp^4}=(\cc^5/\ell)\otimes\ell^*$ be the tangent vector
\begin{displaymath}
v=(v_0,\ldots,v_4)\otimes u(z)^*.
\end{displaymath}
The proof then follows from a straightforward calculation:
\begin{eqnarray*}
&\pairing{\tilde s(z)}{v}=\pairing{s(\gamma z)}{dg^2\cdot v}=\pairing{s(\gamma z)}{(g_0^2v_0e_0+\cdots+g_4^2v_4e_4)\otimes (g^2u(z))^*}=\\
&\pairing{s(\gamma z)}{(g_0^2v_0e_0+\cdots+g_4^2v_4e_4)\otimes (\lambda_uu(\gamma z))^*}=\\
&\pairing{(\frac{n_0(\gamma z)}{m_0(\gamma z)}e_0^*+\cdots+\frac{n_4(\gamma z)}{m_4(\gamma z)}e_4^*)\otimes u(\gamma z)}{ (\lambda_u^{-1}g_0^2v_0e_0+\cdots+\lambda_u^{-1}g_4^2v_4e_4)\otimes u(\gamma z)^*}=\\
&\frac{\lambda_u^{-1}g_0^2n_0(\gamma z)v_0}{m_0(\gamma z)}+\cdots+\frac{\lambda_u^{-1}g_4^2n_4(\gamma z)v_4}{m_4(\gamma z)}=\\
&\pairing{(\frac{\lambda_u^{-1}g_0^2n_0(\gamma z)}{m_0(\gamma z)}e_0^*+\cdots+\frac{\lambda_u^{-1}g_4^2n_4(\gamma z)}{m_4(\gamma z)}e_4^*)\otimes u(z)}{v}.
\end{eqnarray*}
\end{proof}

To determine the $\z_5$ action on $\zeta=(0,0,Z_0Z_1)$, $\zeta$ needs to first be written in the form (\ref{eq:case3-form}).
The map $u^*T_{\cp^4}^*\to T_X^*$ is
\begin{displaymath}
\left[
\begin{array}{cccc}
0 & aZ_1^4 & bZ_0^2Z_1^2 & cZ_0^3\\
dZ_0^6 & eZ_0^8 & fZ_1^8 & gZ_0Z_1^6\\
hZ_0^5Z_1^2 & iZ_0^7Z_1^2 & jZ_0^9 & kZ_1^8
\end{array}
\right].
\end{displaymath}
$\zeta$ pulls back to, for example, the section $(p,q,r,s)$ where
\begin{eqnarray*}
p&=&\frac{beZ_0^6}{(hbc-ibd)Z_0^{10}Z_1-ahfZ_1^{11}},\\
q&=&\frac{-bdZ_0^4}{(hbc-ibd)Z_1Z_0^{10}-hfaZ_1^{11}},\\
r&=&\frac{adZ_0^2Z_1}{(hbc-ibd)Z_0^{10}-hfaZ_1^{10}},\\
s&=&0.
\end{eqnarray*}
Under the map $u^*T_{\cp^4}^*\to u^*\mathcal O(1)^{*\oplus 5}$, $(p,q,r,s)$ pushes forward to
\begin{displaymath}
(2^{-1/5}p-Z_0^2q,-\frac{1}{5}Z_1^2q-\frac{1}{5}Z_0^2r,\frac{1}{10}Z_1^2r-\frac{1}{10}Z_0s,\frac{1}{2}q+2^{-4/5}Z_1s,\frac{1}{2}p+2^{-4/5}Z_1s).
\end{displaymath}
In affine coordinates this is
\begin{equation}\label{eq:case3-section}
(2^{-1/5}p-z^2q,-\frac{1}{5}q-\frac{1}{5}z^2r,\frac{1}{10}r-\frac{1}{10}zs,\frac{1}{2}q+2^{-4/5}s,\frac{1}{2}p+2^{-4/5}s).
\end{equation}
If $\lambda_u$ is such that $\lambda_uu(\gamma z)=g^2u(z)$ then $\lambda_u=1$.
By Lemma \ref{lemma:case3-action}, under the $\z_5$ action (\ref{eq:case3-section}) maps to
\begin{displaymath}
\gamma (2^{-1/5}p-z^2q,-\frac{1}{5}q-\frac{1}{5}z^2r,\frac{1}{10}r-\frac{1}{10}zs,\frac{1}{2}q+2^{-4/5}s,\frac{1}{2}p+2^{-4/5}s).
\end{displaymath}
Therefore the $\z_5$ action on $\zeta$ is
\begin{displaymath}
\zeta\mapsto \gamma\zeta.
\end{displaymath}
Therefore, the $\z_5$ fixed subspace is $0$.
This proves
\begin{lemma}
The maps in Lemma \ref{lemma:case3-maps} are regular.
\end{lemma}
Note that the maps are regular even though they have Maslov index 0.
This is because $L\cap gL$ is positive dimensional, so the relevant linearized $\dbar$ operator is
\begin{displaymath}
D_u\dbar:T_pR_h\oplus T_qR_{h'}\oplus W^{1,p;\delta}_\lambda\to L^{p;\delta}.
\end{displaymath}
The presence of $T_pR_h\oplus T_qR_{h'}$ forces additional boundary values at $\pm\infty$ on elements of the cokernel.
If these conditions are not imposed, then $(0,0,Z_0^2)$, $(0,0,Z_0Z_1)$, $(0,0,Z_1^2)$, $(0,Z_0,0)$, and $(0,Z_1,0)$ are the holomorphic sections to consider, and it can be checked that two of them are fixed by the $\z_5$ action.
Hence the cokernel of $D_u\dbar$ without the $T_pR_h\oplus T_qR_{h'}$ summand is two-dimensional, as expected.

\begin{theorem}\label{thm:case3-main_thm}
$$HF(L,gL;\Lambda_{nov})\cong 0.$$
\end{theorem}
\begin{proof}
By Lemma \ref{lemma:case3-maps} there are three pairs of conjugate holomorphic strips.
These strips are regular by the previous lemma, and by Proposition \ref{prop:conjugation-main_prop} each strip of a conjugate pair contributes the same sign to the Floer coboundary operator.
By the discussion at the beginning of this section, the only part of the coboundary operator that needs to be considered on the $E_2$ page is 
\begin{displaymath}
n_{0,0}:H^0(\rp^1;\qq)=\qq\to H^1(\rp^1;\qq)=\qq,
\end{displaymath}
and this map is determined by $n_{0,0}([\rp^1])$.
(Recall that $H^*$ is homology with cohomological grading.
In particular, $H^0$ contains the fundamental class.)
Since the strips are regular, $n_{0,0}([\rp^1])=ev_{+\infty*}([\mathcal M\times_{ev_{-\infty}}[\rp^1]])$, where $\mathcal M$ is the moduli space that consists of the six strips.
Since each strip in a conjugate pair has the same sign, it follows that $ev_{+\infty*}[\mathcal M\times_{ev_{-\infty}}[\rp^1]]$ is homologous to $\pm3[pt]$ or $\pm1[pt]$.
Either way $n_{0,0}\neq0$ and hence the $E_3$ term of the spectral sequence consists entirely of $0$'s.
\end{proof}

\subsection{The cases where $(-1)^{r''+s''}=-1$}
This includes four cases: $r=1,s=3$; $r=1,s=4$; $r=2,s=3$; and $r=2,s=4$.
By Lemma \ref{lemma:case3-line_bundle}, in these cases the line bundle $\Det(T\rp^1)\otimes \Theta_{\rp^1}^-$ over $\rp^1$ is non-trivial.
Since $\rp^1$ is orientable, it follows that
\begin{displaymath}
H^*(\rp^1;\Det(T\rp^1)\otimes\Theta_{\rp^1}^-;\qq)=0.
\end{displaymath}
(See \cite{hatcher} Example 3H.3 for the proof of this fact.
Remember that $H^*$ stands for homology, with cohomological degree grading.)
The spectral sequence for Floer cohomology then immediately proves
\begin{theorem}\label{thm:case3-main_thm2}
\begin{displaymath}
HF(L,gL;\Lambda_{nov})\cong 0.
\end{displaymath}
\end{theorem}

\subsection{The remaining case: $r=3,s=4$}
Arguing as before, the degree 5 strips (if they exist) must be of the form
\begin{displaymath}
u(z)=[a_0z^5+b_0:5c_1z:10c_2z^3:a_3z^5+b_3:a_4z^5+b_4]
\end{displaymath}
where not all of $b_0,b_3,b_4$ are $0$ and not both of $c_1$ and $c_2$ are $0$ (by Lemma \ref{lemma:strips-polynomials}).
Without loss of generality assume that $b_0=1$ and either $c_1=\pm1$ or $c_2=\pm1$.
If, for example, $c_1=1$ then the coefficients must satisfy the system of equations
\begin{eqnarray*}
a_0^5+a_3^5+a_4^5&=&0,\\
a_0^4+a_3^4b_3+a_4^4b_4 &=&0,\\
a_0^3+c_2^5+a_3^3b_3^2+a_4^3b_4^2 &=& 0,\\
a_0^2+a_3^2b_3^3+a_4^2b_4^3 &=& 0,\\
a_0+1+a_3b_3^4+a_4b_4^4 &=& 0,\\
1+b_3^5+b_4^4 &=& 0.
\end{eqnarray*}
These equations can be ``simplified'' by entering them into a computer and using Groebner bases.
For example, entering the following commands into Magma will do this:
\begin{verbatim}
R<a0,c2,a3,b3,a4,b4>:=PolynomialRing(RationalField(),6);
I:=ideal<R|[a0^5+a3^5+a4^5,a0^4+a3^4*b3+a4^4*b4,
a0^3+c2^5+a3^3*b3^2+a4^3*b4^2,a0^2+a3^2*b3^3+
a4^2*b4^3,a0+1+a3*b3^4+a4*b4^4,1+b3^5+b4^5]>;
GroebnerBasis(I);
\end{verbatim}
However the output is prodigious and not at all ``simplified''.
It seems, then, that this case is out of reach of the techniques developed in this paper.
Investigation is left to further research.

\section{Case (4): $L\cap gL\cong\rp^1\amalg \rp^0$}
Assume without loss of generality that $g=g_1^rg_2^r$ where $0<r\leq 4$.
Then $L\cap gL=\rp^1\amalg\rp^0$, where the $\rp^1$ component is
\begin{displaymath}
\rp^1=\sett{[X_0:0:0:X_3:X_4]\in X}
\end{displaymath}
and the $\rp^0$ component is
\begin{displaymath}
\rp^0=\sett{[0:1:-1:0:0]}.
\end{displaymath}

If $u$ is a strip that starts and stops in the same component, then for dimension reasons (that is, arguing as in Theorem \ref{thm:case1-main_thm}) the only way it can have an effect on the Floer coboundary operator is if it starts and stops in $\rp^1$, in which case it has an effect on the map $H^0(\rp^1)\to H^1(\rp^1)$ in the spectral sequence.

If $u$ is a strip that starts in $\rp^1$ and stops in $\rp^0$ then by Lemmas \ref{lemma:gradings-formula} and \ref{lemma:real_lagrangians-angles} the Maslov index of $u$ is
\begin{displaymath}
\frac{2}{5}(3(5-r)'-2r')=\left\{
\begin{array}{ll}
1&,\ r=1,\\
-1&,\ r=2,\\
2&,\ r=3,\\
0&,\ r=4.
\end{array}
\right.
\end{displaymath}
If $S$ is a chain in $\rp^1$ then
\begin{displaymath}
\dim n_{0,0}(S)=\dim(S)+0-1+\mu(u)-1.
\end{displaymath}
If $n_{0,0}(S)\neq 0$ then $0=\dim n_{0,0}(S)$, so $\dim(S)=2-\mu(u)$.
$\dim(S)$ is $1$ or $0$.
If $\dim(S)=1$, then $n_{0,0}(S)$ has the possibility of being non-zero when $r=1$.
If $\dim(S)=0$, then $n_{0,0}(S)$ has the possibility of being non-zero when $r=3$.

Going the other way, if $u$ is a strip that starts in $\rp^0$ and ends in $\rp^1$ then the Maslov index of $u$ is the negative of that listed above.
Also, if $S$ is a chain in $\rp^0$, then
\begin{displaymath}
\dim n_{0,0}(S)=\dim(S)+1-0+\mu(u)-1=\mu(u).
\end{displaymath}
Thus $n_{0,0}(S)$ has the possibility of being non-zero if $r=2$ or $r=4$.

In summary, the possible non-zero maps between $H^*(\rp^1)$ and $H^*(\rp^0)$ are
\begin{itemize}
\item $r=1$: $H^0(\rp^1)\to H^0(\rp^0)$,
\item $r=2$: $H^0(\rp^0)\to H^0(\rp^0)$,
\item $r=3$: $H^1(\rp^1)\to H^0(\rp^0)$, and
\item $r=4$: $H^0(\rp^0)\to H^1(\rp^1)$.
\end{itemize}

Consequently, in each case there are two moduli spaces that need to be described:
the moduli space of strips that start and stop on $\rp^1$, and either the moduli space of strips that connect $\rp^1$ to $\rp^0$ or vice versa.
A priori many different stable maps enter into these particular moduli spaces.
For example the moduli space of strips that start on $\rp^1$ and end on $\rp^0$ could potentially include a stable map with two strip components: one strip connecting $\rp^1$ to $\rp^1$ followed by another strip connecting $\rp^1$ to $\rp^0$.
However, if attention is restricted to the lowest energy strips, this behavior cannot happen.
That is, the lowest energy strips can only have one domain component.
It is these lowest energy strips that will be described and used to determine the operator on the $E_2$ page of the spectral sequence.

The moduli spaces will be examined in detail in a moment, but first a lemma:
\begin{lemma}\label{lemma:case4-line_bundle}
The line bundle $\Theta_{\rp^1}^-$ is trivial.
\end{lemma}
\begin{proof}
The eigenvalues of the $dg$ action on $T_xX$ for $x\in\rp^1$ are $1,\gamma^r,\gamma^r$.
Let $E_0$ denote the subbundle of $TL|\rp^1$ corresponding to the $1$-eigenspace, and let $E_r$ denote the subbundle corresponding to the $\gamma^r$-eigenspace.
Then $E_0=T\rp^1$, and hence $TL|\rp^1$ splits as
\begin{displaymath}
TL|\rp^1\cong T\rp^1\oplus E_r.
\end{displaymath}
It follows that $E_r$ is orientable.
Arguing as in the proof of Lemma \ref{lemma:case2-line_bundle}, it follows that $\Theta_{\rp^1}^-$ is orientable.
\end{proof}
It follows that
\begin{displaymath}
H^*(\rp^1;\Det(T\rp^1)\otimes\Theta_{\rp^1}^-)\cong H^*(\rp^1).
\end{displaymath}
Hence the local systems can be ignored in this section.
%The goal is to first describe these moduli spaces in the lowest energy setting, and then use these descriptions to determine the maps on the $E_2$ term of the spectral sequence.

%\subsection{Strips from $\rp^1$ to $\rp^1$}
%Let $u(z)=[u_0(z):u_1(z):u_2(z):u_3(z):u_4(z)]$ be a strip that goes from $\rp^1$ to $\rp^1$.
%$\rp^1$ consists of points of the form $[X_0:0:0:X_1:X_2]$, therefore by Lemma ?? $u$ is of the form
%\begin{displaymath}
%u(z)=[p_0(z^5):z^kp_1(z^5):z^kp_2(z^5):p_3(z^5):p_4(z^5)]
%\end{displaymath}
%where $k\equiv 2r\textrm{ (mod 5)}$, $0<k<5$ and $\textrm{max}\sett{\deg u_0,\deg u_3,\deg u_4}>\textrm{max}\sett{\deg u_1,\deg u_2}$.
%Therefore the lowest energy strips are of the form
%\begin{displaymath}
%u(z)=[a_0z^5+b_0:b_1z^k:b_2z^k:a_3z^5+b_3:a_4z^5+b_4].
%\end{displaymath}
%\begin{lemma}
%The lowest energy strips are of the form
%\begin{displaymath}
%u(z)=[a_0z^5+b_0:b_1z^k:b_2z^k:a_3z^5+b_3:a_4z^5+b_4].
%\end{displaymath}
%\end{lemma}
%\begin{proof}
%By Lemma ??, $b_1$ and $b_2$ cannot both be zero.
%\end{proof}

\subsection{Strips from $\rp^1$ to $\rp^1$}\label{section:case4-rp1_strips}
In this section, all the lowest energy strips that start and stop in the $\rp^1$ component will be found (for all the cases).
It will be seen that these strips have energy $E_0/2$ (i.e. they are given by degree 5 polynomals), which, as will be seen later, is greater than the lowest energy strips that go between the $\rp^0$ and $\rp^1$ components.
Therefore an exact description of these strips is not really needed.
However, they provide nice examples of the bubbling phenomenon and Gromov's compactness theorem, so it seems worthwhile to spend some time describing them.

If $u=[u_0:u_1:u_2:u_3:u_4]$ is a strip from $\rp^1$ to $\rp^1$, then not all of $u_0(0),u_3(0),u_4(0)$ can be $0$.
Therefore, by Lemma \ref{lemma:strips-polynomials}, the $u_i$'s are of the form
\begin{eqnarray*}
u_0(z)&=&p_0(z^5),\\
u_1(z)&=& z^kp_1(z^5),\\
u_2(z)&=& z^kp_2(z^5),\\
u_3(z)&=&p_3(z^5),\\
u_4(z)&=&p_4(z^5),
\end{eqnarray*}
where $0<k\leq 4$ is such that $\gamma^k=\gamma^{2r}$.
\begin{lemma}
$u$ is a lowest energy strip if and only if $u$ is of the form
\begin{displaymath}
u(z)=[a(z^5\pm 1):dz^k:-dz^k:b(z^5\pm 1):c(z^5\pm 1)],
\end{displaymath}
where $a^5+b^5+c^5=0$, not all of $a,b,c$ are $0$, and $d>0$.
\end{lemma}
\begin{proof}
Since $u$ ends in $\rp^1$, it must be true that
\begin{displaymath}
\max\sett{\deg(u_0),\deg(u_3),\deg(u_4)}>\max\sett{\deg(u_1),\deg(u_2)}.
\end{displaymath}
Therefore the lowest energy strips are of the form
\begin{displaymath}
u(z)=[a_0z^5+b_0:b_1z^k:b_2z^k:a_3z^5+b_3:a_4z^5+b_4].
\end{displaymath}
Not all of $b_0,b_3,b_4$ can be zero, so assume that $b_0\neq0$, and hence assume that $b_0=1$.
It is easy to see that $a_0\neq0$, hence by $\rr$ translation it may be assumed that $a_0=\pm 1$.
Assume that $a_0=1$; the case $a_0=-1$ is similar.
Therefore $u$ is of the form
\begin{displaymath}
u(z)=[z^5+1:b_1z^k:b_2z^k:a_3z^5+b_3:a_4z^5+b_4].
\end{displaymath}

Suppose $a_3=0$.
Then it is easy to see that $a_4=-1$, $b_4=-1$, and $b_1=-b_2$.
That is, $u$ is of the desired form.

The same result holds if $a_4=0$.
Likewise, the same result holds if $b_3$ or $b_4$ is $0$, as can be seen by letting $w=1/z$ and applying the same argument.
So it may be assumed that $a_3,a_4,b_3,b_4$ are all non-zero.

Now suppose $a_3=b_3$.
Then
\begin{eqnarray*}
1+a_3^5+a_4^5&=&0,\\
1+a_3^5+a_4^ib_4^{5-i}&=&0,\\
\end{eqnarray*}
where $i\neq k$.
This shows that $a_4^i(a_4^{5-i}-b_4^{5-i})=0$.
Thus $a_4=b_4$ and $u$ is of the desired form.
So it may be assumed also that $a_3\neq b_3$ and $a_4\neq b_4$.

If $k$ is 1 or 2 then
\begin{eqnarray*}
1+a_3^5+a_4^5&=&0,\\
1+a_3^4b_3+a_4^4b_4&=&0,\\
1+a_3^3b_3^2+a_4^3b_4^2&=&0.
\end{eqnarray*}
Subtracting the first two and second two equations leads to the equations
\begin{eqnarray*}
a_3^4(a_3-b_3)+a_4^4(a_4-b_4)&=&0,\\
a_3^3b_3(a_3-b_3)+a_4^3b_4(a_4-b_4)&=&0.
\end{eqnarray*}
Thus
\begin{displaymath}
\frac{a_3^4}{a_4^4}=-\frac{a_4-b_4}{a_3-b_3}=\frac{a_3^3b_3}{a_4^3b_4}
\end{displaymath}
and hence $\frac{a_3}{b_3}=\frac{a_4}{b_4}$.
Likewise, if $k$ is $3$ or $4$, then consideration of the $z^0$, $z^5$, and $z^{10}$ coefficients in the equation $u_0^5+\cdots+u_4^5=0$ will lead to the same conclusion.

Let $c=\frac{a_3}{b_3}=\frac{a_4}{b_4}$.
Then $u$ is of the form
\begin{displaymath}
u(z)=[z^5+1:b_1z^k:b_2z^k:b_3(cz^5+1):b_4(cz^5+1)].
\end{displaymath}
Then 
\begin{eqnarray*}
1+b_3^5+b_4^5&=&0,\\
1+c^5(b_3^5+b_4^5)&=&0.
\end{eqnarray*}
Thus $c=1$ and hence $a_3=b_3$, a contradiction.
\end{proof}

\begin{lemma}\label{lemma:case4-rp1_maps}
Let $\mathcal M(L,gL:\rp^1,\rp^1:E_0/2)$ denote the moduli space of maps in the previous lemma.
Then $\mathcal M(L,gL:\rp^1,\rp^1:E_0/2)$ has two components, $C_1$ and $C_2$, where
\begin{eqnarray*}
C_1 &=&\sett{[a(z^5+1):dz^k:-dz^k:b(z^5+1):c(z^5+1)]},\\
C_2 &=&\sett{[a(z^5-1):dz^k:-dz^k:b(z^5-1):c(z^5-1)]}.
\end{eqnarray*}
Moreover, $S^1\times(0,\infty)\cong C_i$ via the diffeomorphism
\begin{eqnarray*}
&((\cos\theta,\sin\theta),d)\mapsto&\\ & [\cos\theta(z^5\pm1):dz^k:-dz^k:\sin\theta(z^5\pm1):-(\cos^5\theta+\sin^5\theta)^{1/5}(z^5\pm1)].
\end{eqnarray*}
Furthermore, $C_1$ and $C_2$ are conjugates of each other.
\end{lemma}
\begin{proof}
It suffices to show that if $a^5+b^5+c^5=0$ then there exists a unique $\theta\in[0,2\pi)$ such that the maps 
\begin{equation}\label{eq:case4-form_1}
[a(z^5+1):dz^k:-dz^k:b(z^5+1):c(z^5+1)]
\end{equation}
and
\begin{equation}\label{eq:case4-form_2}
[\cos\theta(z^5+1):dz^k:-dz^k:\sin\theta(z^5+1):-(\cos^5\theta+\sin^5\theta)^{1/5}(z^5+1)]
\end{equation}
are the same.

$a$ and $b$ cannot both be zero, hence (\ref{eq:case4-form_1}) is the same as
\begin{displaymath}
[\frac{a}{a^2+b^2}(z^5+1):dz^k:-dz^k:\frac{b}{a^2+b^2}(z^5+1):\frac{c}{a^2+b^2}(z^5+1)].
\end{displaymath}
Let $\theta$ be such that $\cos\theta=a/(a^2+b^2)$ and $\sin\theta=b/(a^2+b^2)$.
With this value of $\theta$, (\ref{eq:case4-form_2}) is the same map as (\ref{eq:case4-form_1}).
It is easy to see that this is the unique value of $\theta$ that works.

It remains to prove the final statement.
Suppose $u(z)=[a(z^5+1):dz^k:-dz^k:b(z^5+1):c(z^5+1)]$ is in $C_1$.
Then the conjugate of $u(z)$ is
\begin{displaymath}
\tilde u(z)=g\tau u(\sigma \bar z)=[a(-z^5+1):dz^k:-dz^k:b(-z^5+1):c(-z^5+1)],
\end{displaymath}
which is an element of $C_2$.
\end{proof}

Lemma \ref{lemma:case4-rp1_maps} implies that the moduli space $\mathcal M(L,gL:\rp^1,\rp^1:E_0/2)$ is not compact; the source of non-compactness is the parameter $d\in(0,\infty)$.
By Gromov's compactness theorem, these strips must degenerate into a stable map with multiple domain components as $d\to 0$ and $d\to \infty$.
The next two lemmas describe what these degenerations are.
\begin{lemma}
As $d\to0^+$, the one-parameter family of strips
\begin{displaymath}
u_d(z)=[a(z^5+1):dz^k:-dz^k:b(z^5+1):c(z^5+1)]
\end{displaymath}
converges to a constant strip with the disc $u_0(z)=[5az:\sigma^k:-\sigma^k:5bz:5cz]$ attached (recall that $\sigma=e^{2\pi i/10}$).
The disc bubbles off at the domain point $z=\sigma$.

Likewise, as $d\to0^+$, the one-parameter family of strips
\begin{displaymath}
u_d(z)=[a(z^5-1):dz^k:-dz^k:b(z^5-1):c(z^5-1)]
\end{displaymath}
converges to a constant strip with the disc $u_0(z)=[5az:1:-1:5bz:5cz]$ attached.
The disc bubbles off at the domain point $z=1$.

\end{lemma}
\begin{proof}
To see how the bubbles appear, the domain needs to be rescaled near the point where the bubble grows.
In the first case, let $w=\gamma^2\lambda(z-\sigma)$ where $\lambda>0$.
If $z$ is restricted to a small half-disc in $S_0$ with center $\sigma$ and boundary on $\sigma\cdot \rr_{\geq0}$, then $w$ is a point in a half-disc in the upper half-plane.
As $\lambda\to\infty$, this disc grows to fill out the entire upper half-plane.

In the first case, as a function of $w$ the maps $u_d$ are
\begin{eqnarray*}
u_d(w)&=&[a((\gamma^3\lambda^{-1}w+\sigma)^5+1):d(\gamma^3\lambda^{-1}w+\sigma)^k:-d(\gamma^3\lambda^{-1}w+\sigma)^k:\\
&&b((\gamma^3\lambda^{-1}w+\sigma)^5+1):c((\gamma^3\lambda^{-1}w+\sigma)^5+1)].
\end{eqnarray*}
Let $\lambda=d^{-1}$.
Dividing each homogeneous coordinate by $d$ and letting $d\to0$ gives
\begin{displaymath}
u_0(w)=[5aw:\sigma^k:-\sigma^k:5bw:5cw].
\end{displaymath}

Now consider the second case.
This time the reparameterization to use is $w=\lambda(z-1)$.
Then, as a function of $w$ the maps are
\begin{eqnarray*}
u_d(w)&=&[a((\lambda^{-1}w+1)^5+1):d(\lambda^{-1}w+1)^k:-d(\lambda^{-1}w+1)^k:\\
&&b((\lambda^{-1}w+1)^5+1):c((\lambda^{-1}w+1)^5+1)]. 
\end{eqnarray*}
Dividing each homogeneous coordinate by $d$ and letting $d\to 0^+$ gives
\begin{displaymath}
u_0(w)=[5aw:1:-1:5bw:5cw].
\end{displaymath}
\end{proof}

\begin{lemma}
As $d\to\infty$, the one-parameter family of strips
\begin{displaymath}
u_d(z)=[a(z^5+1):dz^k:-dz^k:b(z^5+1):c(z^5+1)]
\end{displaymath}
converges to the broken trajectory
\begin{displaymath}
u_\infty(z)=[a:z^k:-z^k:b:c]\#[az^{5-k}:1:-1:bz^{5-k}:cz^{5-k}].
\end{displaymath}

Likewise, the one-parameter family of strips
\begin{displaymath}
u_d(z)=[a(z^5-1):dz^k:-dz^k:b(z^5-1):c(z^5-1)]
\end{displaymath}
converges to the broken trajectory
\begin{displaymath}
u_\infty(z)=[-a:z^k:-z^k:-b:-c]\#[az^{5-k}:1:-1:bz^{5-k}:cz^{5-k}].
\end{displaymath}
\end{lemma}
\begin{proof}
Consider first $u_d(z)=[a(z^5+1):dz^k:-dz^k:b(z^5+1):c(z^5+1)]$ near $z=0$.
Let $w=\lambda z$.
Then as a function of $w$, $u_d$ is
\begin{displaymath}
u_d(w)=[a(\lambda^{-5}w^5+1):d\lambda^{-k}w^k:-d\lambda^{-k}w^k:b(\lambda^{-5}w^5+1):c(\lambda^{-5}w^5+1)].
\end{displaymath}
Let $\lambda=d^{1/k}$, then as $d\to\infty$ this map converges to 
\begin{displaymath}
[a:w^k:-w^{-k}:b:c].
\end{displaymath}
To see what happens near $z=\infty$, first change coordinates by letting $z'=1/z$.
Then as a function of $z'$, $u_d$ is
\begin{displaymath}
u_d(z')=[a(1+z'^5):dz'^{5-k}:-dz'^{5-k}:b(1+z'^5):c(1+z'^5)].
\end{displaymath}
Now rescale by letting $w=\lambda z'$, so
\begin{displaymath}
u_d(w)=[a(1+\lambda^{-5}w^5):d\lambda^{k-5}w^{5-k}:-d\lambda^{k-5}w^{5-k}:b(1+\lambda^{-5}w^5):c(1+\lambda^{-5}w^5)].
\end{displaymath}
Let $\lambda=d^{1/(k-5)}$.
Then as $d\to\infty$ this map converges to
\begin{displaymath}
[a:w^{5-k}:-w^{5-k}:b:c].
\end{displaymath}
Since $w=\lambda z'=\lambda/z$, this is defined for $w$ near $\infty$.
To change coordinates back to near $0$, let $z=1/w$, and this is the map
\begin{displaymath}
[az^{5-k}:1:-1:bz^{5-k}:cz^{5-k}].
\end{displaymath}

The case for the other family of strips is similar.
\end{proof}

Now that the strips from $\rp^1$ to $\rp^1$ have been found, the cases will be handled on an individual basis.

\subsection{The case $r=1$}
In this case, strips from $\rp^1$ to $\rp^1$ and from $\rp^1$ to $\rp^0$ need to be considered.
The strips from $\rp^1$ to $\rp^0$ have Maslov index 1.
Conjugation of strips changes orientation by
\begin{displaymath}
(-1)^{\frac{\mu}{2}+\frac{3}{2}(a_q''+b_q''+c_q''-a_p''-b_p''-c_p'')}.
\end{displaymath}
By Lemma \ref{lemma:real_lagrangians-angles}, 
\begin{displaymath}
a_q''+b_q''+c_q''-a_p''-b_p''-c_p''=3(5-r)''-2r''=3\cdot4''-2\cdot1''=3.
\end{displaymath}
The change in sign is therefore $-1$.
The Floer cohomology will not be calculated completely for this case.
However, it is easy to prove
\begin{theorem}\label{thm:case4-main_thm_1}
The Floer cohomology $HF(L,gL;\Lambda_{nov})$ contains a subspace isomorphic to $H^0(\rp^0;\Lambda_{nov})$.
In particular, the rank of $HF(L,gL;\Lambda_{nov})$ is at least $1$.
\end{theorem}
\begin{proof}
By the previous discussion, the effect that $n_{0,0}$ has on $H^0(\rp^0)$ in the spectral sequence is determined by the moduli space of strips $\mathcal M(L,gL:\rp^1,\rp^0)$ that go from $\rp^1$ to $\rp^0$.
Moreover, for dimension reasons, the only part of $n_{0,0}$ that can be non-zero is the part it induces on
\begin{displaymath}
H^0(\rp^1)\to H^0(\rp^0).
\end{displaymath}
(Recall that $H^0(\rp^1)$ contains the fundamental class $[\rp^1]$.)
$\rho_*$ gives a $\z_2$ action on $\mathcal M(L,gL:\rp^1,\rp^0)$, and it reverses sign.
The virtual dimension of this moduli space is $0$, that is, after perturbation it will consist of a finite number of strips.
Therefore, by Proposition \ref{prop:conjugation-main_prop_2}, the part of $n_{0,0}$ that goes from $H^0(\rp^1)$ to $H^0(\rp^0)$ is $0$.
Hence $H^0(\rp^0)$ will survive in the spectral sequence.

\end{proof}

To calculate the Floer cohomology completely would require determining how the strips from $\rp^1$ to $\rp^1$ affect $H^*(\rp^1)$ in the spectral sequence.
Hence the first step would involve the strips from Section \ref{section:case4-rp1_strips}.
These strips are not regular, so the theory of Kuranishi structures and perturbations would have to enter, greatly complicating the picture.
This, therefore, is left to further research.
However, given that the lowest energy strips from $\rp^1$ to $\rp^1$ have explicit descriptions, it may be possible to explicitly carry out the perturbations.
Such an explicit description would be intersting, as I believe such a thing has never been done before.

\subsection{The case $r=2$}
In this case, strips that go from $\rp^1$ to $\rp^1$ and $\rp^0$ to $\rp^1$ need to be considered.
The lowest energy strips from $\rp^0$ to $\rp^1$ are of the form
\begin{displaymath}
u(z)=[a_0z:1:-1:a_3z:a_4z]
\end{displaymath}
where $a_0^5+a_3^5+a_4^5=0$.
These strips are regular:
In the sequence $0\to u^*\mathcal O\to u^*\mathcal O(1)^{\oplus 5}\to u^*T_{\cp^4}\to 0$, the second map is (assuming $a_0,a_3\neq0$)
\begin{displaymath}
\left[
\begin{array}{ccccc}
a_3 & 0 & 0 & -a_0 & 0\\
a_4 & 0 & 0 & 0 & -a_0\\
0 & 1 & 1 & 0 & 0\\
Z_1 & -a_0Z_0 & 0 & 0 & 0
\end{array}
\right].
\end{displaymath}
Therefore $u^*T_{\cp^4}\cong \mathcal O(1) \oplus \mathcal O(1)\oplus\mathcal O(1)\oplus\mathcal O(2)$.
In the sequence $0\to u^*T_X\to u^*T_{\cp^4}\to u^*N\to 0$ the second map is
\begin{displaymath}
\left[
\begin{array}{cccc}
-\frac{5a_3^4}{a_0}Z_0^4 & -\frac{5a_4^4}{a_0}Z_0^4 & 5Z_1^4 & 0
\end{array}
\right]
\end{displaymath}
and the first map is
\begin{displaymath}
\left[
\begin{array}{ccc}
0 & \frac{a_4^4}{a_0} & Z_1^4\\
0 & -\frac{a_3^4}{a_0} & 0\\
0 & 0 & \frac{a_3^4}{a_0}Z_1^4\\
1 & 0 & 0
\end{array}
\right],
\end{displaymath}
so $u^*T_X\cong \mathcal O(2)\oplus\mathcal O(1)\oplus\mathcal O(-3)$.
Therefore
\begin{displaymath}
 u^*T_X^*\cong \mathcal O(-2)\oplus\mathcal O(-1)\oplus\mathcal O(3).
\end{displaymath}
The holomorphic sections that vanish at $0$ and $\infty$ are $(0,0,Z_0Z_1^2)$ and $(0,0,Z_0^2Z_1)$.
$(0,0,Z_0^iZ_1^{3-i})$ lifts to $(0,0,\frac{a_0Z_0^iZ_1^{3-i}}{a_3^4Z_1^4},0)$, which pushes forward to 
$$
(0,\frac{a_0Z_0^{i}Z_1^{3-i}}{a_3^4Z_1^4},\frac{a_0Z_0^{i}Z_1^{3-i}}{a_3^4Z_1^4},0,0).
$$
$\lambda_u=\gamma^4$ so the $\z_5$ action is $(0,0,Z_0^iZ_1^{3-i})\mapsto\gamma^{i+1}(0,0,Z_0^iZ_1^{3-i})$.
Since $i=1$ or $2$ the fixed subspace is empty and the cokernel is $0$.

\begin{theorem}\label{thm:case4-main_thm_2}
\begin{displaymath}
HF^*(L, gL;\Lambda_{nov})\cong H^1(\rp^1;\Lambda_{nov}).
\end{displaymath}
\end{theorem}
\begin{proof}
The $E_2$ term of the spectral sequence has three non-trivial terms: $H^1(\rp^1)$, $H^0(\rp^1)$, and $H^0(\rp^0).$
The maps that can possibly be non-zero are $$H^0(\rp^1)\to H^1(\rp^1)$$ and $$H^0(\rp^0)\to H^0(\rp^1).$$
Actually, the map $H^0(\rp^1)\to H^1(\rp^1)$ is $0$ in $E_2$ because it is defined using strips that go from $\rp^1$ to $\rp^1$, and these strips do not have lowest energy.
Therefore the only map to consider is the map $H^0(\rp^0)\to H^0(\rp^1)$; this map is defined (in $E_2$) using the strips described above.
These strips are regular; let $\mathcal M$ denote the moduli space of these strips, and note that $\mathcal M\cong S^1$.
The map $ev_{+\infty}:\mathcal M\to\rp^1$ is the double cover map $\mathcal M\cong S^1\to \rp^1$.
Therefore if $[p]\in\rp^0$ is the fundamental class, $$ev_{+\infty*}([p]\times_{ev_{-\infty}}\mathcal M)=\pm2[\rp^1].$$
Passing to the $E_3$ term by taking cohomology, it follows that $E_3$ contains a single non-zero group, namely $H^1(\rp^1)$.
The theorem then follows from the fact that there are no non-zero maps $H^1(\rp^1)\to H^1(\rp^1)$, for dimension reasons.
\end{proof}

\subsection{The case $r=3$}
In this case, strips that go from $\rp^1$ to $\rp^1$ and from $\rp^1$ to $\rp^0$ need to be considered.
The strips from $\rp^1$ to $\rp^0$ have Maslov index 2.
The dimension of the moduli space of these strips is $1$.
The lowest energy strips are easily seen to be of the form 
\begin{displaymath}
u(z)=[a_0:z:-z:a_3:a_4]
\end{displaymath}
where $a_0^5+a_3^5+a_4^5=0$.
If $\mathcal M$ denotes the moduli space of these strips, then $\mathcal M\cong S^1$.
Moreover, conjugation is rotation of $S^1$ by $\pi$, because
\begin{eqnarray*}
\tilde u(z)&=&g\overline{u(\sigma \bar z)}=g[a_0:\sigma^9z:-\sigma^9z:a_3:a_4]=[a_0:-z:z:a_3:a_4]\\&=&[-a_0:z:-z:-a_3:-a_4].
\end{eqnarray*}
These strips are regular:
In the sequence $0\to u^*\mathcal O\to u^*\mathcal O(1)^{\oplus 5}\to u^*T_{\cp^4}\to 0$, the second map is
\begin{displaymath}
\left[
\begin{array}{ccccc}
\frac{a_3}{a_0} & 0 & 0 & -1 & 0\\
\frac{a_4}{a_0} & 0 & 0 & 0 & -1\\
0 & 1 & 1 & 0 & 0\\
\frac{1}{a_0}Z_0 & -Z_1 & 0 & 0 & 0\\
\end{array}
\right],
\end{displaymath}
so $u^*T_{\cp^4}\cong \mathcal O(1) \oplus \mathcal O(1)\oplus\mathcal O(1)\oplus\mathcal O(2)$.
In the sequence $0\to u^*T_X\to u^*T_{\cp^4}\to u^*N\to 0$ the second map is
\begin{displaymath}
\left[
\begin{array}{cccc}
-5a_3^4Z_1^4 & -5a_4^4Z_1^4 & 5Z_0^4 & 0
\end{array}
\right]
\end{displaymath}
and the first map is
\begin{displaymath}
\left[
\begin{array}{ccc}
0 & -a_4^4 & Z_0^4\\
0 & a_3^4 & 0\\
0 & 0 & a_3^4Z_1^4\\
1 & 0 & 0
\end{array}
\right],
\end{displaymath}
so $u^*T_X\cong \mathcal O(2)\oplus\mathcal O(1)\oplus\mathcal O(-3)$.
Therefore
\begin{displaymath}
u^*T_X^*\cong \mathcal O(-2)\oplus\mathcal O(-1)\oplus\mathcal O(3).
\end{displaymath}
The holomorphic sections that vanish at $0$ and $\infty$ are $(0,0,Z_0Z_1^2)$ and $(0,0,Z_0^2Z_1)$.
$(0,0,Z_0^iZ_1^{3-i})$ lifts to $(0,0,\frac{Z_0^iZ_1^{3-i}}{a_3^4Z_1^4},0)$ which pushes forward to 
$
(0,\frac{Z_0^{i}Z_1^{3-i}}{a_3^4Z_1^4},\frac{Z_0^{i}Z_1^{3-i}}{a_3^4Z_1^4},0,0).
$
$\lambda_u=1$ so the $\z_5$ action is $(0,0,Z_0^iZ_1^{3-i})\mapsto\gamma^{i+1}(0,0,Z_0^iZ_1^{3-i})$.
Since $i=1$ or $2$ the fixed subspace is empty and the cokernel is $0$.

\begin{theorem}\label{thm:case4-main_thm_3}
\begin{displaymath}
HF^{*-2}(L,gL;\Lambda_{nov})\cong H^0(\rp^1;\Lambda_{nov}).
\end{displaymath}
That is, $HF$ is $0$ in all degrees except degree $-2$, where it has rank $1$.
\end{theorem}
\begin{proof}
The $E_2$ term of the spectral sequence has three non-trivial terms: $H^1(\rp^1)$, $H^0(\rp^1)$, and $H^0(\rp^0)$.
The maps that can possibly be non-zero are $H^1(\rp^1)\to H^0(\rp^1)$ and $H^1(\rp^1)\to H^0(\rp^0)$.
Actually, the map $H^1(\rp^1)\to H^0(\rp^1)$ is $0$ in $E_2$ because it is defined using strips that go from $\rp^1$ to $\rp^1$, and these strips do not have lowest energy.
Therefore the only map to consider is the map $H^1(\rp^1)\to H^0(\rp^0)$.
This map is defined using the strips described above.
The map $ev_{-\infty}:\mathcal M\to\rp^1$ is the double cover map $\mathcal M\cong S^1\to \rp^1$.
Therefore if $p\in\rp^1$ is any point then $ev_{-\infty}^{-1}(p)$ consists of a conjugate pair of strips.
Conjugation preserves orientation because
\begin{displaymath}
(-1)^{\frac{1}{2}\mu(u)+\frac{3}{2}(a_q''+b_q''+c_q''-a_p''-b_p''-c_p'')}=(-1)^{1+\frac{3}{2}(-2)}=+1.
\end{displaymath}
(This can also be seen because conjugation is rotation by $\pi$.)
The strips are regular, so $n_{0,0}([p])=\pm 2[\rp^0]$.
Therefore, passing to the $E_3$ term by taking homology, it follows that $E_3$ contains a single non-zero group, namely $H^0(\rp^1)$.
The theorem then follows from the fact that in the spectral sequence there are no non-zero maps $H^0(\rp^1)\to H^0(\rp^1)$, for dimension reasons.
\end{proof}

\subsection{The case $r=4$}
In this case, strips from $\rp^1$ to $\rp^1$ and $\rp^0$ to $\rp^1$ need to be considered.
The strips from $\rp^0$ to $\rp^1$ have Maslov index $0$.
The $E_2$ page of the spectral sequence has three non-zero terms, namely $H^1(\rp^1)$, $H^0(\rp^1)$, and $H^0(\rp^0)$.
Moreover, the possibly non-zero maps are
\begin{displaymath}
H^1(\rp^1)\to H^0(\rp^1)
\end{displaymath}
and
\begin{displaymath}
H^0(\rp^0)\to H^1(\rp^1).
\end{displaymath}
By Proposition \ref{prop:conjugation-main_prop}, the conjugation map $\rho_*:\mathcal M(L,gL:\rp^0,\rp^1)\to\mathcal M(L,gL:\rp^0,\rp^1)$ changes orientation by
\begin{displaymath}
(-1)^{\frac{0}{2}+\frac{3}{2}(2r''-3(5-r)'')}=-1.
\end{displaymath}
Since the virtual dimension of $\mathcal M(L,gL:\rp^0,\rp^1)$ is $0$, Proposition \ref{prop:conjugation-main_prop_2} implies that the map $H^0(\rp^0)\to H^1(\rp^1)$ is zero (on all pages of the spectral sequence).
This proves
\begin{theorem}\label{thm:case4-main_thm_4}
$HF(L,gL;\Lambda_{nov})$ has a subspace isomorphic to $H^0(\rp^0;\Lambda_{nov})$.
In particular, the rank of $HF(L,gL;\Lambda_{nov})$ is at least 1.
\end{theorem}

To complete the calculation of the Floer cohomology, the strips from $\rp^1$ to $\rp^1$ need to be considered.
As in the case for $r=1$, it turns out that these strips are not regular, and hence Kuranishi structures need to be used.
This is left to further research.

\section{Case (5): $L\cap gL\cong \rp^0$}
In this case, $L\cap gL$ consists of a single point.

  \begin{theorem}\label{thm:case5-main_thm}
  $$HF^*(L,gL;\Lambda_{nov})\cong H^{*+r''+s''+t''}(\rp^0;\Lambda_{nov}).$$
That is, the Floer cohomology is isomorphic to the singular cohomology of $\rp^0\cong L\cap gL$.
  \end{theorem}
\begin{proof}
First, note that $L\cap gL$ consists of a single component, hence the Maslov index of any strip is $0$.
Second, note that $L\cap gL$ does not have non-zero singular cohomology in any two consecutive degrees.
With these two facts in hand, the proof is exactly the same as the proof of Theorem \ref{thm:case1-main_thm}.
\end{proof}

\section{Case (6): $L\cap gL=\rp^0\amalg\rp^0$}
In this case $g=g_i^rg_j^rg_k^t$.
Let $R_p$ be the component where the $i,j,k$ coordinates are $0$, and let $R_q$ be the other component.
Without loss of generality, assume $i=1,j=2,k=3$, so
\begin{eqnarray*}
R_p&=&\sett{[1:0:0:0:-1]},\\
R_q&=&\sett{[0:1:-1:0:0]}.
\end{eqnarray*}
This case breaks down into several subcases:
First, there are some cases where the Floer cohomology can be calculated purely from degree considerations.
Second, there are some cases where the Floer differential can be shown to be zero because the conjugation involution reverses sign.
Finally, there are the remaining cases, where the moduli spaces need to be calculated and analyzed.
Some of these final cases can be done, and some cannot.

The gradings do not conform to a nice formula in this case, so they will be ignored.
(They can, however, still be computed in the same way as before.)
\subsection{Degree considerations}

\begin{lemma}\label{lemma:case6-maslov_index}
If $u$ is a strip from $R_p$ to $R_q$ then the Maslov index of $u$ is
\begin{equation}
\begin{array}{ccccc}
& t=1 & t=2 & t=3 & t=4 \\
r=1 & * & 0 & 1 & 0 \\
r=2 & -1 & * & -1 & -1 \\
r=3 & 1 & 1 & * & 1\\
r=4 & 0 & -1 & 0 & *
\end{array}.
\end{equation}
If $u$ goes from $R_q$ to $R_p$ then the Maslov index is $(-1)$ times the entry listed in the chart.
\end{lemma}
\begin{proof}
This follows from Lemmas \ref{lemma:gradings-formula} and \ref{lemma:real_lagrangians-angles}.
\end{proof}

Arguing as before, this lemma allows the calculation of some Floer cohomologies from degree considerations:
\begin{theorem}\label{thm:case6-main_thm_1}
If $r=1,t=2$ or $r=1,t=4$ or $r=4,t=1$ or $r=4,t=3$ then
$$
HF(L,gL;\Lambda_{nov})\cong H(\rp^0\amalg\rp^0;\Lambda_{nov}).
$$
\end{theorem}
\begin{proof}
The Floer coboundary operator counts Maslov index 1 strips, and in these cases there are none.
Therefore the spectral sequence degenerates at $E_2$ and the Floer cohomology is isomorphic to the singular cohomology.
  \end{proof}

\subsection{Involution considerations}

The cases where the conjugation involution reverses sign are also easy to calculate.
\begin{theorem}\label{thm:case6-main_thm_2}
If $r=2,t=3$ or $r=3,t=1$ or $r=2,t=4$ or $r=3,t=2$ then
$$
HF(L,gL;\Lambda_{nov})\cong H(\rp^0\amalg\rp^0;\Lambda_{nov}).
$$
\end{theorem}
\begin{proof}
The involution $\rho_*:\mathcal M(L,gL)\to\mathcal M(L,gL)$ reverses orientation in these cases by Proposition \ref{prop:conjugation-main_prop} and Lemmas \ref{lemma:real_lagrangians-angles} and \ref{lemma:case6-maslov_index}.
The theorem then follows from Proposition \ref{prop:conjugation-main_prop_2}.
\end{proof}

\subsection{The remaining cases}
The remaining possibilities to investigate are $r=1,t=3$; $r=4,t=2$; $r=2,t=1$ and $r=3,t=4$.
Actually, by symmetry it can be reduced to the cases $r=4,t=2$ and $r=2,t=1$.
For example, if $g=g_1^3g_2^3g_3^4$, then $g$ is equivalent to $g_0^2g_4^2g_3^1$, which is the case $r=2,t=1$.

\subsection{The case $r=2,t=1$}
Without loss of generalizty assume that $g=g_1^2g_2^2g_3$.
%Let $R_p$ be the component
%\begin{displaymath}
%R_p=\sett{[1:0:0:0:-1]}
%\end{displaymath}
%and let $R_q$ be the component 
%\begin{displaymath}
%R_q=\sett{[0:1:-1:0:0]}.
%\end{displaymath}
By Lemma \ref{lemma:case6-maslov_index}, the Maslov index 1 strips are the strips that go from $R_q$ to $R_p$.
If $u=[u_0:u_1:u_2:u_3:u_4]$ is such a s strip then $u_1(z)\neq0$, so by Lemma \ref{lemma:strips-polynomials} the $u_i$'s are of the form
\begin{eqnarray*}
u_0&=&zp_0(z^5),\\
u_1&=&p_1(z^5),\\
u_2&=&p_2(z^5),\\
u_3&=&z^3p_3(z^5),\\
u_4&=&zp_4(z^5).
\end{eqnarray*}
Moreover, since the strip starts in $R_q$ and ends in $R_p$ it must satisfy
\begin{itemize}
\item the constant terms of $u_1$ and $u_2$ cannot both be $0$, and
\item $\max\sett{\deg u_0,\deg u_3,\deg u_4}>\max\sett{\deg u_1,\deg u_2}$.
\end{itemize}
Therfore the lowest energy solutions must be of the form
\begin{displaymath}
u(z)=[a_0z:b_1:b_2:a_3z^3:a_4z].
\end{displaymath}
The only solutions of this form that also satisfy the quintic equation are easily seen to be (up to $\rr$ translation)
\begin{displaymath}
z\mapsto [\pm z:1:-1:0:\mp z].
\end{displaymath}
These maps are conjugates of each other.
Indeed, if $u(z)=[z:1:-1:0:-z]$ then
\begin{displaymath}
\tilde u(z)=g\tau u(\sigma\bar z)=[\sigma^{-1}z:\gamma^2:-\gamma^2:0:-\sigma^{-1}z]=[-z:1:-1:0:z].
\end{displaymath}

The next step is to calculate the cokernels of these maps.
By symmetry, it suffices to calculate the cokernel of $u$.
Consider the sequence $0\to u^*\mathcal O(1)^{\oplus 5}\to u^*T_{\cp^4}\to 0$.
The first map is given by $u$ and hence is
\begin{displaymath}
\left[
\begin{array}{c}
Z_0\\Z_1\\-Z_1\\0\\-Z_0
\end{array}
\right].
\end{displaymath}
The second map is given by the relations among these polynomials and hence is
\begin{displaymath}
\left[
\begin{array}{ccccc}
1 & 0 & 0 & 0 & 1\\
0 & 1 & 1 & 0 & 0\\
0 & 0 & 0 & 1 & 0\\
Z_1 & -Z_0 & 0 & 0 & 0
\end{array}
\right].
\end{displaymath}
Therefore $u^*T_{\cp^4}$ splits as
\begin{displaymath}
u^*T_{\cp^4}\cong \mathcal O(1)\oplus\mathcal O(1)\oplus\mathcal O(1)\oplus\mathcal O(2).
\end{displaymath}

Now consider the sequence $0\to u^*T_X\to u^*T_{\cp^4}\to N\to 0$.
If the last map is denoted $[\begin{array}{cccc}a_1 & a_2 & a_3 & a_4\end{array}]$, then the $a_i$'s satisfy
\begin{eqnarray*}
a_1+Z_1a_4 &=& 5Z_0^4,\\
a_2-Z_0a_4 &=& 5Z_1^4,\\
a_2 &=& 5Z_1^4,\\
a_3 &=& 0,\\
a_1 &=& 5Z_0^4.
\end{eqnarray*}
Therefore
\begin{displaymath}
[\begin{array}{cccc}a_1 & a_2 & a_3 & a_4\end{array}]=[\begin{array}{cccc}5Z_0^4 & 5Z_1^4 & 0 & 0\end{array}].
\end{displaymath}
The first map is given by the relations among these polynomials and hence is
\begin{displaymath}
\left[
\begin{array}{ccc}
0 & 0 & Z_1^4\\
0 & 0 & -Z_0^4\\
1 & 0 & 0\\
0 & 1 & 0
\end{array}
\right].
\end{displaymath}
Comparing degrees shows that $u^*T_X$ splits as
\begin{displaymath}
u^*T_X\cong \mathcal O(1)\oplus\mathcal O(1)\oplus\mathcal O(-3).
\end{displaymath}

The next step is to determine the $\z_5$ action on holomorphic sections of $u^*T_X^*$.
A basis for sections that vanish at $0$ and $\infty$ is $(0,0,Z_0Z_1^2),(0,0,Z_0^2Z_1)$.
The section $(0,0,Z_0^iZ_1^{3-i})$ pulls back to the section 
\begin{displaymath}
(\frac{Z_o^iZ_1^{3-i}}{Z_1^4},0,0,0)
\end{displaymath}
of $u^*T_{\cp^4}$, which then pushes forward to the section
\begin{displaymath}
(\frac{Z_0^iZ_1^{3-i}}{Z_1^4},0,0,0,\frac{Z_o^iZ_1^{3-i}}{Z_1^4})
\end{displaymath}
of $u^*\mathcal O(1)^{\oplus 5}$.
If $\lambda_uu(\gamma z)=g^2u(z)$, then $\lambda_u=\gamma^4$.
Thus, by Lemma \ref{lemma:case3-action}, the $\z_5$ action is
\begin{displaymath}
(\frac{Z_0^iZ_1^{3-i}}{Z_1^4},0,0,0,\frac{Z_o^iZ_1^{3-i}}{Z_1^4})\mapsto(\gamma^{4-i}\frac{Z_o^iZ_1^{3-i}}{Z_1^4},0,0,0,\gamma^{4-i}\frac{Z_o^iZ_1^{3-i}}{Z_1^4}).
\end{displaymath}
That is, as sections of $u^*T_X^*$, the action is
\begin{displaymath}
(0,0,Z_0^iZ_1^{3-i})\mapsto \gamma^{4-i}(0,0,Z_0^iZ_1^{3-i}).
\end{displaymath}
Hence the $\z_5$ fixed point subspace is $0$ (because $i$ is 1 or 2).

In summary:
\begin{lemma}
The lowest energy strips are
\begin{displaymath}
z\mapsto[\pm z:1:-1:0:\mp z].
\end{displaymath}
These two strips are conjugates of each other.
Moreover, the cokernels of the linearzed $\dbar$ operators are $0$.
\end{lemma}

\begin{theorem}\label{thm:case6-main_thm_3}
$$HF(L,gL;\Lambda_{nov})=0.$$
\end{theorem}
\begin{proof}
By the previous discussion, the two strips in the previous lemma determine the differential in the $E_2$ term of the spectral sequence.
The two strips are conjugates of each other, and they contribute the same sign by Proposition \ref{prop:conjugation-main_prop}, so the differential is $\pm 2$.
Either way, the cohomology is $0$, and hence the $E_3$ term of the spectral sequence is $0$.
\end{proof}

\subsection{The case $r=4,t=2$}
%NOTE: THIS SECTION WILL BE REMOVED, I HAD DONE IT BEFORE I FOUND A SIGN ERROR, I NO LONGER NEED IT BY THEOREM 18.3
The first part of the calculation follows in the same manner as before.
However, it turns out that the strips are not regular so Kuranishi spaces need to be used.
Therefore, the consideration of this case is left to further research.

%% Start the appendices:
\appendix       % Chapters, sections are now appendix style
\part{Appendix }
%\chapter{Appendix}
\section{Kuranishi structures}\label{section:appendix}
This appendix contains a very brief review of Kuranishi spaces.
It is included mainly to fix notation; readers are referred to the appendix of \cite{fooo} for a more serious discussion.
\subsection{Definitions}
Let $\mathcal M$ be a compact Hausdorff topological space.
A Kuranishi chart on $\mathcal M$ consists of the data $(V,E,\Gamma,\psi,s)$, where
\begin{enumerate}
\item[(kc1)] $V$ is a smooth manifold which may have boundaries or corners,
\item[(kc2)] $E$ is a real vector space,
\item[(kc3)] $\Gamma$ is a finite group that acts smoothly and effectively on $V$, and on $E$ linearly,
\item[(kc4)] $s$ is a $\Gamma$-equivariant section of the bundle $E\times V\to V$, and
\item[(kc5)] $\psi:s^{-1}(0)/\Gamma\to\mathcal M$ is a homeomorphism onto its image.
\end{enumerate}
$s$ is called the Kuranishi map and $E$ is called the obstruction bundle.

A Kuranishi structure on $\mathcal M$ consists of a collection $\sett{(V_p,E_p,\Gamma_p,\psi_p,s_p)}_{p\in\mathcal P}$ of Kuranishi charts on $\mathcal M$ such that
\begin{enumerate}
\item[(ks1)] $\cup_{p\in\mathcal P}\psi_p(s_p^{-1})(0)/\Gamma_p)=\mathcal M$,
\item[(ks2)] the number $\dim V_p-\dim E_p$ is the same for all $p$, and
\item[(ks3)] the charts staisfy certain compatibility conditions (called coordinate changes).
\end{enumerate}
The common number $\dim V_p-\dim E_p$ is called the virtual dimension of $\mathcal M$.

A good coordinate system on $\mathcal M$ is a Kuranishi structure $\sett{(V_p,E_p,\Gamma_p,\psi_p,s_p)}_{p\in\mathcal P}$ with $\mathcal P$ finite and a partial order $<$ on $\mathcal P$ such that the following condition holds: if $q<p$ and $\Image\psi_p\cap \Image \psi_q\neq\emptyset$, then there exists 
\begin{enumerate}
\item[(gc1)] a $\Gamma_q$ invariant open subset $V_{pq}$ in $V_q$ such that $$\psi_q^{-1}(\Image \psi_p)\subset V_{pq}/\Gamma_q,$$
\item[(gc2)] an injective group homomorphism $h_{pq}:\Gamma_q\to\Gamma_p$,
\item[(gc3)] an $h_{pq}$-equivariant smooth embedding $\phi_{pq}:V_{pq}\to V_p$ such that the induced map $\underline \phi_{pq}:V_{pq}/\Gamma_q\to V_p/\Gamma_p$ is injective,
\item[(gc4)] an $h_{pq}$-equivariant embedding $\hat\phi_{pq}:E_q\times V_{pq}\to E_p\times V_p$ of vector bundles which covers $\phi_{pq}$ and satisfies
$$\hat\phi_{pq}\circ s_q=s_p\circ\phi_{pq},\,\psi_q=\psi_p\circ\underline\phi_{pq}.$$
%where $\underline\phi_{pq}:V_{pq}/\Gamma_q\to V_p/\Gamma_p$ is the map induced by $\phi_{pq}$.
\end{enumerate}
Moreover, if $r<q<p$ and $$\Image \psi_p\cap\Image\psi_q\cap\Image\psi_r\neq\emptyset,$$ then there exists
\begin{enumerate}
\item[(gc5)] $\gamma_{pqr}\in\Gamma_p$ such that
\begin{displaymath}
h_{pq}\circ h_{qr}=\gamma_{pqr}\cdot h_{pr}\cdot \gamma_{pqr}^{-1},\,\phi_{pq}\circ\phi_{qr}=\gamma_{pqr}\cdot\phi_{qr},\,\hat\phi_{pq}\circ\hat\phi_{qr}=\gamma_{pqr}\cdot\hat\phi_{pr}.
\end{displaymath}
\end{enumerate}
These conditions imply, in particular, that a good coordinate system satisfies the compatibility conditions (ks3).

Now suppose $\mathcal M$ is a Kuranishi space with a fixed good coordinate system.
Identify a neighborhood of $\phi_{pq}(V_{pq})$ in $V_p$ with a neighborhood of the zero section of the normal bundle $N_{V_{pq}}V_p\to V_{pq}$.
Then the differential of the Kuranishi map $s_p$ along the fiber direction defines an $h_{pq}$-equivariant bundle homomorphism
\begin{displaymath}
d_{\textrm{fiber}}s_p:N_{V_{pq}}V_p\to E_p\times V_{pq}.
\end{displaymath}
\begin{definition}
The Kuranishi space $\mathcal M$ has a tangent bundle if $d_{\textrm{fiber}}s_p$ induces a bundle isomorphism 
\begin{equation}\label{eq:appendix-tangent_bundle}
N_{V_{pq}}V_p\cong \frac{E_p\times V_{pq}}{\hat\phi_{pq}(E_q\times V_{pq})}
\end{equation}
as $\Gamma_q$-equivariant bundles on $V_{pq}$.
\end{definition}
\begin{definition}
If $\mathcal M$ is a Kuranishi space with tangent bundle, then $\mathcal M$ is orientable if, for every $p\in\mathcal P$, there exists a trivialization of 
\begin{displaymath}
\Det(E_p^*)\otimes \Det(TV_p)
\end{displaymath}
which is compatible with the isomorphism (\ref{eq:appendix-tangent_bundle}) and whose homotopy class is preserved by the $\Gamma_p$ action.
\end{definition}

\subsection{Involutions of Kuranishi spaces}\label{section:appendix-involution}
Let $\mathcal M$ be a Kuranishi space.
An automorphism of $\mathcal M$ is a homeomorphism $\phi:\mathcal M\to\mathcal M$ such that the following condition holds:
Let $p\in\mathcal M$, and let $p'=\phi(p)$. 
Then, for the Kuranishi neighborhoods $(V_p,E_p,\Gamma_p,\psi_p,s_p)$ and $(V_{p'},E_{p'},\psi_{p'},s_{p'})$ of $p$ and $p'$, there exists $\rho_p:\Gamma_p\to\Gamma_{p'}$, $\phi_p:V_p\to V_{p'}$, and $\hat\phi_p:E_p\to E_{p'}$ such that
\begin{enumerate}
\item $\rho_p$ is an isomorphism,
\item $\phi_p$ is a $\rho_p$-equivariant diffeomorphism,
\item $\hat\phi_p$ is a $\rho_p$-equivariant bundle isomorphism which covers $\phi_p$,
\item $s_{p'}\circ\phi_p=\hat\phi_p\circ s_p$,
\item $\psi_{p'}\circ\underline\phi_p=\phi\circ\psi_p$, where $\underline\phi_p:s_p^{-1}(0)/\Gamma_p\to s_{p'}^{-1}(0)/\Gamma_{p'}$ is the homeomorphism induced by $\phi_p$, and
\item additional compatibility conditions with the coordinate changes are required.
\end{enumerate}

\begin{definition}
An involution of a Kuranishi space $\mathcal M$ is a homeomorphism $\phi:\mathcal M\to\mathcal M$ such that $\phi^2=id$ and $\phi$ lifts to an automorphism of $\mathcal M$.
\end{definition}

An involution of $\mathcal M$ is the same thing as a $\z_2$ action on $\mathcal M$.
\begin{lemma}[\cite{foooasi} Lemma 7.6]
If a finite group $G$ acts on $\mathcal M$, then $\mathcal M/G$ has a Kuranishi structure.

If $\mathcal M$ has a tangent bundle and the action preserves it, then the quotient space has a tangent bundle.
If $\mathcal M$ is oriented and the action preserves the orientation, then the quotient has an orientation.
\end{lemma}

\subsection{Multisections and the virtual fundamental class}
Let $(V,E,\Gamma,\psi,s)$ be a Kuranishi chart.
Let $S^n(E)$ denote the quotient of $E^n$ by the $S_n$ (symmetric group on $n$-letters) action that permutes coordinates.
That is, two tuples $(x_1,\ldots,x_n)$ and $(y_1,\ldots,y_n)$ of $E^n$ are equal in $S^n(E)$ if and only if there exists a permutation $\sigma\in S_n$ such that
$$x_1=y_{\sigma(1)},\ldots,x_n=y_{\sigma(n)}.$$
\begin{definition}
An $n$-multisection $s$ of $E\times V\to V$ is a $\Gamma-equivariant$ map $s:V\to S^n(E)$.
$s$ is liftable if there exists a map $\tilde s=(\tilde s_1,\ldots,\tilde s_n):V\to E^n$ whose image in $S^n(E)$ is $s$.
(The lift $\tilde s$ is not required to be $\Gamma$-equivariant.)
The $\tilde s_i$'s are called the branches of the mutlisection.
\end{definition} 

If $s$ is an $n$-multisection and $m$ is a positive integer, then an $nm$-multisection is defined by composing $s$ with the function $S^n(E)\to S^{nm}(E)$ that sends
\begin{displaymath}
[(x_1,\ldots,x_n)]\in S^n(E)\mapsto [(x_1,\ldots,x_1,\ldots,x_n,\ldots,x_n)]\in S^{nm}(E),
\end{displaymath}
where each $x_i$ is repeated $m$ times.
If $s$ is an $n$-multisection and $s'$ is an $m$-multisection, then $s$ and $s'$ are equivalent if their images as $nm$-multisections are the same.

If $s$ is liftable, then $s$ is transverse to the zero section if each of its branches is transverse to the zero section.
A family of multisections $s_\epsilon$ converges to $s$ if there exists an $n$ and a family of $n$-multisections $s_\epsilon^n$ that converges to an $n$-multisection $s^n$, and $s_\epsilon^n$ are equivalent to $s_\epsilon$ and $s^n$ is equivalent to $s$.

Now let $\mathcal M$ be a Kuranishi space with a good Kuranishi coordinate system $$\sett{(V_p,E_p,\Gamma_p,\psi_p,s_p)}_{p\in\mathcal P}.$$
A multisection $\mathfrak s=(s_p')$ on $\mathcal M$ is a collection of multisections, where $s_p'$ is defined on $(V_p,E_p,\Gamma_p,\psi_p,s_p)$, that satisfy a certain compatibility condition (they have to be equivalent on overlaps of Kuranishi charts).
It can be shown that the Kuranishi map $\sett{s_p}$ is a multisection on $\mathcal M$ (with a single branch).
%Suppose $q<p$ and $\Image \psi_p\cap\Image \psi_q\neq\emptyset$.
%Consider the embedding $\phi_{pq}:V_{pq}\to V_p$.
%Identify a tubular neighborhood of the $\phi_{pq}(V_{pq})$ with a neighborhood of the zero seciton of $N_{V_{pq}}V_p$.
%For each $x\in V_{pq}$, fix a splitting
%\begin{displaymath}
%E_p\cong \hat\phi_{pq,x}(E_q)\oplus \frac{E_p}{\hat\phi_{pq,x}(E_q)}.
%\end{displaymath}
%For each $y\in N_{V_{pq}}V_p$, define $1(y)$ to be the image of $y$ in $E_p/\hat\phi_{pq,x}(E_q)$ under the isomorphism (\ref{eq:appendix-tangent_bundle}).
%Then, define the multisection $s_q'\oplus 1$ on $N_{V_{pq}}V_p$ as follows:

The main application of multisections is the construction of the virtual fundamental class.
First, some preliminary results are needed.
\begin{theorem}[\cite{fooo} Theorem A1.23]\label{thm:appendix-multisections}
Let $\mathcal M$ be a Kuranishi space with tangent bundle. 
Then there exists a system of multisections $\mathfrak s_\epsilon=\sett{s_{p,\epsilon}'}$ that converges to $\sett{s_p}$ (the Kuranishi map) and such that $s_{\epsilon,p}'$ is transverse to the zero section for all $\epsilon>0$.
\end{theorem}

Let $Y$ be a topological space.
A collection of continuous maps $f_p:V_p\to Y$ is called a strongly continuous map if $f_p\circ\phi_{pq}=f_q$ on $V_{pq}$.
A strongly continuous map induces a continuous map $f:\mathcal M\to Y$.

Now suppose $\sett{f_p}:\mathcal M\to Y$ is a strongly continuous map.
Then a chain in $Y$, called the virtual fundamental chain, can be constructed from the maps $\sett{f_p}$ and the multisection $\mathfrak s_\epsilon$ in Theorem \ref{thm:appendix-multisections}.
Namely, for each multisection $s_{p,\epsilon}'$, let $\tilde s_{p,\epsilon,i}'$ be the branches of $s_{p,\epsilon}$.
Let
\begin{eqnarray*}
\tilde s_{p,\epsilon}^{'-1}(0)&=&\set{y\in V_p}{\textrm{there exists a branch of $s_{p,\epsilon}'$ that vanishes at $y$}},\\
s_{p,\epsilon}^{'-1}(0)&=&\tilde s_{p,\epsilon}^{'-1}(0)/\Gamma_p.
\end{eqnarray*}
Let 
\begin{displaymath}
\mathfrak s_\epsilon^{-1}(0)_{set}=\bigcup_{p\in\mathcal P}s_{p,\epsilon}^{'-1}(0)/\sim,
\end{displaymath}
where $x\in s_{q,\epsilon}^{'-1}(0)\sim \underline\phi_{pq}(x)\in s_{p,\epsilon}^{'-1}(0)$.
%For $y\in s_{p,\epsilon}^{'-1}(0)\subset V_p/\Gamma_p$, let $val_p(y)$ be the number of branches that vanish on a fixed lift of $y$ to $V_p$.

For each $p\in\mathcal P$, choose $n_p$ and a lift
\begin{displaymath}
\tilde s_{p,\epsilon}'=(\tilde s'_{p,\epsilon,1},\ldots,\tilde s'_{p,\epsilon,n_p}):V_p\to E_p^{n_p}
\end{displaymath}
that represents $s_{p,\epsilon}'$.
Let $\pi:V_p\to V_p/\Gamma_p$ be the quotient map, and for $y\in \tilde s_{p,\epsilon}^{'-1}(0)$ let
\begin{displaymath}
val_p(\pi(y))=\#\set{i}{\tilde s_{p,\epsilon,i}'(y)=0}.
\end{displaymath}
By equivariance, $val_p$ does not depend on the choice of lift $y$ of $\pi(y)$.
Moreover, it is independent of the choice of representative of the multisection.

\begin{lemma}[\cite{fooo} Lemma A1.26]
For a generic choice of $\mathfrak s_{\epsilon}$, there exists a triangulation of $\mathfrak s_{\epsilon}^{-1}(0)_{set}$ such that for each simplex $\Delta_a^m$ in the triangulation, there exists a $p_a$ such that $\Delta_a^m\subset s_{p_a,\epsilon}^{'-1}(0)$ and $val_{p_a}$ is constant on the interior of $\Delta_a^m$.
\end{lemma}

Let $d$ be the virtual dimension of $\mathcal M$.
Then a weight $mul_{\Delta_a^d}$ can be assigned to each $d$-dimensional simplex in the triangulation as follows:
Fix a $y\in V_{p_a}$ such that $\pi(y)$ is in the interior of $\Delta_a^d$, where $\pi:V_{p_a}\to V_{p_a}/\Gamma_a$ is the quotient map.
For each $i$ with $\tilde s_{p_a,\epsilon,i}'(y)=0$, define $\epsilon_i$ to be $1$ or $-1$ as follows:
Since $\mathfrak s_\epsilon$ is transverse to $0$, it follows that the map
\begin{displaymath}\label{eq:appendix-multisection_orientation_1}
d_y\tilde s_{p_a,\epsilon,i}':T_yV_{p_a}\to E_{p_a}
\end{displaymath}
is surjective.
Using the trivialization of $\Det(E_{p_a}^*)\otimes \Det(T_yV_{p_a})$  given by the orientation on $\mathcal M$, this map assigns an orientation to 
\begin{displaymath}\label{eq:appendix-multisection_orientation_2}
d_ys_{p_a,\epsilon,i}^{-1}(0).
\end{displaymath}
Let $\epsilon_i=+1$ if this orientation agrees with the orientation of $T_y\Delta_a^d$ coming from the triangulation, and let $\epsilon_i=-1$ otherwise.
Then let
\begin{displaymath}
mul_{\Delta_a^d}=\frac{\sum {\epsilon_i}}{n_{p_a}}\in\qq.
\end{displaymath}
The sum is over all $i$ such that $\tilde s'_{p_a,\epsilon,i}(y)=0$.
Notice that the lift $y$ of $\pi(y)$ is fixed in this formulation, and the formula for the multiplicity agrees with that give in \cite{fo} on page 948.
If $y$ is allowed to vary over all lifts of $\pi(y)$, and the sum is adjusted accordingly, the multiplicity formula becomes
\begin{displaymath}
mul_{\Delta_a^d}=\frac{\sum {\epsilon_i}}{n_{p_a}\cdot\# \Gamma_{p_a}}\in\qq,
\end{displaymath}
as given in \cite{fooo}, Volume 2 on page 442.

\begin{definition}
The virtual fundamental chain of $\mathcal M$ is the singular chain
\begin{displaymath}
f_*[\mathcal M]=\sum_a \textrm{mul}_{\Delta_a}f_{p_a,*}[\Delta_a].
\end{displaymath}
\end{definition}

%
% The choice of bibliography style is a major decision, jointly made
% by you, your thesis advisor and the thesis editor. Common choices are
% "siam", "acm", "amsplain", "plain", "chicago".
%
\bibliographystyle{plain}
\bibliography{first-bib}

\noindent{Department of Mathematics, Kansas State University\\
galston@math.ksu.edu
}

\end{document}